\documentclass[11pt,a4paper]{article}

\usepackage[USenglish]{babel}
\usepackage{booktabs}
\usepackage[T1]{fontenc}
\usepackage[utf8]{inputenc}
\usepackage{lmodern}
\usepackage{fullpage}
\usepackage{microtype}
\usepackage{csquotes}
\usepackage{authblk}
\usepackage{xcolor}
\usepackage{graphicx}
\usepackage{hyperref}
\usepackage{float}
\usepackage{array}
\usepackage{subfigure}
\usepackage[mathscr]{euscript}
\usepackage{enumitem}
\usepackage{amsmath}
\usepackage{mathtools}
\usepackage{amsthm}
\usepackage{thmtools}
\usepackage{thm-restate}
\usepackage{dsfont}
\usepackage{amssymb}
\usepackage[noabbrev]{cleveref}
\usepackage{crossreftools}
\usepackage{framed}
\usepackage{soul}
\usepackage{comment}

\declaretheorem[style=plain,parent=section,title=Theorem,refname={Theorem,Theorems}]{theo}
\declaretheorem[style=plain,sibling=theo,title=Proposition,refname={Proposition,Propositions}]{prop}

\declaretheorem[style=plain,sibling=theo,title=Lemma,refname={Lemma,Lemmas}]{lem}
\declaretheorem[style=definition,sibling=theo,title=Definition,refname={Definition,Definitions}]{defin}
\declaretheorem[style=plain,parent=section,title=Remark,refname={Remark,Remarks}]{rem}
\newtheorem{example}{\textbf{Example}}[section]

\newlist{myenum}{enumerate}{1}
\setlist[myenum]{label={\bf\roman*)},ref={\bf(\roman*)}}
\newlist{mylist}{itemize}{1}
\setlist[mylist]{label=\textbullet}
\newlist{Hassum}{enumerate}{1}
\setlist[Hassum]{label={\bf(H\arabic*)}}
\crefname{Hassumi}{condition}{conditions}

\crefrangelabelformat{Hassumi}{#3#1--#2#4}
\newcommand{\onlyref}[1]{%
    \begingroup
    \def\crefrangeconjunction{-}%
    \labelcref{#1}%
    \endgroup
}

\newcommand{\email}[2][]{\textsuperscript{#1}\href{mailto:#2}{\texttt{#2}}}

\newcommand{\RR}{\mathbb{R}}
\newcommand{\NN}{\mathbb{N}}

\newcommand{\ZZ}{\mathbb{Z}}

\newcommand{\Ix}{\mathcal{I}_{\Delta x}}

\newcommand{\It}{\mathcal{I}_{\Delta t}}

\newcommand{\CC}{C}
\newcommand{\BB}{\mathcal{B}}
\newcommand{\GG}{\mathcal{G}}
\newcommand{\DD}{\mathcal{D}}

\newcommand{\eps}{\epsilon}

\newcommand{\diver}{{\rm{div}}}
\newcommand{\ov}[1]{\overline{#1}}

\def\supp{\mathop{\rm supp}}
\def\P{\mathcal{P}}

\def\A{\mathcal{A}}

\def\dd{{\rm d}}
\def\co{{\rm co}}

\title{
First-Order Mean Field Games with Time-Measurable and $ C ^{1,1}$ Hamiltonians: Analysis and Approximation
}
\author[1]{Elisabetta Carlini}
\author[2]{Valentina Coscetti}
\author[3]{Francisco J. Silva}
\affil[1]{Sapienza, University of Rome, Italy. \textit{Email address:} \email{carlini@mat.uniroma1.it}}
\affil[2]{Sapienza, University of Rome, Italy. \textit{Email address:} \email{valentina.coscetti@uniroma1.it}}
\affil[3]{Institut de recherche XLIM-DMI, UMR-CNRS 7252, Faculté des Sciences et Techniques, Université de Limoges, 87060
Limoges, France. \textit{Email address:} \email{francisco.silva@unilim.fr}}
\date{}

\begin{document}
\maketitle

\begin{abstract}
We introduce and analyze a numerical framework for a class of first-order, 
non-local Mean Field Game (MFG) systems with two sources of low regularity: the 
time-dependent data are merely measurable, and the Hamiltonian is only of class 
$\CC^{1,1}$ with respect to the momentum variable. The proposed method combines 
a semi-Lagrangian approximation of the Hamilton--Jacobi--Bellman equation with a 
Lagrange--Galerkin discretization of the continuity equation. Temporal 
measurability is handled by averaging the data over the time-grid intervals and 
using the stability theory of $L^1$-viscosity solutions. The limited regularity 
of the Hamiltonian is addressed through arguments from nonsmooth analysis, 
based in particular on Clarke's generalized differential calculus. We prove 
that the discrete solutions converge, up to some subsequence, to a solution of the 
MFG system. As a consequence, the numerical approximation also provides a 
constructive existence result. Under a monotonicity assumption on the coupling terms, we show that the solution to the MFG system is unique, implying the convergence of the whole sequence of discrete solutions. Finally, numerical experiments in one and two spatial dimensions illustrate the robustness of the method, including a  linear--quadratic example with a running cost that is discontinuous in time and for which we compute an exact solution.
\end{abstract}

  \medskip
 \noindent{\textbf{AMS-Subject Classification: 91A16, 49N80, 35Q89, 65M12, 65M25.}}  

 \paragraph{Acknowledgments.}
 The first two  authors were partially supported  by Italian Ministry of Instruction, University and Research
(MIUR) (PRIN Project 2022 2022238YY5, ``Optimal control problems: analysis, approximation''). The first and third  authors were  partially supported  by KAUST through the subaward agreement ORA-2021-CRG10-4674.6.
 The second author was partially supported by 
 INdAM - GNCS Project codice CUP$\_$E53C25002010001, by Avvio alla Ricerca codice CUP$\_$B83C25004300005, 
and by European Union - Next Generation EU, Missione 4, Componente 1, CUP$\_$  B53C23002010006.
 The third author was partially supported by l'Agence
Nationale de la Recherche (ANR), project ANR-22-CE40-0010, and by Minist\`ere de l'Europe et des Affaires \'etrang\`eres (MEAE), project MATH AmSud 23-MATH-17.

\section{Introduction}

The mathematical framework of Mean Field Games (MFGs), introduced independently 
by Lasry and Lions~\cite{LasryLions07} and by Huang, Malhamé, and Caines~\cite{HMC06}, provides a powerful tool 
for characterizing Nash equilibria in symmetric differential games involving a 
continuum of rational agents. The evolution of such systems is typically governed 
by a coupled system consisting of a backward Hamilton--Jacobi--Bellman (HJB) 
equation, which determines the value function and the optimal strategies, and a 
forward Fokker--Planck equation, which describes the evolution of 
the probability distribution of the population.

The numerical approximation of MFG systems has developed into a broad and 
active research area; see \cite{AchdouLauriere20,Lauriere21} for general 
overviews. For second-order MFG systems with nondegenerate diffusion, 
finite-difference schemes have been extensively analyzed; see, for instance, 
\cite{AchdouCapuzzo10,AchdouCamilliCapuzzo13,AchdouPorretta16,
AlmullaFerreiraGomes17,BonnansPfeiffer23}. Semi-Lagrangian and 
Lagrange--Galerkin methods provide an alternative class of discretizations 
\cite{CarliniSilva2015,CarliniSilva18,CalzolaCarliniSilva2024}, with extensions 
to equations involving non-local operators and fractional diffusion 
\cite{Chowdhury_Jakobsen23}. To address high-dimensional problems, machine-learning approaches, including 
deep-learning and reinforcement-learning methods, have also been developed 
\cite{RuthottoOsherLiNurbekyanFung20,Carmona21,Carmona22,Lauriere22}, along with 
low-rank methods based on Tensor Train representations 
\cite{CarliniSaluzzi26}. 
Finite element methods constitute another important direction for 
second-order MFG systems, especially in connection with Hamiltonians of limited 
regularity. Since this issue is directly related to one of the main analytical 
difficulties considered in the present work, we discuss the relevant 
contributions below.

For first-order MFG systems, the underlying agent dynamics are deterministic and 
the forward equation reduces to a continuity equation. Several approximation 
strategies have been proposed, including semi-Lagrangian schemes 
\cite{camsilva12,CarliniSilva2014,AshrafyanGomes,CarliniCoscetti26}, 
Lagrange--Galerkin discretizations \cite{CarliniSilvaZorkot}, and approximations 
through discrete-time finite-state MFGs 
\cite{HadikhanlooSilva,GianattiSilva2024,GianattiSilvaZorkot2024}.
Related finite-state formulations and monotonicity-based numerical
methods were developed in \cite{GomesMohrSouza10,GomesSaude21}. Fourier-based 
methods and related optimization techniques have also been developed~\cite{NurbekyanSaude18,Liu21}. More recently, 
neural-network approaches designed for high-dimensional deterministic MFGs have 
combined learned approximations of the value function and of the characteristic 
flow with fictitious-play iterations 
\cite{AssouliGianattiMissaouiSilva25}.
A recent particle approximation, based on an implicit time
discretization and a spatial
sampling of the characteristic Hamiltonian system, has also been analyzed under
displacement monotonicity in \cite{MeszarosOsborne26}.

In this work, we consider the following first-order MFG system with
non-local couplings, \textit{measurable time-dependent data}, and a Hamiltonian
of class $\CC^{1,1}$ with respect to its third variable:
\begin{equation}
\label{eq:MFG}
\tag{MFG}
\left\{
\begin{aligned}
-\partial_t v+H(t,x,D_xv)
    &=f(t,x,m(t))
    \quad\textrm{in }\RR^d\times(0,T),\\
v(x,T)
    &=g(x,m(T))
    \quad\textrm{in }\RR^d,\\
\partial_t m-\diver\Bigl(D_pH(t,x,D_xv)m\Bigr)
    &=0
    \quad\textrm{in }\RR^d\times(0,T),\\
m(0)&=m_0^*
    \quad\textrm{in }\RR^d.
\end{aligned}
\right.
\end{equation}
The Hamiltonian $H\colon[0,T]\times\RR^d\times\RR^d\to\RR$ is defined by
$$
H(t,x,p)
\coloneqq
\sup_{a\in A}
\bigl\{a\cdot p-\ell(t,x,a)\bigr\}
\quad
\textrm{for a.e. }t\in[0,T]
\textrm{ and for all }(x,p)\in\RR^d\times\RR^d,
$$
where $T>0$, $A\subset\RR^d$ is a nonempty convex compact set, $\ell\colon[0,T]\times\RR^d\times A\to\RR$, $f\colon[0,T]\times\RR^d\times\P_1(\RR^d)\to\RR$, $g\colon\RR^d\times\P_1(\RR^d)\to\RR$, and $m_0^*\in\P_1(\RR^d)$. Here, $\P_1(\RR^d)$ denotes the space of probability measures on $\RR^d$ with 
finite first moment. We assume, in particular, that
$\ell(t,x,\cdot)$ is uniformly strongly convex. Under the regularity assumptions
specified below, this property, together with the convexity and compactness of
$A$, implies that $H(t,x,\cdot)$ is of class $\CC^{1,1}$. This differs from the
smoother framework used in previous analyses of~\eqref{eq:MFG}, such as
\cite{CarliniSilvaZorkot}, where $A=\RR^d$ and the assumptions on the running
cost yield a Hamiltonian of class $\CC^2$ with respect to $p$. The time-dependent data are allowed to be merely measurable 
with respect to the time variable.

The convergence analysis must address two distinct and equally
significant sources of low regularity, which affect different components of the
system and call for complementary tools. Temporal measurability concerns the
consistency of the semi-Lagrangian approximation of the HJB equation, whereas
the $\CC^{1,1}$ regularity of the Hamiltonian with respect to $p$ requires a
nonsmooth treatment of the characteristic flow in the continuity equation. We
describe these two analytical ingredients separately below before combining
them to prove convergence of the full numerical approximation and, as a
consequence, existence of solutions to~\eqref{eq:MFG}.

The first source of difficulty is the low temporal regularity of the data. 
Allowing Hamiltonians and couplings that are merely measurable in time is 
important for models involving abrupt regime changes, switching dynamics, or 
rough exogenous signals. It also avoids imposing an artificial temporal 
regularization on the underlying control problem. Hamilton--Jacobi equations 
with time-measurable or discontinuous Hamiltonians have been studied through 
several generalized notions of viscosity solutions; see, among others, 
\cite{Ishii85,BarronJensen87,LionsPerthame87,BrianiRampazzo05,Barles2006}. In 
this work, we rely on the notion of $L^1$-viscosity solution introduced in 
\cite[Definition 7.1]{Ishii85}, which provides a natural framework for studying 
stability under weak temporal regularity.

Our analysis is also motivated by the recent work 
\cite{CamilliFestaMarzufero} on approximation schemes for linear viscous 
transport and conservative equations with measurable coefficients. By 
combining $L^1$-viscosity solutions with suitable notions of duality solution, 
the authors proved that convergence can be retained despite the absence of 
temporal continuity. Extending these ideas to MFG systems is substantially more 
involved because of the nonlinear coupling between the HJB and continuity 
equations.

We introduce a new semi-Lagrangian scheme adapted to HJB equations with 
measurable-in-time data. Since the values of measurable coefficients sampled at 
the time-grid nodes need not converge as the time step vanishes, the usual 
pointwise-in-time consistency argument is no longer available. We overcome this 
difficulty by averaging the data over each time interval and by establishing 
consistency in an $L^1$-in-time sense. This formulation is naturally adapted to 
the stability theory of $L^1$-viscosity solutions and allows us to prove the 
convergence of the semi-Lagrangian approximation under minimal temporal 
regularity assumptions.

The second, equally central, difficulty concerns the consequences of
the $\CC^{1,1}$ regularity of the Hamiltonian with respect to the momentum
variable. Bounded control constraints arise naturally in models involving
bounded velocities or bounded control efforts and may lead to Hamiltonians that
are convex and globally Lipschitz with respect to $p$, without being twice
continuously differentiable. In the second-order setting, MFG systems with possibly 
nondifferentiable Hamiltonians have been analyzed in 
\cite{OsborneSmears24,OsborneSmears25}, where the Fokker--Planck equation is 
formulated as a partial differential inclusion involving the convex 
subdifferential of the Hamiltonian. The approximation of such systems by means 
of $\CC^{1,1}$ regularizations is studied in 
\cite{OsborneSmearsRegularization25}. Moreover, the stability framework 
introduced in \cite{MR3735559} has led to quantitative and quasi-optimal error 
estimates for finite element approximations of stable equilibria, both in 
stationary and time-dependent settings and under limited smoothness assumptions; 
see \cite{BerryLeySilva25,BerryLeySilva25Nonsmooth,Berry25Semidiscrete}. 
Particular MFG models with bounded velocities and nonsmooth Hamiltonians have 
also been considered in 
\cite{DucasseMazantiSantambrogio22,SadeghiArjmandMazanti22}.

In our first-order setting, the characteristic velocity
$D_pH(t,x,D_xv)$ is single-valued, but the second derivative of $H$ need not
exist in the classical sense. This intermediate regularity is sufficient to
formulate the continuity equation in the usual form, but the smooth
second-order expansions underlying the analysis of the Lagrange--Galerkin
scheme in \cite{CarliniSilvaZorkot} are no longer available. We therefore
develop nonsmooth arguments specifically adapted to the discrete characteristic
flow. These arguments rely on Clarke's generalized differential calculus
\cite{Clarke} and allow us to establish the consistency, stability, and
compactness properties required for the convergence analysis of the
Lagrange--Galerkin approximation.

By coupling the semi-Lagrangian approximation of the HJB equation with the 
nonsmooth Lagrange--Galerkin approximation of the continuity equation, we obtain 
a numerical scheme for the full system \eqref{eq:MFG}. We establish estimates 
that are uniform with respect to the discretization parameters and prove the 
compactness of the corresponding family of discrete solutions. We then show 
that, as the discretization parameters vanish, every sequence of discrete 
solutions admits a subsequence converging to a solution of \eqref{eq:MFG}. In 
particular, the limit of the numerical approximations provides an existence 
result for~\eqref{eq:MFG}. Thus, the proposed 
scheme is not only a convergent numerical method: it also yields a constructive 
existence theory under the simultaneous presence of merely measurable time 
dependence and low regularity of the Hamiltonian. If the coupling terms satisfy a monotonicity assumption, we prove that the solution to~\eqref{eq:MFG} is unique and, hence, the whole sequence of discrete solutions converges to this solution.

Finally, we present numerical experiments in one and two spatial
dimensions illustrating the performance of the proposed scheme, including a
linear--quadratic example with a running cost that is discontinuous in time and
for which we compute an exact solution. This example provides a quantitative
benchmark for the method. The numerical results show excellent agreement with
the exact solution and further illustrate the robustness of the scheme in the
presence of discontinuous data.

The paper is organized as follows. Section~2 introduces the notation and the main assumptions on the data. Section~3 addresses the well-posedness of \eqref{eq:MFG} in the framework of time-measurable data. Section~4, which is the core of the paper, develops a numerical approximation of \eqref{eq:MFG} by combining a semi-Lagrangian scheme for the HJB equation with a Lagrange--Galerkin discretization of the continuity equation. Finally, Section~5 presents the numerical experiments.
\section{Preliminaries and Assumptions}\label{sec: pre}

This section is devoted to establishing the mathematical foundations for our analysis. We first fix the notation and subsequently, we state the structural assumptions on the data.

 Let $d\in\NN$. For every  $x,y \in \RR^d$, we denote by $ x\cdot y$ their standard inner product and by $|\cdot|$ its induced norm. We write $B(0,R)$ and $\ov{B}(0,R)$  for the associated open and closed balls, respectively, centered at the origin and radius $R>0$. We denote by  $|\cdot|_{\infty}$ the maximum norm in $\RR^d$, and we write $B_{\infty}(0,R)$ and $\ov{B}_{\infty}(0,R)$ for the corresponding open and closed balls, respectively, centered at the origin and radius $R>0$. We endow $\P_{1}(\RR^{d})$ with the Kantorovich–Rubinstein distance 
\begin{equation*} \label{e:d_1_alternative} 
d_1(\mu,\tilde \mu)=\sup\bigg\{\int_{\RR^d}\varphi(x)\dd[\mu-\tilde \mu](x)\,\Big|\,\varphi\colon \RR^d \to \RR  \text{ is 1- Lipschitz } \bigg\}.
\end{equation*} 
For every  $\mu\in\P_1(\RR^{d})$ we denote by $\supp(\mu)$ its support. In what follows, given an absolutely continuous measure (w.r.t. the Lebesgue measure in $\RR^{d}$)  $\mu\in \P_1(\RR^d)$,  we still denote by $\mu$ its density. For every set $X$ and every normed vector space $Y$, we denote by
$\mathcal{B}(X;Y)$ the space of bounded functions from $X$ to $Y$.
If $X$ is also a normed vector space and $k$ is a nonnegative integer, we
denote by $\CC_b(X;Y)$, $\CC^k(X;Y)$, and $\CC_c^k(X;Y)$ the spaces of
bounded continuous functions, functions of class $\CC^k$, and compactly
supported functions of class $\CC^k$ from $X$ to $Y$, respectively.
When $Y=\RR$, we omit the codomain and simply write
$\mathcal{B}(X)$, $\CC_b(X)$, $\CC^k(X)$, and $\CC_c^k(X)$. We also introduce the notation $\|M\| \coloneqq  \sup_{\substack{x \in \mathbb{R}^d}} \{|Mx|\,\,|\,\,|x|=1\}$
for the matrix $2$-norm of any matrix $M \in \mathbb{R}^{d \times d}$.

We consider the following assumptions: 
\begin{Hassum}
\item \label{hyp:terminal cost} 
The coupling cost $f$ is measurable and $ f(t,x,\cdot)$, for a.e. $t \in [0,T]$ and every $x\in \RR^d$,  is continuous in $\P_{1}(\RR^{d})$. The terminal cost $g(x,\cdot)$  is continuous in $\P_{1}(\RR^{d})$. Moreover, there exist some constants $C_{f,i},C_{g,i}>0$ ($i=1,2,3$) such that for a.e. $t\in [0,T]$ and for every  $x,y\in \RR^d $ and $\mu\in\P_1(\RR^d)$,
       \begin{align}
           |g(x,\mu)|&\leq C_{g,1},\label{eq: bound g}\\ 
           |g(x,\mu)-g(y,\mu)|&\leq C_{g,2}|x-y|,\label{eq: lip g}\\ 
           g(x+y,\mu)-2g(x,\mu)+g(x-y,\mu)&\leq C_{g,3}|y|^2,\label{eq: sconc g}\\
            |f(t,x,\mu)|&\leq C_{f,1},\label{eq: bound f}\\ 
           |f(t,x,\mu)-f(t,y,\mu)|&\leq C_{f,2}|x-y|,\label{eq: lip f}\\ 
           f(t,x+y,\mu)-2f(t,x,\mu)+f(t,x-y,\mu)&\leq C_{f,3}|y|^2. \label{eq: sconc f}
       \end{align}
              
\item  \label{hyp: compact set} $A\subset \RR^d$ is a nonempty convex and compact set.
\medskip\\
In what follows, we denote by $C_{\rm b}>0$ a constant such that $A\subseteq \ov B(0, C_{\rm b})$. 
 \item \label{hyp: lagrangian new}
 The running cost $\ell$ is measurable and such that $\ell(t,\cdot,\cdot)\in \CC^2(\RR^d\times U)$, for some open neighborhood $U$ of $A$, for a.e. $t\in [0,T]$. Moreover there exist some constants   $\ov C_{\ell,i}>0$ ($i=1,\dots,5$) such that  for a.e. $t\in [0,T]$, and for every  $(x,a) \in  \RR^d\times A$
\begin{align}
|\ell(t,x,a)|&\leq \ov  C_{\ell,1}, \label{eq: bounded L}\\ 
|D_{x}\ell(t,x,a)|&\leq \ov C_{\ell,2},\label{eq: lip L}
\\
D ^{2}_{aa}\ell(t,x,a)b\cdot b&\geq \ov C_{\ell,3}|b|^{2} \quad \text{for all } b \in \RR^d,
\label{eq: convex L}\\
 D ^{2}_{xx}\ell(t,x,a)y\cdot y &\leq \ov C_{\ell,4}|y|^2\quad  \text{for all } y \in   \RR^d,\label{eq: semiconc L}\\ 
| D_a\ell(t,x_1,a)- D_a\ell(t,x_2,a)|&\leq  \ov C_{\ell,5}|x_1-x_2|\quad \text{for all } x_1,\,x_2\in\RR^{d}. \label{eq: DaL lip}
\end{align}

\item \label{hyp: initial measure} The initial condition $m^*_{0}\in \P_{1}(\RR^d)$ is absolutely continuous with respect to the Lebesgue measure, with density still denoted by $m^*_{0}$, and there exist $C_0^*>0$ and $\mathsf{p}\in(1,+\infty]$ such that $\supp (m^*_{0})\subseteq \ov B(0,C_0^*)$ and $||m^*_{0}||_{L^{\mathsf{p}}(\RR^d)}\leq C_0^*$.  

\item \label{hyp: monoticityassumption} The following Lasry--Lions strict-monotonicity condition for $f$ holds: for a.e. $t\in [0,T]$        
\begin{equation*}
\int_{\RR^{d}}\Big(f(t,x,\mu_1)-f(t,x,\mu_2)\Big)\dd[\mu_1-\mu_2](x)> 0 \quad \text{for all }  \mu_1,\mu_2\in \P_{1}(\RR^d), \, \mu_1\neq\mu_2,
\end{equation*}
 and the following monotonicity condition for $g$ holds         \begin{equation*}
  \int_{\RR^{d}}\Big(g(x,\mu_1)-g(x,\mu_2)\Big)\dd[\mu_1-\mu_2](x)\geq 0 \quad \text{for all }  \mu_1,\, \mu_2\in \P_{1}(\RR^d).
\end{equation*}
\end{Hassum} 
\begin{example}
Following{\rm~\cite[Example 4.1]{cannarsa_capuani_2018}}, let us consider a continuous strictly increasing function $h\colon[0,\infty)\to\RR$ and a smooth, even, nonnegative, bounded, and integrable kernel $\phi\colon\RR^{d}\to\RR$. For $\mu\in\P_{1}(\RR^{d})$, set $(\phi\ast\mu)(y):=\int_{\RR^{d}}\phi(y-z)\dd\mu(z)$ and define 
$$
f(t,x,\mu):=\int_{\RR^{d}}h\bigl((\phi\ast\mu)(y)\bigr)\phi(y-x)\dd y\quad\text{for all }(t,x,\mu)\in[0,T]\times\RR^{d}\times\P_{1}(\RR^{d}).
$$
Notice that this coupling functional does not depend on $t$ and, for every $\mu_{1},\mu_{2}\in\P_{1}(\RR^{d})$, Fubini's theorem yields
$$
\begin{aligned}
&\int_{\RR^{d}}\bigl(f(t,x,\mu_{1})-f(t,x,\mu_{2})\bigr)\dd(\mu_{1}-\mu_{2})(x)
\\
&\quad=\int_{\RR^{d}}\Bigl[h\bigl((\phi\ast\mu_{1})(y)\bigr)-h\bigl((\phi\ast\mu_{2})(y)\bigr)\Bigr]\Bigl[
(\phi\ast\mu_{1})(y)-(\phi\ast\mu_{2})(y)\Bigr]\dd y\geq0.
\end{aligned}
$$
Moreover, since $h$ is strictly increasing, equality holds above only if $\phi\ast\mu_{1}=\phi\ast\mu_{2}$ almost everywhere in $\RR^{d}$, and hence everywhere by continuity. Consequently, $f$ is strictly monotone whenever the map $\P_{1}(\RR^{d})\ni\mu\longmapsto\phi\ast\mu\in C(\RR^{d})$ is injective. This is the case, for instance, for the Gaussian kernel $\RR^{d}\ni z\mapsto \phi_{\eps}(z) := \exp\left(-|z|^{2}/(2\eps^{2})\right)\in\RR$, where $\eps>0$. Indeed, there exists $c_{d,\eps}>0$ such that $\widehat{\phi_{\eps}}(\xi)=c_{d,\eps}\exp(-\eps^{2}|\xi|^{2}/2)>0$ for all $\xi\in\RR^{d}$. Therefore, $\phi_{\eps}\ast\mu_{1}=\phi_{\eps}\ast\mu_{2}$ implies $\widehat{\mu_{1}}=\widehat{\mu_{2}}$, and hence $\mu_{1}=\mu_{2}$. 
\end{example}

We collect in the following results some easy consequences of~\ref{hyp: compact set} and  \ref{hyp: lagrangian new}.

\begin{prop}\label{rem: properties H}Assume \ref{hyp: compact set} and  \ref{hyp: lagrangian new}. Then the following hold:
 \begin{enumerate}[label={\rm(\roman*)},     ref={\rm(\roman*)}]
    \item \label{item: convex H} For every $(x,p)\in\mathbb{R}^d\times\mathbb{R}^d$, the map 
$t\mapsto H(t,x,p)$ admits a measurable representative on $[0,T]$.
Moreover, for a.e.\ $t\in[0,T]$ and every $x\in\mathbb{R}^d$, 
the map $p\mapsto H(t,x,p)$ is convex.
        \item \label{item: bound H} For a.e. $t \in [0,T]$ and for every  $x\in \RR^d$ 
\begin{equation*}
   |H(t,x,p)|\leq C_{\rm b}|p|+\ov C_{\ell,1}\quad\text{for all }p\in\RR^{d}.
\end{equation*}  
   \item    \label{item: Lip space H} For a.e. $t \in [0,T]$ and for every  $ p\in \RR^d$   
   \begin{equation*}
         | H(t,x_1,p)-H(t,x_2,p)|\leq \ov C_{\ell,2}|x_1-x_2| \quad \text{ for all }  x_1,\, x_2\in \RR^d.
     \end{equation*}   
      \item  \label{item: Lip grad H} For a.e. $t \in [0,T],$  and for every  $ x\in \RR^d$  
           \begin{equation*}
         | H(t,x,p_1)- H(t,x,p_2)|\leq C_{\rm b}|p_1-p_2| \quad \text{ for all } p_1,\,  p_2 \in \RR^d.
     \end{equation*}
        \item \label{item: properies DpH}The following hold: 
 \begin{enumerate}[label={\rm(\alph*)},     ref={\rm(\roman*)}]\item \label{item: diff H}  For a.e. $t\in[0,T]$ and every $x\in\RR^d$, the map $p\mapsto H(t,x,p)$ is differentiable on $\RR^d$.
\item \label{item: meas DpH}  The map $t\mapsto D_pH(t,x,p)$ admits a measurable representative on $[0,T]$.
\item \label{item: bound DpH}  For a.e. $t\in[0,T]$ and every $x,p\in\RR^d$, we have $|D_pH(t,x,p)|\leq C_{\rm b}$.
\item\label{item: Lip DpH}   There exists $c_H>0$ such that, for a.e.
$t\in[0,T]$,
\begin{equation}\label{eq:Lips_gradient}
|D_pH(t,x_1,p_1)-D_pH(t,x_2,p_2)|\leq c_H\bigl(|x_1-x_2|+|p_1-p_2|\bigr)
\end{equation}
for all $x_1,x_2,p_1,p_2\in\RR^d$.
\end{enumerate}
\end{enumerate}
\end{prop}

\begin{proof} 
Let us show \ref{item: convex H}. 
Since $A$ is compact, it is separable. Hence, there exists a countable dense subset $A_0$ of $A$. We can write
\begin{equation*}
  H(t,x,p) = \sup_{a\in A_0}\{a\cdot p - \ell(t,x,a)\},
\end{equation*}
where the restriction to $A_0$ is justified  by the continuity of $a\mapsto \ell(t,x,a)$ given in \ref{hyp: lagrangian new}.
Since, for every $a\in A$, the function 
$t\mapsto a\cdot p - \ell(t,x,a)$ is measurable by \ref{hyp: lagrangian new}, 
$t\mapsto H(t,x,p)$ is measurable as a countable supremum of 
measurable functions.
The convexity of $p\mapsto H(t,x,p)$ follows directly from the definition.

Assertions \onlyref{item: bound H,item: Lip space H,item: Lip grad H} follow easily from \ref{hyp: compact set} and  \ref{hyp: lagrangian new}.

Let us show \ref{item: properies DpH}. The existence of $D_pH(t, x,p)$ and the bound $|D_pH(t, x,p)|\leq C_{\rm b}$ for a.e. $t\in[0,T]$ and for every  $x,\,p\in\RR^{d}$ follow directly from the compactness of $A$, the strict convexity of $a\mapsto \ell(t,x,a)$, and Danskin's theorem (see, e.g.,~\cite[Theorem 4.13]{MR1756264}). Since, for every $i=1,\ldots,d$ and for a.e. $t\in[0,T]$,
\begin{equation*}
\partial_{p_i}H(t,x,p)
=
\lim_{n\to+\infty}
n\left[
H\left(t,x,p+\frac{1}{n}e_i\right)-H(t,x,p)
\right]
\end{equation*}
and $t\mapsto H(t,x,q)$ is measurable for every $(x,q)$, we deduce that $t\mapsto\partial_{p_i}H(t,x,p)$ has a measurable representative.

 Finally, let us show~\eqref{eq:Lips_gradient}.  For a.e. $t\in [0,T]$ and for every  $(x,p,a)\in \RR^d\times \RR^d \times A$,  let us define $\mathcal{H}(t,x,p,a)\coloneqq a\cdot p-\ell(t,x,a)$. Let us fix $t \in [0,T]$ such that $H(t,\cdot,\cdot)$ is differentiable with respect to $p$ and  $x_1,\, x_2,\, p_1\, , p_2\in  \RR^d$.  Let us set $a^*_1=D_pH(t,x_1,p_1)$ and $a^*_2=D_pH(t,x_2,p_2)$. If $a^*_1= a^*_2$ the assertion is true, and hence we suppose that $a^*_1\neq a^*_2$.
Since  $a^*_1$ and $a^*_2$ are the maximizers over the convex set $A$ of $\mathcal{H}(t,x_1,p_1,\cdot )$ and of $\mathcal{H}(t,x_2,p_2,\cdot)$, respectively,  we have that 
\begin{equation*}
    \begin{aligned}
      &   D_a\mathcal{H}(t,x_1,p_1,a^*_1)\cdot ( a-a^*_1) \leq 0 \qquad \text{for all }a \in A, \\ 
        & D_a\mathcal{H}(t,x_2,p_2,a^*_2) \cdot ( a-a^*_2 ) \leq 0 \qquad \text{for all }a \in A.
    \end{aligned}
\end{equation*}
By summing the first inequality evaluated at $a=a_2^*$ and the second evaluated at $a=a_1^*$, we obtain
\begin{equation*}
      \big ( D_a\mathcal{H}(t,x_1,p_1,a^*_1)-D_a\mathcal{H}(t,x_2,p_2,a^*_2)\big ) \cdot ( a^*_1-a^*_2 )\geq 0,
\end{equation*}
which, adding and subtracting $D_a\mathcal{H}(t,x_1,p_1,a_2^*)$,  yields
\begin{equation}\label{eq: est mtcH}
\begin{aligned}
       &   \big( D_a\mathcal{H}(t,x_1,p_1,a^*_1)-D_a\mathcal{H}(t,x_1,p_1,a^*_2)\big) \cdot ( a^*_1-a^*_2) 
      \\ &    \geq \big( D_a\mathcal{H}(t,x_2,p_2,a^*_2)- D_a\mathcal{H}(t,x_1,p_1,a^*_2)\big)\cdot ( a^*_1-a^*_2).
\end{aligned}
\end{equation}
Estimating the left-hand side of \eqref{eq: est mtcH} from above thanks \eqref{eq: convex L}, we get 
\begin{equation}\label{eq: above mtcH}
    \begin{aligned}
      &   \big ( D_a\mathcal{H}(t,x_1,p_1,a^*_1)-D_a\mathcal{H}(t,x_1,p_1,a^*_2) \big )\cdot ( a^*_1-a^*_2) \\ &=  -  \big (  D_a\ell(t,x_1,a^*_1) -D_a\ell(t,x_1,a^*_2) \big ) \cdot (a^*_1-a^*_2 ) \leq     - \ov C_{\ell,3} |a^*_1-a^*_2|^2,
    \end{aligned}
\end{equation}
while estimating  the right-hand side of \eqref{eq: est mtcH}  from below thanks \eqref{eq: DaL lip}, we obtain  
\begin{equation}\label{eq: below mtcH}
\begin{aligned}
      & \big ( D_a\mathcal{H}(t,x_2,p_2,a^*_2)- D_a\mathcal{H}(t,x_1,p_1,a^*_2)\big ) \cdot ( a^*_1-a^*_2 ) \\
      &= \big (  p_2-D_a\ell(t,x_2,a^*_2)- p_1+D_a\ell(t,x_1,a^*_2)\big ) \cdot( a^*_1-a^*_2)\\ 
      &=  ( p_2-p_1  )\cdot( a^*_1-a^*_2) + \big( D_a\ell(t,x_1,a^*_2)- D_a\ell(t,x_2,a^*_2)\big)\cdot ( a^*_1-a^*_2 )\\ 
      &\geq  -|p_2-p_1|\, |a^*_1-a^*_2|- | D_a\ell(t,x_1,a^*_2)- D_a\ell(t,x_2,a^*_2)|\, | a^*_1-a^*_2 |\\  
       &\geq  -|p_2-p_1|\, |a^*_1-a^*_2|- \ov C_{\ell,5}|x_1-x_2|\, | a^*_1-a^*_2 |.\\ 
    \end{aligned}
\end{equation}
Combining  \eqref{eq: est mtcH} \eqref{eq: above mtcH} \eqref{eq: below mtcH}, we have   
\begin{equation*}
    \ov C_{\ell,3} |a^*_1-a^*_2|^2\leq  |p_2-p_1|\, |a^*_1-a^*_2|+ \ov C_{\ell,5}\, |x_1-x_2| \, |a^*_1-a^*_2 |,
\end{equation*}
which  implies
\begin{equation*}
    |a^*_1-a^*_2|\leq \frac{1}{\ov C_{\ell,3}}\left(|p_2-p_1|+ \ov C_{\ell,5}|x_1-x_2|\right).
\end{equation*}
\end{proof}
\begin{rem}
Notice that for a.e. $t\in [0,T]$ and for every  $x\in\RR^{d}$ the function $p\mapsto H(t,x,p)$ is convex, but,
in general, $H$ could be  not strict convex in the third variable. For example, we consider  $d=1$, the strict convex function $\ell(t,x,a)\coloneqq (a-\frac{1}{2})^2$ and $A=[0,1]$. In this case 
    \[ 
H(t,x,p)=
\begin{cases}
 -\frac{1}{4}  & \text{if } p\in (-\infty,-1), \\ 
   \frac{p^2}{4}+ \frac{p}{2}  & \text{if } p\in [-1,1],\\ 
 p-\frac{1}{4} & \text{if } p\in (1,+\infty), 
\end{cases}
\]
which is not strict convex.
\end{rem}

\section{Well-posedness  of \eqref{eq:MFG}}
Given  $\mu \in \CC([0,T];\P_1(\RR^d))$, we consider the following HJB equation associated with \eqref{eq:MFG}: 
\begin{equation}\label{eq:HJB}\tag{HJB}
\begin{aligned}
  -\partial_t v(t,x)+H(t,x,D_xv(t,x))&= f(t,x,\mu(t)) && \text{ in } (0,T)\times \RR^d, \\
  v(T,x)&= g(x,\mu(T))  &&\text{ in } \RR^d.\\
\end{aligned}
\end{equation}
We define 
\begin{equation*}
    \begin{aligned}
        \mathcal{USC}([0,T]\times\RR^d)&\coloneqq\{v\colon[0,T]\times \RR^d\to \RR \, |\, v \text{ is upper semi-continuous}\},\\ 
        \mathcal{LSC}([0,T]\times\RR^d)&\coloneqq\{v\colon[0,T]\times \RR^d\to \RR \, |\, v \text{ is lower semi-continuous}\}.
    \end{aligned}
\end{equation*}

\begin{defin}\label{def: L1viscosity}
  A function $v\in \mathcal{USC}([0,T]\times\RR^d)$ (resp. $v\in \mathcal {LSC}([0,T]\times\RR^d)$ ) is an \textit{$L^1$-viscosity subsolution (resp. supersolution)}  to \eqref{eq:HJB} if 
  \begin{equation}\label{eq: terminal L1visc}
      v(T,x)\leq g(x,\mu(T)) \quad ( \text{resp}. \geq) \qquad \text{ for all } x \in \RR^d
  \end{equation}
  and
  for every $(t^*,x^*)\in (0,T)\times \RR^d$, $\phi\in \CC^\infty_c((0,T)\times \RR^d)$, $b\in L^1([0,T])$, and $G\in \CC([0,T]\times \RR^d\times \RR^d)$ such that 
 \begin{enumerate}[label={\rm(\roman*)},     ref={\rm(\roman*)}]
    \item  $(t,x) \mapsto v(t,x)-\phi(t,x)-\displaystyle\int_t^Tb(s)\dd s$ has a local maximum (resp. minimum) at $(t^*,x^*)$,
    \item  $G(t,x,p)\leq H(t,x,p)-f(t,x,\mu(t))+b(t)$  (resp. $\geq$) for a.e. $t$ in a neighbourhood of $t^*$ and for every  $(x,p)$ in a neighbourhood of $(x^*,D_x \phi(t^*,x^*))$,
\end{enumerate}
we have  that 
\begin{equation}\label{eq: L1visc}
    -\partial_t\phi(t^*,x^*)+G(t^*,x^*,D_x \phi(t^*,x^*))\leq 0 \qquad(\text{resp. }\geq).
\end{equation}
A function  $v\in \CC([0,T]\times\RR^d$) is an \textit{$L^1$-viscosity solution} if it is both $L^1$-viscosity subsolution and  $L^1-$ viscosity supersolution.
\end{defin}
\begin{rem}\label{rem: funzioni test L1viscosity} 
In \cite{Ishii85}, the test functions $\phi$ appearing in \cref{def: L1viscosity} are assumed to belong to $  \CC^1((0,T)\times \RR^d)$. However, by standard arguments, it suffices to consider test functions in $ \CC^\infty_{c}((0,T)\times \RR^d)$, yielding an equivalent definition.
\end{rem} 

We now recall the comparison principle, which requires to define  following space of  functions
\begin{equation*}
    \mathcal{UC}_x([0,T]\times \RR^d)\coloneq \left\{v\colon [0,T]\times \RR^d\to \RR \, \middle| \, \lim_{r \to 0}\sup_{\substack{t \in [0,T]\\  |x-y|\leq r}}|v(t,x)-v(t,y)|=0 \right\}. 
\end{equation*}
\begin{theo}\label{thm: comparison HJB}
       Assume \onlyref{hyp:terminal cost,hyp: lagrangian new,hyp: compact set}, and let $\mu \in \CC([0,T];\P_1(\RR^d))$. 
Let \begin{equation}\label{eq: sub super sol comparison}
    \begin{aligned}
        \underline{v}&\in \mathcal{USC}([0,T]\times \RR^d)\cap \mathcal{UC}_x([0,T]\times \RR^d)\cap \BB([0,T]\times \RR^d),\\ 
        \ov v&\in \mathcal{LSC}([0,T]\times \RR^d)\cap \mathcal{UC}_x([0,T]\times \RR^d)\cap \BB([0,T]\times \RR^d),
    \end{aligned}
\end{equation}
be, respectively, a viscosity $L^1-$subsolution and a $L^1-$viscosity supersolution to \eqref{eq:HJB}. Then 
\begin{equation*}
   \sup_{(t,x)\in [0,T]\times \RR^d}(\underline{v}(t,x)-\ov v(t,x))\leq \sup_{ x\in \RR^d}(\underline{v}(T,x)-\ov v(T,x))^+.
\end{equation*}
\end{theo}
\begin{proof}
    This result follows from the comparison principle in \cite[Theorem 8.1]{Ishii85}, whose assumptions are satisfied by \onlyref{hyp:terminal cost,hyp: lagrangian new,hyp: compact set} and \cref{rem: properties H}. We note that, although \cite[Theorem 8.1]{Ishii85} is stated for $\underline{v},\ov v \in \CC([0,T]\times\RR^d)\cap\mathcal{UC}_x([0,T]\times\RR^d)$, its proof remains valid under the  conditions \eqref{eq: sub super sol comparison}.
\end{proof}

 We define  $\A\coloneqq\{\alpha\colon[0,T]\to A \, | \, \alpha \text{ measurable}\}$ and, for every  $(t,x)\in [0,T]\times \RR^d$,
 \begin{equation*} 
    J^{t,x}[\mu](\alpha)\coloneqq \int_{t}^T\Big (\ell(s,X^{t,x,\alpha}(s),\alpha(s)) +f(s,X^{t,x,\alpha}(s),\mu(s)) \Big)\dd s     + g(X^{t,x,\alpha}(T),\mu(T)),
\end{equation*}
where 
$$
X^{t,x,\alpha}(s)\coloneqq x-\displaystyle\int_{t}^s\alpha(\tau )\dd \tau\quad\text{for all }s\in [t,T].
$$
We also define the value function $v[\mu]\colon[0,T] \times \RR^d \to \RR$ as 
\begin{equation*}
    v[\mu](t,x)\coloneqq \inf_{\alpha \in \A}J^{t,x}[\mu](\alpha).
\end{equation*}
The value function is well defined and the Dynamic Programming Principle (DPP) holds, that is   for every $ t_1,t_2\in [0,T]$ with $t_2\geq t_1$ and for every $x\in \RR^d $ 
           \begin{equation*}
           \begin{aligned}
                   v[\mu](t_1,x)=\inf_{\alpha \in \A}\left\{ \int_{t_1}^{t_2} \Big(\ell\big(s,X^{t_1,x,\alpha}(s),\alpha(s)\big)  +f(s,X^{t_1,x,\alpha}(s),\mu(s))\Big)\dd s+v[\mu](t_2,X^{t_1,x,\alpha}(t_2))\right\}.
                              \end{aligned}
\end{equation*}

\begin{prop}\label{prop: properties HJB}
    Assume \onlyref{hyp:terminal cost,hyp: lagrangian new,hyp: compact set}, and let $\mu \in \CC([0,T];\P_1(\RR^d))$. Then the following hold:
 \begin{enumerate}[label={\rm(\roman*)},     ref={\rm(\roman*)}]
\item \label{item: optimal control v}{\rm(Existence of an optimal control)}  For every  $(t,x)\in [0,T)\times \RR^d$, there exists $\alpha^{t,x}\in\A$ such that $v[\mu](t,x)=J^{t,x}[\mu](\alpha^{t,x})$.
    \item \label{item: bound v}  {\rm(Uniform bound)} There exists    $C_{{\rm v}}$  independent of $\mu$ such that
    \begin{equation*}
        |v[\mu](t,x)|\leq C_{{\rm v}} \quad \text{for all }(t,x)\in [0,T]\times \RR^d.
    \end{equation*}
    \item \label{item: Lip v}{\rm (Lipschitz property)} There exists $C_{{\rm Lip}}>0$  independent of $\mu$  such that
    \begin{equation*}
        |v[\mu](t_1,x_1)-v[\mu](t_2,x_2)|\leq C_{{\rm Lip}}(|t_1-t_2|+|x_1-x_2|)  \quad   \text{for all } t_1\, t_2\in [0,T], \, x_1,x_2 \in \RR^d.
    \end{equation*}
    \item \label{item: semiconc v} {\rm (Semiconcavity property)} 
   There exists  $C_{{\rm sc}}>0$  independent of $\mu$ such that 
   \begin{equation*}
      v[\mu](t,x_1+x_2)-2v[\mu](t,x_1)+v[\mu](t,x_1-x_2)\leq C_{{\rm sc}}|x_2|^2   \quad \text{for all } t\in [0,T], \, x_1,\, x_2 \in \RR^d.
   \end{equation*}
\end{enumerate}
\end{prop}
\begin{proof}
     Assertions \onlyref{item: optimal control v,item: bound v,item: semiconc v} follow easily from   \onlyref{hyp: lagrangian new,hyp:terminal cost,hyp: compact set} arguing as in \cite[Proposition 2.1]{CarliniSilvaZorkot}. Moreover, for every $x_1,x_2 \in \RR^d$, \eqref{eq: lip g}, \eqref{eq: lip f}   and \eqref{eq: lip L} imply that
     \begin{equation}\label{eq: Lip in space v}
         |v[\mu](t,x_1)-v[\mu](t,x_2)|\leq (\ov C_{\ell,2}+ C_{f,2}+C_{g,2})|x_1-x_2| \qquad \text{for all }t \in [0,T].
     \end{equation}
By  the DPP, \ref{hyp: compact set}, \eqref{eq: bound f} \eqref{eq: bounded L},  and \eqref{eq: Lip in space v}, we have for every $t_1,t_2 \in [0,T]$
\begin{equation}\label{eq: Lip in time v}
       |v[\mu](t_1,x)-v[\mu](t_2,x)|
  \leq \big (\ov C_{\ell,1}+C_{f,1}+ (\ov C_{\ell,2}+ C_{f,2}+C_{g,2})C_{\rm b} \big )|t_1-t_2|\qquad \text{for all }x \in \RR^d.
\end{equation}
Combining \eqref{eq: Lip in space v} and \eqref{eq: Lip in time v} with the triangle inequality yields assertion \ref{item: Lip v}.
\end{proof}

 \begin{prop}\label{prop: existence HJB}
       Assume \onlyref{hyp:terminal cost,hyp: lagrangian new,hyp: compact set}, and let $\mu \in \CC([0,T];\P_1(\RR^d))$. The value function $v[\mu]$ is a $L^1-$viscosity solution to \eqref{eq:HJB}. 
\end{prop}
\begin{proof} 
We  notice that, by definition, $v[\mu](T,x)=g(x,\mu(T))$ for every $x\in \RR^d$. 
We first show that $v$ is a $L^1-$viscosity subsolution. Let $(t^*,x^*)\in (0,T)\times \RR^d$, $\phi \in \CC^{\infty}_c((0,T)\times \RR^d)$, $b \in L^1([0,T])$, and  $G\in \CC([0,T]\times \RR^d\times \RR^d)$  as in \cref{def: L1viscosity} for the subsolution property. 
Let $R^*>0$ be sufficiently small so that   $ [t^*-R^*,t^*+R^*]\subset(0,T)$ and $(t^*,x^*)$ is a maximum of $(t,x) \mapsto v[\mu](t,x)-\phi(t,x)- \displaystyle\int_t^Tb(s)\dd s$ over $[t^*-R^*,t^*+R^*]\times \ov B(x^*,R^*)$.  For every $(t,x)\in [t^*-R^*,t^*+R^*]\times \ov B(x^*,R^*)$, we have
\begin{equation}\label{eq: locmax_0}
    \begin{aligned}
        0&\leq v[\mu](t^*,x^*)-\phi(t^*,x^*)- \int_{t^*}^Tb(s)\dd s- v[\mu](t,x)+\phi(t,x)+ \int_t^{T}b(s)\dd s.\\
    \end{aligned}
\end{equation}
We fix $\alpha \in \A$. By the DPP, for every $\tau \in (t^*,T]$, we have 
\begin{equation}\label{eq: DPP sub}
    v[\mu](t^*,x^*)\leq \int_{t^*}^{\tau}\Big ( \ell\big(s,X^{t^*,x^*,\alpha}(s),\alpha(s)\big)+f(s,X^{t^*,x^*,\alpha}(s) ,\mu(s))\Big)\dd s+v[\mu](\tau,X^{t^*,x^*,\alpha}(\tau)).
\end{equation}
Since $ |X^{t^*,x^*,\alpha}(s)-x^*|\leq C_{\rm b}|s-t^*|$ for every $s \in [t^*,T]$, 
 we can choose $\tau \in (t^*,T]$ such that $|\tau-t^*|\leq R^*$ and $|X^{t^*,x^*,\alpha}(s)-x^*|\leq R^*$ for every $s\in [t^*,\tau]$. 
Combining \eqref{eq: locmax_0} and \eqref{eq: DPP sub}, we obtain 
\begin{equation}\label{eq: proof sub}
    \begin{aligned}
       &\phi(t^*,x^*)-\phi(\tau,X^{t^*,x^*,\alpha}(\tau))+ \int_{t^*}^{\tau}b(s)\dd s \\ 
        &\leq v[\mu](t^*,x^*)-v[\mu](\tau,X^{t^*,x^*,\alpha}(\tau)) \leq \int_{t^*}^{\tau}\Big ( \ell\big(s,X^{t^*,x^*,\alpha}(s),\alpha(s)\big)+f(s,X^{t^*,x^*,\alpha}(s) ,\mu(s))\Big)\dd s.\\ 
    \end{aligned}
\end{equation}
We observe that, by \eqref{eq: lip L} and~\eqref{eq: lip f},
\begin{equation*}
    \begin{aligned}
       & \left|\int_{t^*}^{\tau}\ell\big(s,X^{t^*,x^*,\alpha}(s),\alpha(s)\big)\dd s -  \int_{t^*}^{\tau}\ell\big(s,x^*,a\big)\dd s\right|\leq 
       \dfrac{\ov C_{\ell,2}C_{\rm b}}{2}(\tau-t^*)^2\\
       &\left|\int_{t^*}^{\tau}f\big (s,X^{t^*,x^*,\alpha}(s), \mu(s)\big )\dd s-  \int_{t^*}^{\tau}f\big (s,x^*, \mu(s)\big)\dd s\right|\leq  \frac{C_{f,2}C_{\rm b}}{2}(\tau-t^*)^2,
    \end{aligned}
\end{equation*}
and, since $\phi \in \CC^\infty_c([0,T]\times \RR^d)$,  
\begin{equation*}
    \begin{aligned}
        &\phi(t^*,x^*)-\phi(\tau,X^{t^*,x^*,\alpha}(\tau))\\ &=-\partial_t\phi(t^*,x^*)(\tau-t^*)-D_x\phi(t^*,x^*)\cdot \left(X^{t^*,x^*,\alpha}(\tau)-x^*\right)+\mathcal{O}((\tau-t^*)^2)\\ &=-\partial_t\phi(t^*,x^*)(\tau-t^*)+\int_{t^*}^{\tau}\alpha(s)\cdot D_x\phi(t^*,x^*)\dd s+\mathcal{O}((\tau-t^*)^2),
    \end{aligned}
\end{equation*}
where the $\mathcal{O}$ term is independent of $\alpha\in\A$.
Therefore, \eqref{eq: proof sub} yields 
\begin{multline}
 -\partial_t\phi(t^*,x^*)(\tau-t^*)+\int_{t^*}^{\tau}\Big(\alpha(s)\cdot D_x\phi(t^*,x^*)-\ell(s,x^*,\alpha(s))\Big)\dd s \\
 + \int_{t^*}^{\tau}b(s)\dd s - \int_{t^*}^{\tau}f(s,x^*, \mu(s))\dd s        \leq     \mathcal{O}((\tau-t^*)^2).
\label{eq:putting_phi_and_H}
\end{multline}
On the other hand, by  \onlyref{hyp: compact set}, \ref{hyp: lagrangian new}, and~\cite[Theorem 14.60]{RockafellarWets1998}, one has

\begin{align*}
\sup_{\alpha\in \A}\int_{t^*}^{\tau}\Big(\alpha(s)\cdot D_x\phi(t^*,x^*)-\ell(s,x^*,\alpha(s))\Big)\dd s&=\int_{t^*}^{\tau}\sup_{a\in A}\Big(a\cdot D_x\phi(t^*,x^*)-\ell(s,x^*,a)\Big)\dd s\\
&=\int_{t^*}^{\tau}H(s,x^*,D_x\phi(t^*,x^*))\dd s.
\end{align*}
It then follows from~\eqref{eq:putting_phi_and_H} that
\begin{equation*}
    \begin{aligned}
      &  -\partial_t\phi(t^*,x^*)(\tau-t^*)+\int_{t^*}^{\tau} H(s,x^*,D_x\phi(t^*,x^*))\dd s+ \int_{t^*}^{\tau}b(s)\dd s - \int_{t^*}^{\tau}f(s,x^*, \mu(s))\dd s\\ & 
        \leq     \mathcal{O}((\tau -t^*)^2),
    \end{aligned}
\end{equation*}
and then
\begin{equation*}
    \begin{aligned}
        -\partial_t\phi(t^*,x^*)(\tau-t^*)+\int_{t^*}^{\tau} G(s,x^*,D_x\phi(t^*,x^*))\dd s\leq  \mathcal{O}((\tau -t^*)^2).
    \end{aligned}
\end{equation*}
Finally, since $G$ is continuous, dividing by $(\tau -t^*)$, and letting $\tau \to t^*$, yields
    \begin{equation*}
    \begin{aligned}
        -\partial_t\phi(t^*,x^*)+ G(t^*,x^*,D_x\phi(t^*,x^*))\leq 0.
    \end{aligned}
\end{equation*}
Similarly, by standard arguments, $v[\mu]$ is also a $L^1-$viscosity supersolution to \eqref{eq:HJB}.

\end{proof}

\begin{prop}\label{prop: uniqueness HJB}
Assume \onlyref{hyp:terminal cost,hyp: lagrangian new,hyp: compact set},  and  let $\mu\in \CC([0,T];\P_1(\RR^d))$. The value function $v[\mu]$ is the unique $L^1$-viscosity solution to \eqref{eq:HJB} in $\CC_b([0,T]\times\RR^d)\cap\mathcal{UC}_x([0,T]\times\RR^d)$.
\end{prop}
\begin{proof}
The result follows directly from \cref{prop: properties HJB}, \Cref{prop: existence HJB},  and \cref{thm: comparison HJB}.
\end{proof}
To establish uniqueness of solutions to \eqref{eq:MFG}, we introduce the following notion of generalized solution to \eqref{eq:HJB} and subsequently show that the value function $v[\mu]$ satisfies it.
\begin{defin} \label{def: generalized solution} A function $v\colon[0,T]\times \RR^d\to \RR$  is a \textit{generalized solution} to \eqref{eq:HJB} if it is differentiable a.e. on $(0,T)\times \RR^d$ and
\begin{equation*}
    -\partial_tv(t,x)+H(t,x,D_xv(t,x))=f(t,x,\mu(t)) \qquad \text{a.e. in } (0,T)\times \RR^d.
\end{equation*}
\end{defin}

\begin{prop}\label{prop:v is generalized solution}
  Assume \onlyref{hyp:terminal cost,hyp: lagrangian new,hyp: compact set},  and  let $\mu\in \CC([0,T];\P_1(\RR^d))$. The value function $v[\mu]$ is a generalized solution to   \eqref{eq:HJB}.
\end{prop}
\begin{proof}
     Let us define 
\begin{equation*}
    \Omega \coloneq\Omega_{v[\mu]}\cap\Omega_{f} \cap \Omega_H,
\end{equation*}
    where $\Omega_{v[\mu]}\coloneq\{(t,x)\in (0,T)\times \RR^d \,| \, v[\mu] \text{ is differentiable at } (t,x)\}$, $\Omega_{f}\coloneq \{(t,x)\in (0,T)\times \RR^d \,|\, t \text{ is a Lebesgue point for } t\mapsto f(t,x,\mu(t)) \}$, and   $\Omega_H\coloneqq\bigcap_{p\in \RR^d}\Omega_{H}(p)$ with $\Omega_{H}(p)\coloneq\{(t,x)\in (0,T)\times \RR^d \,| \, t \text{ is a Lebesgue point for } t\mapsto H(t,x,p) \}$.
It follows from \cref{prop: properties HJB}\ref{item: Lip v} that
$\Omega_{v[\mu]}$ has full measure. We now prove that
$\Omega_f$ has full measure.
For every $q\in\mathbb{Q}^d$, the map
$t\mapsto f(t,q,\mu(t))$ belongs to $L^\infty(0,T)$ by
\eqref{eq: bound f}. Hence, by the Lebesgue differentiation theorem and the countability of $\mathbb{Q}^d$, the set
\begin{equation*}
E\coloneq\bigcap_{q\in\mathbb{Q}^d}
\{t\in(0,T)\mid t\text{ is a Lebesgue point for }s\mapsto f(s,q,\mu(s))\}
\end{equation*}
has full measure.
Let $(t,x)\in E\times\RR^d$ and let $(q_n)_n\subset\mathbb{Q}^d$ be such that
$q_n\to x$. By \eqref{eq: lip f},
\begin{equation*}
\begin{aligned}
\Bigl|
\frac1{2h}\int_{t-h}^{t+h}
\bigl(f(s,x,\mu(s))-f(t,x,\mu(t))\bigr)\dd s
\Bigr|
&\le
2C_{f,2}|x-q_n| \\
&\quad+
\frac1{2h}\int_{t-h}^{t+h}
|f(s,q_n,\mu(s))-f(t,q_n,\mu(t))|\dd s.
\end{aligned}  
\end{equation*}
Passing first to the
$\limsup$  as $h\to0$ and then the limit as $n\to\infty$ shows
that $t$ is a Lebesgue point for
$t\mapsto f(t,x,\mu(t))$. Thus, $E\times\RR^d\subseteq\Omega_f$. Since $E$ has full measure,
$E\times B(0,n)$ has full measure in $(0,T)\times B(0,n)$ for every
$n\in\NN^+$. Hence, $\Omega_f$ has full measure in
$(0,T)\times B(0,n)$ for every $n\in\NN^+$. Since
$(0,T)\times\RR^d=\bigcup_{n\geq1}(0,T)\times B(0,n)$,
we conclude that $\Omega_f$ has full measure in
$(0,T)\times\RR^d$. 
Arguing in the same way
and using the properties on $H$, 
\cref{rem: properties H} \onlyref{item: bound H,item: Lip space H,item: Lip grad H}, we obtain that $\Omega_H=\bigcap_{p\in \mathbb{Q}^d}\Omega_{H}(p)$ and it also has full measure.
    Finally, since $\Omega$ is the intersection of finitely many full-measure sets, it follows that $\Omega$ has full measure. 
  Let us take $(t^*,x^*)\in \Omega$ and $\phi\in \CC^\infty_c((0,T)\times\RR^d)$ such that 
    $(t,x)\mapsto v[\mu](t,x)-\phi(t,x)$ has a local maximum in $(t^*,x^*)$. 
    Arguing similarly to the first part of the proof of \cref{prop: existence HJB}, for $\tau \in (t^*,T]$ sufficiently small we have 
\begin{equation*}
        -\partial_t\phi(t^*,x^*)(\tau-t^*)+\int_{t^*}^{\tau} H(s,x^*,D_x\phi(t^*,x^*))\dd s - \int_{t^*}^{\tau}f(s,x^*, \mu(s))\dd s
        \leq     \mathcal{O}\left((\tau -t^*)^2\right).
\end{equation*}
   Dividing by $(\tau-t^*)$ and passing to the limit as $\tau \to t^*$, we obtain that 
\begin{equation*}
   -\partial_t\phi(t^*,x^*)+H(t^*,x^*,D_x\phi(t^*,x^*)) - f(t^*,x^*, \mu(t^*)) 
        \leq   0.
\end{equation*}
 We remark that, unlike in the proof of \cref{prop: existence HJB}, introducing $b\in L^{1}([0,T])$ and $G\in  C([0,T]\times \mathbb{R}^d \times \mathbb{R}^d)$ is not required, since we are testing the solution at points of $\Omega$. Moreover, 
 $\partial_t \phi(t^*,x^*)=\partial_t v[\mu](t^*,x^*)$ and $D_x\phi(t^*,x^*)=D_x  v[\mu] (t^*,x^*)$ allow us to  conclude   that
\begin{equation}\label{eq: gen1}
-\partial_tv[\mu](t^*,x^*)+H(t^*,x^*,D_x v[\mu](t^*,x^*)) - f(t^*,x^*, \mu(t^*)) 
\leq   0.
\end{equation}
Similarly, we can show that 
\begin{equation} \label{eq: gen2}
-\partial_tv[\mu](t^*,x^*)+H(t^*,x^*,D_x v[\mu](t^*,x^*)) - f(t^*,x^*, \mu(t^*)) 
\geq   0.
\end{equation}
Combining  \eqref{eq: gen1}, \eqref{eq: gen2} and the fact that $\RR^d\setminus \Omega $ has zero measure, 
we conclude that $v[\mu]$ is a generalized solution to \eqref{eq:HJB}.
\end{proof}

We now consider the following continuity equation associated with \eqref{eq:MFG}
\begin{equation}\label{eq:CE}\tag{CE}
\begin{aligned}
 \partial_t m-\diver (mD_pH(t,x,D_x  v[\mu]))&= 0 && \text{ in } (0,T)\times \RR^d, \\
  m(0)&=m_0^*  &&\text{ in } \RR^d.\\
\end{aligned}
\end{equation}
\begin{defin}\label{def: sol ce}
   Let $\rm p \in (1,+\infty]$ be as in \ref{hyp: initial measure}.  We say that $m\in L^{\rm p}([0,T]\times \RR^d)$ is a \textit{distributional} solution to \eqref{eq:CE}  if for every $\phi \in \CC^\infty_c( \RR^d)$  and for every $t\in [0,T]$
 \begin{equation*}
        \begin{aligned}
           \int_{\RR^d}\phi(x)m(t,x)\dd x&= \int_{\RR^d}\phi(x)m_0^*(x)\dd x    \\ &- \int_0^t\int _{\RR^d} D_pH(s,x,D_xv[\mu](s,x))\cdot D \phi(x)  m(s,x)\dd x \dd s.
        \end{aligned}
    \end{equation*}
\end{defin}

\begin{defin}\label{def: sol MFG}
    We say that a couple of functions $(v,m)$ is a solution to \eqref{eq:MFG} if $m\in \CC([0,T],\P_1(\RR^d))\cap L^{\rm p}([0,T]\times \RR^d)$, $v=v[m]$ and $m$ satisfies \cref{def: sol ce} with $\mu=m$.
\end{defin}

\begin{theo}\label{thm: unicità}
    Assume \onlyref{hyp: compact set,hyp: lagrangian new,hyp:terminal cost,hyp: initial measure,hyp: monoticityassumption}. If \eqref{eq:MFG}  admits  a solution, then it is  unique. 
\end{theo}
\begin{proof}
     The uniqueness result is obtained as in  \cite[Theorem 8]{CardaliaguetPorretta} and \cite[Theorem 3.6]{cardaliaguet2010notes}. We observe that, since the Hamiltonian is not strictly convex with respect to the gradient variable,  in \ref{hyp: monoticityassumption} we require strict monotonicity of the coupling cost $f$.  The proof relies in particular on \cref{rem: properties H} (including the convexity of $H$ in $p$), \ref{hyp: monoticityassumption}, \cref{prop: properties HJB} \onlyref{item: bound v,item: Lip v}, and \cref{prop:v is generalized solution}.
\end{proof}
\begin{theo}\label{thm: exixtence}
    Assume \onlyref{hyp: compact set,hyp: lagrangian new,hyp:terminal cost,hyp: initial measure}. Then, \eqref{eq:MFG} admits at least one solution.
\end{theo}
\begin{proof}  The existence follows from the convergence of the scheme established in \cref{thm: convergence scheme} below.
\end{proof}

\section{Numerical approximation of \eqref{eq:MFG}}\label{sec:Schemes}
In this section, we introduce a Lagrangian scheme  for approximating the solution to \eqref{eq:MFG}. This constitutes the core of the paper and extends the scheme proposed in \cite{CarliniSilvaZorkot} to the present setting. 
We begin by introducing the temporal and spatial discretization together with the associated notation that  will be used throughout the remainder of the paper. 

Let $N_T \in \NN^+$ be the number of time steps and let $\Delta t=T/N_T$ be the time step. 
 We define the sets of indices  
$\It\coloneqq\{0,\dots, N_T\}$ and $\It ^*\coloneqq\It \setminus \{N_T\}$, and the time-grids $\GG_{\Delta t}\coloneqq \{t_{k}=k\Delta t \, | \, k\in \It\}$ and $\GG^*_{\Delta t}\coloneqq \{t_{k}=k\Delta t \, | \, k\in \It^*\}$.
Let $\Delta x>0$ be a space step, we set  the $d-$ dimensional space-grid
$\GG_{\Delta x}\coloneqq \{x_{i}=i\Delta x \, | \,  i\in \ZZ^d\}$. Let us consider the couple $\Delta =(\Delta t,\Delta x)$, and we define  the time-space grids $\GG_{\Delta}\coloneqq \GG_{\Delta t}\times \GG_{\Delta x}$ and    $\GG^*_{\Delta}\coloneqq \GG_{\Delta t}^*\times \GG_{\Delta x}$.   Given $h \in \BB(\GG_{\Delta x})$ and $l\in \BB(\GG_{\Delta})$, we use the notation 
\begin{equation*}
    h_i\coloneqq h(x_i) \qquad l_{k,i}\coloneqq l(t
    _k,x_i) \qquad \text{ for every } (k,i)\in \It\times \ZZ^d,
 \end{equation*} 
 and, for $k\in \It$, we define the function $l_{k}\in \BB(\GG_{\Delta x})$ such that  $ (l_{k})_{i}\coloneqq l_{k,i}$ for every $i\in \ZZ^d$.
 We define $Q_T=[0,T]\times \RR^d$. Given $\varphi \in \CC_{b}(\RR^{d})$, $\psi\in \CC_{b}(Q_T)$ and $t\in [0,T]$, we define the functions  $\hat \varphi, \hat \psi(t)  \in\BB(\GG_{\Delta x})$ and $\hat \psi \in \BB(\GG_{\Delta})$ such that 
\begin{equation*}
  \hat \varphi _i\coloneqq \varphi(x_i) \qquad \hat \psi _{k,i}\coloneqq\psi(t_k,x_i) \qquad \hat \psi(t)_i=\psi(t,x_i) \qquad  \text{ for every } (k,i)\in \It\times \ZZ^d.
\end{equation*}
We consider the $\mathbb{Q}_1$ basis of functions  $\{\beta^{\Delta x}_{i}\colon\RR^{d}\to\RR \, |\,  i\in\ZZ^d\}$ (see \cite{QuarteroniValli}) and  the interpolation operator $I_{\Delta x}\colon\BB(\GG_{\Delta x})\to \CC_{b}(\RR^{d})$ defined by  
\begin{equation*}\label{interpolation operator}
    I_{\Delta x}[W](\cdot)\coloneqq \sum_{i\in \ZZ^d}W_{i}\beta^{\Delta x}_{i}(\cdot).
\end{equation*}
We have that 
\begin{equation} \label{interperr}
    \sup_{x\in \RR^d}|I_{\Delta x}[\hat{\phi}](x)-\phi(x)|= \mathcal{O}\left((\Delta x)^{\lambda}\right),
\end{equation}
where $\lambda=1$ if $\phi$ is Lipschitz, and $\lambda=2$ if $ \phi\in \CC^2(\RR^d)$ with bounded first and second order derivatives 
(see \cite{QuarteroniValli}, Theorem~3.4.2 and Remark~3.4.2). Moreover, if $W\in \BB(\GG_{\Delta x})$ with  $|W_i-W_j|\leq L|x_i-x_j|$ for $L>0$ for every $i,j\in \ZZ^d $, then 
\begin{equation}\label{eq: Lip interp}
    I_{\Delta x}[W] \quad\text{is  Lipschitz with constant}\quad  \sqrt{d}L.
\end{equation}

\subsection{Semi-Lagrangian scheme for the Hamilton-Jacobi-Bellman equation }\label{A_fully_discrete_scheme_for_the_Hamilton_Jacobi_Bellman_equation}
We propose the following  semi-Lagrangian scheme for \eqref{eq:HJB}, which differs from standard semi-Lagrangian formulations in the
discretization in time of the running cost and the coupling term.

Given $\mu \in \CC([0,T];\P_1(\RR^d))$,  we define the  operator  ${S}_{\Delta}[\mu]\colon\BB(\GG_{\Delta x})\times \GG_{\Delta}^* \to \RR$ by
\begin{equation}\label{eq:infS}
    {S}_{\Delta}[\mu](W,t_k,x_i)\coloneqq \inf_{a\in A }\left\{\int_{t_k}^{t_{k+1}} \ell(s,x_i,a)\dd s+I_{\Delta x}[W](x_{i}-\Delta t a)\right\}+\int_{t_k}^{t_{k+1}}f(s,x_{i},\mu(t_k))\dd s.
\end{equation}
The numerical approximation is then defined recursively as following:
define $v\in \BB(\GG_{\Delta})$ as
\begin{equation}\label{eq:HJBdiscr}\tag{${\rm HJB}_{\Delta}$}
            \begin{aligned}
               &v_{k,i}={S}_{\Delta}[\mu](v_{k+1},t_k,x_i) &&\text{for all   }(t_k,x_i)\in \GG_{\Delta}^*, \\
               &v_{N_T,i}=g(x_i,\mu(T)) &&  \text{for all  }x_i\in \GG_{\Delta x}. 
            \end{aligned}
\end{equation}
We define its extension $v_{\Delta}[\mu]\colon[0,T]\times \RR^d\to \RR$ as
\begin{equation*}\label{vdelta}
   v_{\Delta}[\mu](t,x)\coloneqq I_{\Delta x}\left[v_{\left\lceil\frac{t}{\Delta t}\right\rceil}\right](x).
\end{equation*}
We define the {\it numerical Hamiltonian} ${H}_{\Delta}\colon\BB(\GG_{\Delta x})\times \GG_{\Delta} \to \RR$ as
\begin{equation}\label{eq: Discrete Hamiltonian}
    {H}_{\Delta}(W,t_k,x_i)\coloneqq \frac{1}{\Delta t}\left ( W_i- \inf_{a\in A }\left\{\int_{t_k}^{t_{k+1}} \ell(s,x_i,a)\dd s+I_{\Delta x}[W](x_{i}-\Delta t a)\right\}\right).
\end{equation}
We observe that for every $W\in \BB(\GG_{\Delta x})$
\begin{equation}\label{eq: relation H S}
    H_{\Delta}(W,t_k,x_i)= \frac{1}{\Delta t}\left ( W_i- S_{\Delta}[\mu](W,t_k,x_i)+ \int_{t_k}^{t_{k+1}}f(s,x_{i},\mu(t_k))\dd s\right) \quad \text{for all }(t_k,x_i)\in \GG_{\Delta}.
\end{equation}
The following consistency property will be used to prove the convergence of the scheme \eqref{eq:HJBdiscr} to the $L^1-$ viscosity solution to \eqref{eq:HJB}. 
\begin{prop}\label{properties infS}
Assume \onlyref{hyp: compact set,hyp:terminal cost,hyp: lagrangian new}, let $\Delta=(\Delta t,\Delta x)\in(0,\infty)^2$ and $\mu\in \CC([0,T];\P_1(\RR^d))$.
Then the following hold:
 \begin{enumerate}[label={\rm(\roman*)},     ref={\rm(\roman*)}]
\item \label{item: well pos SL}  {\rm(Well-posedness)}  The scheme \eqref{eq:HJBdiscr} admits a unique solution  and  there exists at least one $a\in A$  minimizing
the right-hand-side of \eqref{eq:infS}. 
\item \label{item: Monotone SL}   {\rm(Monotonicity)} 
For every  $u,\widetilde u\in \BB(\GG_{\Delta_x})$ such that $u\leq \widetilde u $, we have 
\begin{equation*}
{S}_{\Delta}[\mu](u,t_k,x_i)\leq {S}_{\Delta}
[\mu](\widetilde u,t_k,x_i) \quad \text{for all } (t_k,x_i)\in \GG_{\Delta}^*.
\end{equation*}
\item \label{item: invariance SL}  {\rm(Invariance with respect to constants)} 
For every  $u \in \BB(\GG_{\Delta_x})$ and $K\in \RR$, we have
\begin{equation*}
{S}_{\Delta}[\mu](u+K,t_k,x_i)={S}_{\Delta}[\mu](u,t_k,x_i)+K\quad \text{for all } (t_k,x_i)\in \GG_{\Delta}^*.
\end{equation*}
\item \label{item: error SL}  {\rm(Error of the numerical Hamiltonian)} 
Let  $\Delta= (\Delta t, \Delta x)\in(0,\infty)^2$ and $\mu\in \CC([0,T];\P_1(\RR^d))$. Then, for every    $\phi \in \CC_b^2(\RR^d)$ and for every $(t_{k},x_{i})\in \GG_{\Delta}^*$, we have
\begin{equation*}
\begin{aligned}
& \left|H_{\Delta}(\hat \phi,t_{k},x_{i})-\sup_{a\in A}\left\{a \cdot D\phi(x_{i})-\frac{1}{\Delta t}\int_{t_{k}}^{t_{k+1}}\ell(s,x_{i},a)\dd s \right\}\right| = \mathcal{O}\left(\frac{(\Delta x)^2}{\Delta t}+\Delta t\right).
 \end{aligned}
\end{equation*}
\item \label{item: consistency SL}  {\rm(Consistency)}
Let  $(\Delta _n)_{n \in \NN}=((\Delta t_n, \Delta x_n))_{n\in \NN}\subset(0,\infty)^2$ with $N^n_T\Delta t_n=T$, 
$(x_{i_{n}})_{n\in \NN}\subset \GG_{\Delta x_n}$, and $x^*\in \RR^d$ be such that, as $n\to \infty$,
$\Delta_n\to 0$,   $(\Delta x_n)^2/\Delta t_n\to 0$,
  and $x_{i_{n}}\to x^*$. Then, for every    $\phi \in \CC_b^2(\RR^d)$, we have
\begin{equation*}
\begin{aligned} 
\lim_{n\to \infty}\sup_{t \in [0,T)} &\left | \sum_{k_n=\left \lfloor \frac{t}{\Delta t_n}\right\rfloor}^{N^n_T-1}\Delta t_nH_{\Delta_n}(\hat \phi,t_{k_n},x_{i_n}) -\int_{t}^TH(s,x^*,D\phi(x^*))\dd s \right |=0.  
\end{aligned}
\end{equation*} 
\end{enumerate}
\end{prop}

\begin{proof}
Since the scheme is explicit, $A$ is compact, and $a\mapsto\int_{t_k}^{t_{k+1}} \ell(s,x_i,a)\dd s+ I_{\Delta x}[W](x_{i}-\Delta t a)$  is continuous, assertion \ref{item: well pos SL} follows.   The proofs of  \onlyref{item: Monotone SL,item: invariance SL} are straightforward. 
   The one-step  error \ref{item: error SL} is obtained as usual from \eqref{eq: Discrete Hamiltonian}, a Taylor second order expansion of  $\phi$, and \eqref{eq: Lip interp}.
Assertion \ref{item: consistency SL} is a consequence of \ref{item: error SL}. Indeed, we define $h_n(s)\coloneqq t_{\left\lfloor\frac{s}{\Delta t_n}\right\rfloor}$ and $\mathcal{J}_n(s)\coloneqq \left\{\left\lfloor\frac{s}{\Delta t_n}\right\rfloor,\dots,N^n_T-1\right\}$ for every $s\in [0,T]$. Observing that, for every $t \in [0,T]$,
\begin{equation*}
\begin{aligned}
&\sum_{k_n\in \mathcal{J}_n(t)}\left | \Delta t_nH_{\Delta_n}(\hat \phi,t_{k_n},x_{i_n})-\sup_{a\in A} \left \{\Delta t_n a \cdot D\phi(x_{i_n})- \int_{t_{k_n}}^{t_{k_n+1}}\ell(s,x_{i_n},a)\dd s \right \}\right |\\ 
& =\mathcal{O}\left(\frac{(\Delta x_n)^2}{\Delta t_n}+\Delta t_n\right),
\end{aligned}
\end{equation*}
and, by adding and subtracting 
$$
\int_{h_n(t)}^{T}
H(s,x^*,D\phi(x^*))\dd s+
\sum_{k_n\in \mathcal{J}_n(t)}
\sup_{a\in A}
\left\{
\Delta t_n\, a\cdot D\phi(x^*)
-
\int_{t_{k_n}}^{t_{k_n+1}}
\ell(s,x^*,a)\dd s
\right\},
$$ we have
\begin{equation}\label{eq:consitencyterm}
    \begin{aligned}
        & \left|\sum_{k_n\in \mathcal{J}_n(t)} \Delta t_nH_{\Delta_n}(\hat \phi,t_{k_n},x_{i_n})-\int_{t}^T\left(H(s,x^*,D\phi(x^*)\right)\dd s \right|\\ 
                &\leq \mathcal{O}\left(\frac{(\Delta x_n)^2}{\Delta t_n}+\Delta t_n\right ) +\left |   \int_{h_n(t)}^{T}H(s,x^*,D\phi(x^*))\dd s-  \int_{t}^{T}H(s,x^*,D\phi(x^*))\dd s \right | \\ 
                &\quad +\left | \sum_{k_n\in \mathcal{J}_n(t)}\left (\sup_{a\in A} \left \{\Delta t_n a \cdot D\phi(x_{i_n})- \int_{t_{k_n}}^{t_{k_n+1}}\ell(s,x_{i_n},a)\dd s \right \}  \right. \right.\\ &\quad  \left. \left. -\sup_{a\in A} \left \{\Delta t_n a \cdot D\phi(x^*)- \int_{t_{k_n}}^{t_{k_n+1}}\ell(s,x^*,a)\dd s \right \} \right )\right |  \\
          & \quad+\left | \sum_{k_n\in \mathcal{J}_n(t)} \sup_{a\in A} \left \{\Delta t_n a \cdot D\phi(x^*)- \int_{t_{k_n}}^{t_{k_n+1}}\ell(s,x^*,a)\dd s \right \}  - \int_{h_n(t)}^{T}H(s,x^*,D\phi(x^*))\dd s \right |.  \\
    \end{aligned}
\end{equation}
We now derive estimates for all the terms appearing on the right-hand side of \eqref{eq:consitencyterm}. The second term in the right-hand side of \eqref{eq:consitencyterm}
 goes to zero, as $n\to \infty$, due to the dominated convergence theorem. 
Using $|\sup w-\sup u|\leq\sup |w-u|$ the term 
\begin{equation*}
\begin{aligned}
  &\left | \sum_{k_n\in \mathcal{J}_n(t)}\left (\sup_{a\in A} \left \{\Delta t_n a \cdot D\phi(x_{i_n})- \int_{t_{k_n}}^{t_{k_n+1}}\ell(s,x_{i_n},a)\dd s \right \}\right.\right.\\ & \left.\left. -\sup_{a\in A} \left \{\Delta t_n a \cdot D\phi(x^*)- \int_{t_{k_n}}^{t_{k_n+1}}\ell(s,x^*,a)\dd s \right \} \right )\right | \\  
 &\leq |x_{i_n}-x^*| (C_{\rm b}C_{\phi}T+\ov  C_{\ell,2}T).
\end{aligned}
\end{equation*}
 goes to zero, as $n\to \infty$, since $x_{i_n}\to x^*$.  
Let us  show that the last term in  the right-hand side of \eqref{eq:consitencyterm}
  goes to zero, as $n\to \infty$. First, we observe that 
 \begin{equation*}\label{eq: cons 1}
      \begin{aligned}
         & \sum_{k_n\in \mathcal{J}_n(t)} \left( \int_{t_{k_n}}^{t_{k_n+1}}H(s,x^*,D\phi(x^*))\dd s-\sup_{a\in A} \Big \{  \int_{t_{k_n}}^{t_{k_n+1}}(a \cdot D\phi(x^*)-\ell(s,x^*,a))\dd s \Big \}  \right )\\ 
         &\geq \sum_{k_n\in \mathcal{J}_n(t)} \left (    \int_{t_{k_n}}^{t_{k_n+1}}H(s ,x^*,D\phi(x^*))\dd s-\int_{t_{k_n}}^{t_{k_n+1}}\sup_{a\in A} \Big \{a \cdot D\phi(x^*)-\ell(s,x^*,a)\Big \}\dd s   \right) =0,
      \end{aligned}
  \end{equation*}
which implies
  \begin{equation}\label{eq: liminf consistence}
      \liminf_{n\to \infty}\left( \int_{h_n(t)}^{T}H(s,x^*,D\phi(x^*))\dd s -  \sum_{k_n\in \mathcal{J}_n(t)} \sup_{a\in A} \Big \{\Delta t_n a \cdot D\phi(x^*)- \int_{t_{k_n}}^{t_{k_n+1}}\ell(s,x^*,a)\dd s\Big \}  \right )\geq 0.
  \end{equation}
On the other hand, by {\bf(H2)}, {\bf(H3)}, and~\cite[Theorem 14.60]{RockafellarWets1998}, there exists $\alpha^* \in \A$ such that 
\begin{equation}\label{eq: consistency RockafellarWets}
\int_{h_n(t)}^{T}H(s,x^*,D\phi(x^*))\dd s=\int_{h_n(t)}^{T}\Big(\alpha^*(s)\cdot D\phi(x^*)-\ell(s,x^*,\alpha^*(s))\Big)\dd s.
\end{equation}
Since $\alpha^* \in L^2([0,T];A)$, due to \cite[Theorem 1]{Clement}, there exists a sequence of piecewise constant functions $ (\alpha_n)_{n \in \NN}\subset L^2([0,T];A)$ constant on each interval of the time-grid $[t_{k_n},t_{k_n+1})$ such that $||\alpha_n - \alpha^* ||_{L^2([0,T];A)}\to 0$ as $n \to \infty$. 
Therefore, we obtain that for every interval $[t_{k_n},t_{k_n+1})$
\begin{equation}\label{eq: cons alpha n}
    \begin{aligned}
 \int_{t_{k_n}}^{t_{k_n+1}}\Big(\alpha^*(s)\cdot D\phi(x^*)-\ell(s,x^*,\alpha^*(s))\Big)\dd s  &\leq   \int_{t_{k_n}}^{t_{k_n+1}}\alpha_n(s)\cdot D\phi(x^*)-\ell(s,x^*,\alpha_n(s))\dd s \\ 
    & \quad +    \int_{t_{k_n}}^{t_{k_n+1}}( \alpha^*(s) \cdot D\phi(x^*)- \alpha_n(s) \cdot D\phi(x^*))\dd s \\ &\quad +\int_{t_{k_n}}^{t_{k_n+1}}( \ell(s,x^*, \alpha_n(s))-\ell(s,x^*, \alpha^*(s)))\dd s \\ 
     &\leq \sup_{a \in A}  \int_{t_{k_n}}^{t_{k_n+1}}(a\cdot D\phi(x^*)-\ell(s,x^*,a))\dd s \\ 
  &\quad  +    \int_{t_{k_n}}^{t_{k_n+1}}( \alpha^*(s) \cdot D\phi(x^*)- \alpha_n(s) \cdot D\phi(x^*))\dd s \\ &\quad +\int_{t_{k_n}}^{t_{k_n+1}}( \ell(s,x^*, \alpha_n(s))-\ell(s,x^*, \alpha^*(s)))\dd s.  \\ 
    \end{aligned}
\end{equation}
Moreover using \eqref{eq: consistency RockafellarWets} together with \eqref{eq: cons alpha n}, we get 
\begin{equation}\label{eq: cons 2 density}
    \begin{aligned}
        \int_{h_n(t)}^{T}H(s,x^*,D\phi(x^*))\dd s
        &=\sum_{k_n\in \mathcal{J}_n(t)}  \int_{t_{k_n}}^{t_{k_n+1}}\Big(\alpha^*(s)\cdot D\phi(x^*)-\ell(s,x^*,\alpha^*(s))\Big)\dd s\\
         &\leq  \sum_{k_n\in \mathcal{J}_n(t)}\sup_{a \in A}  \int_{t_{k_n}}^{t_{k_n+1}}(a\cdot D\phi(x^*)-\ell(s,x^*,a))\dd s \\ 
   & \quad +   || \alpha_n-\alpha^*||_{L^2([0,T];A)}  |D\phi(x^*)|\, T\,\\ 
    &\quad +\int_{0}^{T}| \ell(s,x^*, \alpha_n(s))-\ell(s,x^*, \alpha^*(s))|\dd s, \\ 
    \end{aligned}
\end{equation}
where, since $\ell$ is continuous with respect to the third variable and \eqref{eq: bounded L} holds,  the term  
$\int_{0}^{T}| \ell(s,x^*, \alpha_n(s))-\ell(s,x^*, \alpha^*(s))|\dd s$, up to subsequences, goes to zero as $n \to \infty$. Therefore, \eqref{eq: cons 2 density} implies \begin{equation}\label{eq: limsup consistence}
      \limsup_{n\to \infty}\left( \int_{h_n(t)}^{T}H(s,x^*,D\phi(x^*))\dd s -  \sum_{k_n\in \mathcal{J}_n(t)} \sup_{a\in A} \left \{\Delta t_n a \cdot D\phi(x^*)- \int_{t_{k_n}}^{t_{k_n+1}}\ell(s,x^*,a)\dd s\right \}  \right )\leq 0.
  \end{equation}
The thesis follows by combining \eqref{eq: liminf consistence} and \eqref{eq: limsup consistence}.
\end{proof}

\begin{prop}\label{prop: properties decrete value function}
Assume \onlyref{hyp: compact set,hyp:terminal cost,hyp: lagrangian new}, let $\Delta=(\Delta t,\Delta x)\in(0,\infty)^2$ and $\mu\in \CC([0,T];\P_1(\RR^d))$.
Then the following  hold:
\begin{enumerate}[label={\rm(\roman*)},     ref={\rm(\roman*)}]
      \item \label{item: stability vDelta} {\rm(Stability)}
 There exists  $\widetilde{C}_{\text{{\rm v}}}>0$, independent of $(\mu,\Delta)$, such that
\begin{equation*}
|v_{\Delta}[\mu](t,x)|\leq \widetilde C_{\text{{\rm v}}}\quad\text{for all }(t,x)\in  [0,T]\times\RR^{d}. 
\end{equation*}
\item \label{item: Lip space vDelta} {\rm(Lipschitz property in space)}
There exists  $\widetilde{C}_{\text{{\rm Lip}}}>0$, independent of $(\mu,\Delta)$, such that
    \begin{equation*}
      \big|v_{\Delta}[\mu](t,x_1)-v_{\Delta}[\mu](t,x_2)\big|\leq \widetilde{C}_{\text{{\rm Lip}}}|x_1-x_2|\quad\text{for all } t\in[0,T], \, x_1,\,x_2\in\RR^d.
  \end{equation*}
\item \label{item: Discr Lip time vDelta} {\rm(Discrete Lipschitz property in time)}
There exists 
$\widetilde {C}_{\rm tm}>0$,  independent of $(\mu,\Delta)$, such that 
\begin{equation*}
    |v_{\Delta}[\mu](t_{k_1},x)-v_{\Delta}[\mu](t_{k_2},x)|\leq  \widetilde C_{\rm tm}|k_1-k_2|\Delta t \qquad \text{for all } k_1,k_2 \in \It, \, x\in \RR^d.
\end{equation*}
\item  \label{item: semiconv vDelta} {\rm (Discrete semiconcavity)}
There exists  $\widetilde{C}_{\text{{\rm sc}}}>0$, independent of $(\mu,\Delta)$, such that
for every $x_i\in \GG_{\Delta x}$
\begin{equation*}\label{eq: discr semiconc 1}   
v_{\Delta}[\mu](t,x+x_i)-2v_{\Delta}[\mu](t,x)+   v_{\Delta}[\mu](t,x-x_i) \leq \widetilde{C}_{\text{{\rm sc}}}|x_i|^2\quad \text{for all }(t,x)\in[0,T]\times  \RR^d.
\end{equation*} 
\end{enumerate}
\end{prop}
\begin{proof}
 Results \onlyref{item: stability vDelta,item: Lip space vDelta,item: semiconv vDelta} are proven in the standard way as in \cite[Lemma 3.1 and Lemma 3.2]{CarliniSilva2014}.

  We now prove \ref{item: Discr Lip time vDelta}.
   Let us suppose $k_1=k$ and $k_2={k+1}$, for $k\in \It^*$. Let  $x=x_i$, for some $x_i \in \GG_{\Delta x}$, and let $a^*\in A$ be a minimizer of $S_{\Delta}[\mu](v_{k+1},t_k,x_i)$.
Thanks to the Lipschitz property  in space \ref{item: Lip space vDelta}, we have 
\begin{equation*}
\begin{aligned}
  &  |  v_{k,i}-v_{k+1,i}|\\ &\leq | v_\Delta [\mu](t_{k+1},x_{i}-a^*\Delta t)-  v_{\Delta}[\mu](t_{k+1},x_i)|+\int_{t_k}^{t_{k+1}} |\ell(t,x_i,a^*)|\dd t+\int_{t_k}^{t_{k+1}}|f(t,x_{i},\mu(t_k))|\dd t \\ 
  &\leq  \widetilde C_{\rm tm}\Delta t,
\end{aligned}
\end{equation*}
where we define $ \widetilde C_{\rm tm}\coloneqq \widetilde C_{\rm Lip}C_{\rm b}+\ov C_{\ell,1}+C_{f,1}$.
It follows from the above estimate that, for every  $x\in \mathbb{R}^d$, 
\begin{equation*}
\begin{aligned}
      |  v_{\Delta}[\mu](t_{k+1},x)-  v_{\Delta}[\mu](t_k,x)|&
      \leq \sum_{i\in \ZZ^d}|v_{k+1,i}-v_{k,i}| \beta_i^{\Delta x}(x)\leq 
    \widetilde C_{\rm tm}\Delta t.
\end{aligned}
\end{equation*}
\end{proof}

The following result is obtained by adapting the arguments used in \cite[Theorem 2.2]{Monteillet2010}, concerning the convergence for a non-local eikonal equation, and  in
\cite[Theorem 4.4]{CamilliFestaMarzufero},  proving convergence to  a linear viscous transport equation. 

 \begin{prop} \label{prop: converg discr v}
Assume 
\onlyref{hyp: compact set,hyp: lagrangian new,hyp:terminal cost}.
Let  $(\Delta _n)_{n\in \NN}=((\Delta t_n,\Delta x_n))_{n\in \NN}\subset (0,\infty)^2$, $(\mu_n)_{n\in \NN}\subset \CC([0,T];\P_1(\RR^d))$, and  $\mu\in \CC([0,T];\P_1(\RR^d))$ be  such that, as $n\to \infty$, $\Delta_n\to 0$,   $(\Delta x_n)^2/\Delta t_n\to 0$,  and $\mu_n\to \mu$ in $\CC([0,T];\P_1(\RR^d))$. Then
 $(v_{\Delta_n}[\mu_n])_{n\in \NN}$ converges to $v[\mu]$ uniformly over compact subsets of $[0,T]\times \RR^d$.
\end{prop}

\begin{proof}For $(t,x)\in [0,T]\times \RR^d$ we set 
\begin{equation*}
    v^*(t,x)\coloneqq \limsup_{\substack{(t^n,x^n)\to (t,x) \\ 
    n\to \infty}}v_{\Delta _n}[\mu_n](t^n,x^n), \qquad  v_*(t,x)\coloneqq \liminf_{\substack{(t^n,x^n)\to (t,x) \\ 
    n\to \infty}}v_{\Delta _n}[\mu_n](t^n,x^n). 
\end{equation*}
Since ($v_{\Delta_n}[\mu_n])_{n \in \NN}$ is equi-bounded (see \cref{prop: properties decrete value function} \ref{item: stability vDelta}), $v^*$ and $v_*$ are well defined. 
We verify that $v^*$ and $v_*$ are a $L^1-$ viscosity subsolution and a $L^1-$viscosity supersolution, respectively, to \eqref{eq:HJB}.

Let $x \in \RR^d$  and $((t^n,x^n))_{n \in \NN}\subset [0,T]\times \RR^d$, with  $(t^n,x^n)\to (T,x)$ as $n\to \infty$. There exists a sequence ($t_{k_n})_{n\in\NN} \subset \GG_{\Delta t_n}$ such that  $v_{\Delta_n}[\mu_n](t^n,x^n)=v_{\Delta_n}[\mu_n](t_{k_n},x^n)$ and  $t_{k_n} \to T$, as $n\to \infty$, which yields
\begin{equation*}
\begin{aligned}
     |v_{\Delta_n}[\mu_n](t^n,x^n)-g(x,\mu(T))|&\leq|v_{\Delta_n}[\mu_n](t^n,x^n)-v_{\Delta_n}[\mu_n](t^n,x)|\\ &\quad +|v_{\Delta_n}[\mu_n](t^n,x)-v_{\Delta_n}[\mu_n](T,x)|\\ & \quad +|v_{\Delta_n}[\mu_n](T,x)-g(x,\mu_n(T))| \\ &\quad +|g(x,\mu_n(T))-g(x,\mu(T))|,
     \end{aligned}
\end{equation*}
where the term on the right-hand side converges to zero as $n\to\infty$, by
\cref{prop: properties decrete value function}
\onlyref{item: Lip space vDelta,item: Discr Lip time vDelta},
\eqref{interperr}, and~\ref{hyp:terminal cost}.
This implies that $v_{\Delta_n}[\mu_n](t^n,x^n)\to g(x,\mu(T))$ and, hence, $$v^*(T,x)=v_*(T,x)=g(x,\mu(T)).$$
Let $(t^*,x^*)\in (0,T)\times \RR^d$, $\phi \in \CC^{\infty}_c((0,T)\times \RR^d)$, $b \in L^1([0,T])$, and $G\in \CC([0,T]\times \RR^d\times \RR^d)$   as in \cref{def: L1viscosity}   for the subsolution property. We now establish  the subsolution inequality in \eqref{eq: L1visc}.
For every $t\in [0,T)$, we define
\begin{equation*}
\begin{aligned}
    \mathcal{E}_{\Delta_n}(t)&\coloneqq \sum_{k_n=\left\lfloor \frac{ t}{\Delta t_n }\right\rfloor}^{N^n_T-1}\left (\Delta t_n H_{\Delta_n}(\hat \phi(t^*),t_{k_n},x_{i_n})-\int_{t_{k_n}}^{t_{k_n+1}}f(s,x_{i_n},\mu_n(t_{k_n}))\dd s\right)\\ &-\int_t^T\Big (H(s,x^*,D_x \phi(t^*,x^*))-f(s,x^*,\mu(s))\Big) \dd s,
\end{aligned}
\end{equation*}
where $N_T^n\Delta t_n=T$, and $x_{i_n}\in \arg\min\{ |x-x^*| \, \big | \, x \in  \GG_{\Delta x_n}\}$.  We define $h_n(s)\coloneqq t_{\left\lfloor\frac{s}{\Delta t_n}\right\rfloor}$ for every $s \in [0,T]$.
 The term
\begin{equation*}
    \begin{aligned}
        &\left | \int_{t}^{T}f(s,x^*,\mu(s))\dd s -\sum_{k_n=\left\lfloor \frac{ t}{\Delta t_n }\right\rfloor}^{N^n_T-1}\int_{t_{k_n}}^{t_{k_n+1}}f(s,x_{i_n},\mu_n(t_{k_n})\Big )\dd s \right | \\ 
        &=\left | \int_{t}^{T}f(s,x^*,\mu(s))\dd s - \int_{h_n(t)}^{T}f(s,x_{i_n},\mu_n(h_n(s)) )\dd s \right |\\ 
        &\leq  \int_{0}^{T}\Big | f(s,x^*,\mu(s) )\chi_{[t,T]}(s)-f(s,x_{i_n},\mu_n(h_n(s)) )\chi_{\left[t_{\left \lfloor t/\Delta t_n\right\rfloor},T\right]}(s) \Big |\dd s 
    \end{aligned}
\end{equation*}
goes to zero, as $n\to \infty$, thanks to the dominated convergence  theorem, since \eqref{eq: bound f} and \eqref{eq: lip f} hold,  $x_{i_n}\to x^*$, $\sup_{s\in [0,T]}|h_n(s)-s|= \Delta t_n \to 0$ and $\sup_{t\in [0,T]}\dd _1(\mu_n(t),\mu(t))\to 0$.  Combining this with the consistency property shown in \cref{properties infS} \ref{item: consistency SL}, we have that $\mathcal{E}_{\Delta_n}\to 0$ uniformly on $[0,T)$, as $n\to \infty$.
We define $  w_{\Delta_n}\colon[0,T)\times \RR^d \to \RR$ as
\begin{equation*}
    w_{\Delta_n}(t,x)\coloneqq v_{\Delta_n}[\mu_n](t,x)-\phi(t,x)-\int_t^Tb(s)\dd s+\mathcal{E}_{\Delta_n}(t). 
\end{equation*}
By standard argument  (see, e.g., \cite[Proposition 2.8]{smears2018hamilton}), there exists  a  subsequence, still denoted by $(\Delta_n)_{n \in \NN}=((\Delta t_n,\Delta x_n))_{n \in \NN}$, and  a sequence $((t_{k_n^*},x_{i_n^*}))_{n \in \NN} \subset \mathcal G_{\Delta_n}^*$  such that $(t_{k_n^*+1})_{n \in \NN}\subset \GG_{\Delta t_n}^*$,   $(t_{k_n^*},x_{i_n^*})\to (t^*,x^*)$, as $n\to \infty$,  and
\begin{equation}\label{eq: max 1}
    \begin{aligned}
        w_{\Delta_n}(t_q,x_j)\leq   w_{\Delta_n}(t_{k_n^*},x_{i_n^*})\qquad \text{for all } (t_q,x_j)\in  \GG_{\Delta_n}. 
    \end{aligned}
\end{equation}
 If we choose $t_q=t_{k_n^*+1}$  in \eqref{eq: max 1}, we have that  for every $x_j$
\begin{equation}\label{eq:dis step 5}
    v_{\Delta_n}[\mu_n](t_{k_n^*+1},x_j)\leq  \phi(t_{k_n^*+1},x_j)+\int_{t_{k_n^*+1}}^Tb(s)\dd s-\mathcal{E}_{\Delta_n}(t_{k_n^*+1}) + w_{\Delta_n}(t_{k_n^*},x_{i_n^*}). 
\end{equation}
Combining  \cref {properties infS} \ref{item: Monotone SL} with \eqref{eq:dis step 5}, and then using the identity
$  S_{\Delta_n}[\mu_n]( v_{k_n^*+1},t_{k_n^*}, x_{i_n^*})=v_{\Delta_n}[\mu_n](t_{k_n^*},x_{i_n^*})$,   \cref {properties infS} \ref{item: invariance SL}, \eqref{eq: relation H S} and the definition of  $w_{\Delta_n}(t_{k_n^*+1},x_{i_n^*})$, we obtain 
\begin{equation*}
    \begin{aligned}
        v_{\Delta_n}[\mu_n](t_{k_n^*},x_{i_n^*})\leq  & \, \phi(t_{k_n^*+1},x_{i_n^*})-\phi(t_{k_n^*},x_{i_n^*})-\Delta t_nH_{\Delta_n}(\hat \phi(t_{k_n^*+1}),t_{k_n^*},x_{i_n^*})+v_{\Delta_n}[\mu_n](t_{k_n^*},x_{i_n^*})\\ &+\int_{t_{k_n^*}}^{t_{k_n^*+1}}f(s,x_{i_n^*},\mu_n(t_{k_n^*}))\dd s -\int_{t_{k_n^*}}^{t_{k_n^*+1}}b(s)\dd s+\mathcal{E}_{\Delta_n}(t_{k_n^*})-\mathcal{E}_{\Delta_n}(t_{k_n^*+1}),  \\
    \end{aligned}
\end{equation*}
which, by definition of $\mathcal{E}_{\Delta_n}$, yields  
\begin{equation*}
    \begin{aligned}
     & \phi(t_{k_n^*},x_{i_n^*})-\phi(t_{k_n^*+1},x_{i_n^*}) +\Delta t_nH_{\Delta_n}(\hat \phi(t_{k_n^*+1}),  t_{k_n^*},x_{i_n^*}) \\ &\leq \int_{t_{k_n^*}}^{t_{k_n^*+1}}f(s,x_{i_n^*},\mu_n(t_{k_n^*}))\dd s-\int_{t_{k_n^*}}^{t_{k_n^*+1}}b(s)\dd s+\Delta t_nH_{\Delta_n}(\hat \phi(t^*),t_{k_n^*},x_{i_n})\\ &\quad -\int_{t_{k_n^*}}^{t_{k_n^*+1}}f(s,x_{i_n},\mu_n(t_{k_n^*}))\dd s-\int_{t_{k_n^*}}^{t_{k_n^*+1}}\Big (H(s,x^*,D_x \phi(t^*,x^*))-f(s,x^*,\mu(s))\Big )\dd s.\\ 
    \end{aligned}
\end{equation*}
Upon rearranging and dividing by $\Delta t_n$, we get
\begin{equation*}
    \begin{aligned}
     & \frac{\phi(t_{k_n^*},x_{i_n^*})-\phi(t_{k_n^*+1},x_{i_n^*})}{\Delta t_n} + \frac{1}{\Delta t_n}\int_{t_{k_n^*}}^{t_{k_n^*+1}}(H(s,x^*,D_x \phi(t^*,x^*))-f(s,x^*,\mu(s))+b(s)\Big )\dd s  \\ &  \leq H_{\Delta_n}(\hat \phi(t^*),t_{k_n^*},x_{i_n})- H_{\Delta_n}(\hat  \phi(t_{k_n^*+1}),t_{k_n^*},x_{i_n^*})\\ &\quad +\frac{1}{\Delta t_n}\int_{t_{k_n^*}}^{t_{k_n^*+1}} (f(s,x_{i_n^*},\mu_n(t_{k_n^*}))-f(s,x_{i_n},\mu_n(t_{k_n^*})) )\dd s.
    \end{aligned}
\end{equation*}
By assumption, we can estimate
\begin{equation*}
    \begin{aligned}
         &   \int_{t_{k_n^*}}^{t_{k_n^*+1}}G(s,x^*,D_x\phi(t^*,x^*)\dd s
        \leq \int_{t_{k_n^*}}^{t_{k_n^*+1}}(H(s,x^*,D_x \phi(t^*,x^*))-f(s,x^*,\mu(s))+b(s)\Big )\dd s
    \end{aligned}
\end{equation*}
to obtain that 
\begin{equation*}
    \begin{aligned}
     & \frac{\phi(t_{k_n^*},x_{i_n^*})-\phi(t_{k_n^*+1},x_{i_n^*})}{\Delta t_n} + \frac{1}{\Delta t_n}\int_{t_{k_n^*}}^{t_{k_n^*+1}}G(s,x^*,D_x\phi(t^*,x^*)\dd s   \\ & \leq H_{\Delta_n}(\hat \phi(t^*),t_{k_n^*},x_{i_n})- H_{\Delta_n}(\hat  \phi(t_{k_n^*+1}),t_{k_n^*},x_{i_n^*})\\ &\quad +\frac{1}{\Delta t_n}\int_{t_{k_n^*}}^{t_{k_n^*+1}} (f(s,x_{i_n^*},\mu_n(t_{k_n^*}))-f(s,x_{i_n},\mu_n(t_{k_n^*})) )\dd s.
    \end{aligned}
\end{equation*}
By \eqref{eq: lip f}, it follows that
\begin{equation*}
    \begin{aligned}
        \left|\frac{1}{\Delta t_n}\int_{t_{k_n^*}}^{t_{k_n^*+1}} (f(s,x_{i_n^*},\mu_n(t_{k_n^*}))-f(s,x_{i_n},\mu_n(t_{k_n^*})) )\dd s\right |\leq C_{f,2}|x_{i_n^*}-x_{i_n}|
    \end{aligned}
\end{equation*}
and, due to the one-step  error (\cref{properties infS} \ref{item: error SL}),  we get\\
\begin{equation*}
    \begin{aligned}
        \Big | H_{\Delta_n}(\hat \phi(t^*),t_{k_n^*},x_{i_n})&- H_{\Delta_n}(\hat  \phi(t_{k_n^*+1}),t_{k_n^*},x_{i_n^*})\Big |  \leq \mathcal{O}\left(\frac{(\Delta x_n)^2}{\Delta t_n}+\Delta t_n\right) \\ 
        &+ \left | \sup_{a\in A}\left\{a \cdot D_x\phi(t^*,x_{i_n^*})-\frac{1}{\Delta t_n} \int_{t_{k_n}}^{t_{k_n+1}}\ell(s,x_{i_n^*},a)\dd s \right\} \right.\\ & \left. -\sup_{a\in A}\left\{ a \cdot D_x\phi(t_{k_n+1},x_{i_n})-\frac{1}{\Delta t_n} \int_{t_{k_n}}^{t_{k_n+1}}\ell(s,x_{i_n},a)\dd s \right\}   \right| \\  
         &\leq  \mathcal{O}\left(\frac{(\Delta x_n)^2}{\Delta t_n}+\Delta t_n\right)+ \ov C_{\ell,2} |x_{i^*_n}-x_{i_n}|\\ &+ C_{\rm b}C_{\phi }|x_{i^*_n}-x_{i_n}|+ C_{\rm b}C_{\phi }|t_{k_n+1}-t^*|,\\
   \end{aligned}
\end{equation*}
where $C_{\phi}\coloneqq ||\phi||_{\CC^2((0,T)\times\RR^d)}$.
Finally,  passing to the limit as $n\to \infty$,
 we have that $v^*$ is a $L^1-$viscosity subsolution  to \eqref{eq:HJB}.
Analogously, we can obtain that $v_*$ is a $L^1$-viscosity supersolution.

Finally, by definition, $v^*\geq v_*$ and, by the comparison principle \cref{thm: comparison HJB}, $v^*\leq v_*$. In conclusion, by \cref{prop: uniqueness HJB}, $v^*=v_*=v[\mu]$, and the convergence over compact subsets holds thanks to \cite[Lemma~V.1.9]{BardiDolcetta}.
\end{proof}

Given $\epsilon>0$, consider the mollifier $\RR^d\ni x\mapsto \rho_{\epsilon}(x)=\rho(x/\epsilon)/\epsilon^d\in\RR$, where $\rho\in \CC^{\infty}(\RR^d)$ has bounded derivatives of any order and satisfies $\rho \geq 0$, $\int_{\RR^d}\rho(x) \dd x=1$, $\int_{\RR^d}|x|\rho(x)\dd x<\infty$,  and for any $m \in \NN^+$,  $\int_{\RR^d}|x|D^{m}\rho(x)\dd x<\infty$. Given $\varphi\in {\rm Lip}(\RR^{d})$, a standard computation shows that 
\begin{equation*}\label{eq:infty_norm_difference_eps}
\sup_{x\in\RR^{d}}\big|(\rho_{\epsilon}*\varphi)(x)-\varphi(x)\big|\leq\epsilon \|D\varphi\|_{L^{\infty}(\RR^{d})}\int_{\RR^d}|x|\rho(x)\dd x,
\end{equation*}
\begin{equation}\label{eq:infty_nabla_norm_difference_eps}
\sup_{x\in\RR^{d}}\big\|D^{m}(\rho_{\epsilon}*\varphi)(x)\big\|\leq \widetilde c_{m}  \epsilon^{1-m}\quad\text{for all }m\in\mathbb{N}, 
\end{equation}
where $\big\|D^{m}(\rho_{\epsilon}*\varphi)(x)\big\|$ denotes the operator norm of $D^{m}(\rho_{\epsilon}*\varphi)(x)$ and $\widetilde c_{m}>0$ depends  on $\int_{\RR^d}|x|D^{m}\rho(x)\dd x$ and $\|D\varphi\|_{L^{\infty}(\RR^{d})}$. Let us  define 
\begin{equation}\label{eq:v_epsilon}
v_{\Delta,\epsilon}[\mu](t,\cdot)=\rho_{\epsilon}*v_{\Delta }[\mu](t,\cdot)
\quad\text{for all }t\in[0,T].
\end{equation}

\begin{prop}\label{prop: properties decrete value function eps}
Assume \onlyref{hyp: compact set,hyp:terminal cost,hyp: lagrangian new},  let $(\Delta,\eps)=(\Delta t,\Delta x,\eps)\in(0,\infty)^3$ and  $\mu\in \CC([0,T];\P_1(\RR^d))$.
Then the following hold:
\begin{enumerate}[label={\rm(\roman*)},     ref={\rm(\roman*)}]
      \item \label{item: stability vDeltaEps} {\rm(Stability)}
Let $\widetilde{C}_{\text{{\rm v}}}>0$ be as in \cref{prop: properties decrete value function} \ref{item: stability vDelta}. Then
\begin{equation*}
|v_{\Delta,\eps}[\mu](t,x)|\leq \widetilde C_{\text{{\rm v}}}\quad\text{for all }(t,x)\in  [0,T]\times\RR^{d}.
\end{equation*}
\item \label{item: Lip space vDeltaEsp} {\rm(Lipschitz property in space)}
 Let $\widetilde{C}_{\text{{\rm Lip}}}>0$ be as in \cref{prop: properties decrete value function} \ref{item: Lip space vDelta}. Then
  \begin{equation*}
      \big|v_{\Delta,\eps}[\mu](t,x_1)-v_{\Delta,\eps}[\mu](t,x_2)\big|\leq \widetilde{C}_{\text{{\rm Lip}}}|x_1-x_2|\quad\text{for all } t\in[0,T], \, x_1,\,x_2\in\RR^d.
  \end{equation*}
  \item  \label{item: Discr Lip time vDeltaEps} {\rm(Discrete  Lipschitz property in time)}
Let $\widetilde {C}_{\rm tm}>0$ be as in \cref{prop: properties decrete value function} \ref{item: Discr Lip time vDelta}. Then,
\begin{equation*}
    |v_{\Delta,\eps}[\mu](t_{k_1},x)-v_{\Delta,\eps}[\mu](t_{k_2},x)|\leq \widetilde C_{\rm tm}|k_1-k_2|\Delta t \qquad \text{for all }k_1,\, k_2  \in \It, \  x\in \RR^d.
\end{equation*}
\item  \label{item: semiconc vDeltaEps} {\rm (Approximate semiconcavity)}   There exist $\widetilde{C}_{\rm asc},\overline{C}_{\rm asc}>0$  independent of $(\mu,\Delta ,\eps)$, such that
\begin{align}
 v_{\Delta,\eps}[\mu](t,x+y)-2v_{\Delta,\eps}[\mu](t,x)+   v_{\Delta,\eps}[\mu](t,x-y) 
        &\leq \widetilde{C}_{{\rm asc}}\Big (|y|^2 +(\Delta x)^2 +\frac{(\Delta x)^2}{\eps}\Big), \label{eq: approx semiconc 1}\\
  \big ( D_xv_{\Delta,\eps}[\mu](t,x)-D_xv_{\Delta,\eps}[\mu](t,y)\big )\cdot  (x-y)  & \leq  \overline{C}_{{\rm asc}}\Big (|y-x|^2 +\frac{(\Delta x)^2}{\eps^2}\Big),\label{eq: approx semiconc 2}
     \end{align}
 for all $t\in[0,T],\, x,\, y\in\RR^d$.
\item \label{item: gradient vDeltaEps}{\rm (Gradient property)} If $\eps \in (0,1)$, there exists $\widetilde{C}_{{\rm hb}}>0$  independent of $(\mu,\Delta,\eps)$ such that
\begin{equation*}\label{eq: Gradient property 1}
 D^2_{xx} v_{\Delta,\epsilon}[\mu](t,x)y\cdot  y \leq \widetilde{C}_{\text{{\rm hb}}}\Bigg(1+\frac{(\Delta x)^{2}}{\epsilon^{4}}\Bigg)|y|^2\quad\text{for all }t\in[0,T],\,x,\,y\in\RR^{d}.
\end{equation*}
 
\end{enumerate}
\end{prop}
\begin{proof}
Properties \onlyref{item: stability vDeltaEps,item: Lip space vDeltaEsp,item: Discr Lip time vDeltaEps} follow directly from \cref{prop: properties decrete value function} \onlyref{item: stability vDelta,item: Lip space vDelta,item: Discr Lip time vDelta}, respectively. Regarding the approximate semiconcavity \ref{item: semiconc vDeltaEps}, the first inequality  \eqref{eq: approx semiconc 1} is proved by the same technique as in \cite[Lemma 4.2]{AchdouCamilliCorrias}  exploiting \cref{prop: properties decrete value function} \ref{item: semiconv vDelta}, while the second inequality  \eqref{eq: approx semiconc 2} follows from \eqref{eq: approx semiconc 1} as in \cite[Remark 6]{AchdouCamilliCorrias}. Finally, assertion \ref{item: gradient vDeltaEps} follows from \eqref{eq: approx semiconc 1} and is proved as in \cite[Proposition~3.5]{CarliniSilvaZorkot}.
  
\end{proof}

We point out that \cref{prop: properties decrete value function eps}\ref{item: gradient vDeltaEps}  will  be used in the next section to establish the stability of the scheme for the continuity equation, which plays a key role in the convergence of the full scheme for the MFG system. 

\begin{prop} \label{prop: converg discr v eps}
Assume 
\onlyref{hyp: compact set,hyp: lagrangian new,hyp:terminal cost}.
Let  $((\Delta _n,\eps_n))_{n \in \NN}=((\Delta t_n, \Delta x_n,\eps_n))_{n\in \NN}\subset (0,\infty)^3$, $(\mu_n)_{n\in \NN}\subset\CC([0,T];\P_1(\RR^d))$, and  $\mu\in \CC([0,T];\P_1(\RR^d))$ be such that, as $n\to \infty$,  $(\Delta_n,\eps_n)\to 0$,   $(\Delta x_n)^2/\Delta t_n\to 0$, and $\mu_n\to \mu$ in $\CC([0,T];\P_1(\RR^d))$. Then the  following  hold.
\begin{enumerate}[label={\rm(\roman*)},     ref={\rm(\roman*)}]
      \item  \label{item: conv veps}$(v_{\Delta_n,\eps_n}[\mu_n])_{n \in \NN}$ converges to $v[\mu]$ uniformly over compact subsets of $[0,T]\times \RR^d$.
      \item \label{item: conv drift} If in addition $\Delta x_n/\eps_n \to 0$, then for every $K\subset [0,T]\times \RR^d$ compact and $q \in [1,+\infty)$, 
      \begin{equation*}
          \mathbb{I}_K(\cdot,\cdot)D_pH(\cdot,\cdot,D_xv_{\Delta_n,\eps_n}[\mu_n](\cdot,\cdot)) \to \mathbb{I}_K(\cdot,\cdot)D_pH(\cdot,\cdot,D_xv[\mu](\cdot,\cdot)) \quad \text{ in }L^q([0,T]\times \RR^d).
      \end{equation*}
\end{enumerate}
\end{prop}
\begin{proof}     The convergence   in  \ref{item: conv veps} is a direct consequence of \cref{prop: converg discr v} and  the uniform  
estimate \eqref{eq:infty_norm_difference_eps}. 
Moreover, thanks to the Lipschitz in space properties in \cref{prop: properties HJB} \ref{item: Lip v} and \cref{prop: properties decrete value function eps} \ref{item: Lip space vDeltaEsp},  and arguing as in \cite[Theorem 3.5]{CarliniSilva2014}   together with \eqref{eq: approx semiconc 2}, it follows that $D_xv_{\Delta_n,\eps_n}[\mu_n]$ converges to $D_xv[\mu]$ a.e. in $[0,T]\times \RR^d$ as  $n \to \infty$. By the continuity of $D_pH$ with respect to the third variable, established in \cref{rem: properties H} \ref{item: properies DpH}, 
we deduce that $D_pH(\cdot,\cdot,D_xv_{\Delta_n,\eps_n}[\mu_n](\cdot,\cdot))$
converges to $D_pH(\cdot,\cdot,D_xv[\mu](\cdot,\cdot))$ a.e in $[0,T]\times \RR^d$ as  $n \to \infty$.  Therefore, assertion \ref{item: conv drift}  follows from the  dominated
convergence theorem, since $||D_pH(\cdot,\cdot,D_xv_{\Delta_n,\eps_n}[\mu_n](\cdot,\cdot))||_{L^\infty([0,T]\times \RR^d)}\leq C_{\rm b}$  for all $n \in \NN$. 
\end{proof}
    
\subsection{ Lagrange-Galerkin  scheme for the continuity equation }\label{A_fully_discrete_scheme_for_the_CE}
Given $\mu  \in \CC([0,T];\P_1(\RR^d))$ and $(\Delta,\eps)= (\Delta t,\Delta x, \eps )\in(0,\infty)^3$, let $v_{\Delta,\eps}[\mu ]$ be defined as in \eqref{eq:v_epsilon}.
Consider the continuity equation
\begin{equation}\label{eq: CE deps}
    \begin{aligned}
        &\partial_tm(t,x)-\diver(m(t,x)D_pH(t,x,D_xv_{\Delta ,\eps}[\mu](t,x)))=0&&(t,x) \in (0,T)\times \RR^d, \\ 
        &m(0,x)=m^*_0(x) && x\in \RR^d.
    \end{aligned}
\end{equation}
We define, for every $i \in \ZZ^d$, the sets and the functions 
\begin{equation*} 
  E_{i}^{\Delta x}\coloneqq \{x\in\mathbb{R}^d \, \big| \,  |x-x_i|_\infty \leq \Delta x/2,\: x_i \in \GG_{\Delta x}\},  \qquad   \beta^{0,\Delta x}_i(x)\coloneqq \mathbb{I}_{E_{i}^{\Delta x}}(x).
\end{equation*}
We consider the following discrete scheme: define recursively $m\in \BB(\GG_{\Delta })$ as 
\begin{equation}\label{eq: Ce scheme}
\begin{aligned}
 m_{k+1,i}&=\frac{1}{{(\Delta x)^d}}\sum_{j\in\ZZ^d}m_{k,j}\int_{E_j^{\Delta x}}\beta_i^{0,\Delta x}(\mathsf{\Phi}_{k}^{\Delta,\epsilon}[\mu](x))\dd x &&\text{for all }(k,i)\in\It^{*}\times \ZZ^d,\\ 
m_{0,i}&=\frac{1}{(\Delta x)^d}\int_{E_{i}^{\Delta x}}m_{0}^{*}(x)\dd x &&\text{for all }i\in\ZZ^d,
\end{aligned}
\end{equation}
where 
\begin{equation} \label{Phi_mu_eps_k}
\mathsf{\Phi}_{k}^{\Delta,\epsilon}[\mu](x) \coloneqq  x- \int _{t_k}^{t_{k+1}}D _{p}H\big(s,x,D_x v_{\Delta,\epsilon}[\mu](t_k,x) \big)\dd s\quad\mbox{for all } k \in \It^*, \,  x\in\RR^d.    
\end{equation}
The scheme  \eqref{eq: Ce scheme} is explicit and conservative.

\begin{lem}\label{lem: compact supp m}
    Assume \onlyref{hyp: compact set,hyp: initial measure,hyp: lagrangian new,hyp:terminal cost}, let $(\Delta,\eps) = (\Delta t,\Delta x, \eps ) \in(0,\infty)^3$, $\mu  \in  \CC([0,T];\P_1(\RR^d))$, and  let $\{ m_{k,i} \,  |  \,  (k,i) \in   \It\times \ZZ^d\}$  be the solution to \eqref{eq: Ce scheme}. Then, for
every $c > 0$ there exists $\widetilde  C^*  > 0$, independent of $\mu$, such that if $\Delta x \leq  c\Delta t$, for every 
$k \in  \It$ we have $m_{k,i} =0$ if $x_i \notin    \ov B_\infty (0, \widetilde  C^* )$.
\end{lem}
\begin{proof}
    The proof is the same as in \cite[Lemma 4.1]{CarliniSilvaZorkot}, exploiting the boundedness of $D_pH$ established in  \cref{rem: properties H} \ref{item: properies DpH}.
\end{proof} 
Given  $\mu  \in \CC([0,T];\P_1(\RR^d))$, $(\Delta,\eps)\in(0,\infty)^3$ and $m\in \BB(\GG_{\Delta})$,  solution to \eqref{eq: Ce scheme}, we define $m_{\Delta,\eps}[\mu]\colon[0,T]\times \RR^d \to \RR$ as
\begin{equation*}
    \begin{aligned}
        m_{\Delta,\eps}[\mu](t_k,x)\coloneqq \sum_{j\in \ZZ^d}m_{k,j}\beta_j
        ^{0,\Delta x}(x)  \qquad \text{for all }k \in \It, \, x \in \RR^d,
    \end{aligned}
\end{equation*}
and,  if $t\in (t_k,t_{k+1})$,  
\begin{equation*}\label{mdelta}
 m_{\Delta,\eps}[\mu](t,x)\coloneqq  \frac{t_{k+1}-t}{\Delta t}m_{\Delta,\eps}[\mu](t_k,x)+ \frac{t-t_k}{\Delta t}m_{\Delta,\eps}[\mu](t_{k+1},x)  \qquad \text{for all }k\in \It^*,\, x \in \RR^d.
\end{equation*}
We collect below some useful properties of $m_{\Delta,\eps}[\mu]$. 

\begin{prop}\label{prop: properties discr m}
    Assume \onlyref{hyp: compact set,hyp: initial measure,hyp: lagrangian new,hyp:terminal cost}, let $(\Delta,\eps) = (\Delta t,\Delta x, \eps ) \in(0,\infty)^3$,  and $\mu  \in  \CC([0,T];\P_1(\RR^d))$. 
    Then the following hold:
\begin{enumerate}[label={\rm(\roman*)},     ref={\rm(\roman*)}]
        \item \label{item: positive m} $ m_{\Delta,\eps}[\mu](t,x)\geq 0$ for all $(t,x)\in [0,T]\times \RR^d$.
         \item \label{item: linear prop m}
        Let $a = (a_i)_{  i\in \ZZ^d} \subset  \RR$ and set $\varphi_a(x) =\sum_{ i\in \ZZ^d} a_i\beta^{0,\Delta x}_i (x)$ for every $ x \in  \RR^d$. Then, for every $k\in \It^*$, we have
         \begin{equation*}
             \begin{aligned}
                \int_{\RR^d} 
\varphi_a(x)m_{\Delta,\eps} [\mu ](t_{k+1},x)\dd x = \int_{\RR^d}
\varphi_a(\Phi^{\Delta,\eps}_k[\mu ](x))m_{\Delta,\eps}[\mu ](t_k,x)\dd x.
             \end{aligned}
         \end{equation*}
         \item \label{item: density m}   $\int_{\RR^d}m_{\Delta,\eps}[\mu](t,x)\dd x=1$ for every $t \in [0,T]$.
         \item \label{item: error m}  For every  $L$-Lipschitz function $\varphi\colon\RR^d\to \RR $ and for every $k \in \It^*$, we have
         \begin{equation*}
             \begin{aligned}
                 \left|  
\int_{\RR^d} \varphi (x)m_{\Delta,\eps}[\mu](t_{k+1},x)\dd x -\int_{\RR^d} \varphi (\Phi^{ \Delta,\eps}_k[\mu ](x))m_{\Delta,\eps}[\mu ](t_k,x)\dd x \right|  \leq L\sqrt{d}\Delta x.
             \end{aligned}
         \end{equation*}
         \end{enumerate}
   Furthermore, given $c>0$ such that $\Delta x \leq  c\Delta t$, the following properties hold.        
\begin{enumerate}[label={\rm(\roman*)},     ref={\rm(\roman*)}, start=5]
  \item \label{item: compact support m}    Let  $\widetilde  C^*  > 0$ be as in \cref{lem: compact supp m}, then $ \supp (m_{\Delta,\eps}[\mu](t))\subseteq \ov B_\infty(0,\widetilde  C^*) $ for all  $t \in [0,T]$.
         \item \label{item: Lip time m} $d_1(m_{\Delta,\eps}[\mu ](t),m_{\Delta,\eps}[\mu ](s))\leq (C_{\rm b} +c\sqrt{d})|t- s|$  for all  $t, s \in  [0,T]$.   
    \end{enumerate}
\end{prop}

\begin{proof}
Results \onlyref{item: positive m,item: linear prop m,item: density m} follow directly from the definition of $m_{\Delta,\eps}[\mu]$. Result \ref{item: error m} follows from the definition  of $m_{\Delta,\eps}[\mu]$ and assertion \ref{item: density m}, arguing as in \cite[Lemma 4.2]{CarliniSilvaZorkot}. Assertion \ref{item: compact support m} follows from \cref{lem: compact supp m}. The Lipschitz in time property \ref{item: Lip time m} is proved as in \cite[Proposition 4.4]{CarliniSilvaZorkot}, as a consequence of assertions \onlyref{item: error m,item: density m} and the boundedness of $D_pH$ established in \cref{rem: properties H} \ref{item: properies DpH}.
\end{proof}

Next, we prove a stability result, which follows arguments similar to those in \cite[Proposition 4.5]{CarliniSilvaZorkot}, using the notion of generalized gradient due to the lower regularity in our setting. Indeed, since the Hamiltonian is not smooth in our framework, we apply the chain rule in the sense of the generalized Jacobian (see Appendix \ref{appendix: generalized g}).
\begin{prop}\label{prop: stability CE}
Assume \onlyref{hyp: compact set,hyp: initial measure,hyp: lagrangian new,hyp:terminal cost}, let $(\Delta,\eps) = (\Delta t,\Delta x, \eps )\in(0,\infty)^2\times(0,1)$ and $\mu\in\CC([0,T];\P_1(\RR^d))$. For every $c_1$, $c_2>0$, there exists $\widetilde  c > 0$, independent of $(\mu,\Delta,\eps)$, such that if $\Delta x\leq c_1\Delta t$,  $\Delta t \leq   c_2\eps ^2$, and $\Delta t$ and $\eps$ are small enough, we have $m_{\Delta,\eps} [\mu ](t, \cdot ) \in L^{\rm p} (\RR^d)$ for all  $t \in  [0,T]$ and
\begin{equation*}
    \begin{aligned}
           ||m_{\Delta,\eps}[\mu](t)||_{L^{\rm p }(\RR^d)}\leq \widetilde c  ||m_0^*||_{L^{\rm p }(\RR^d)}  \qquad \text{for all }t \in [0,T].
    \end{aligned}
\end{equation*}

\end{prop}

\begin{proof}
  We proceed in several steps. 

\textit{Step 1. The map $\Phi_k^{\Delta,\eps}[\mu]\colon\RR^d\to\RR^d$ is a global bi-Lipschitz homeomorphism.} We show that for every $y \in \RR^d$ there exists a unique $x \in \RR^d$ such that $y=\Phi^{\Delta,\eps}_k[\mu](x)$. For simplicity, for a.e. $s \in [0,T]$ and for every $x \in \RR^d$, we define $ b(s,x)\coloneqq D_pH(s,x,D_xv_{\Delta,\eps}[\mu](t_k,x))$ and, given $y \in \RR^d$, set $F(x)\coloneqq y+\int_{t_k}^{t_{k+1}}b(s,x)\dd s$. Thus, $y=\Phi^{\Delta,\eps}_k[\mu](x)$ has a unique solution if and only if $F$ has a unique fixed-point. By \eqref{eq:Lips_gradient} and \eqref{eq:infty_nabla_norm_difference_eps},  we have
\begin{equation}
|b(s,x_1)-b(s,x_2)|\leq c_{H} \left(1+\frac{\widetilde c_2}{\eps } \right) |x_1-x_2|\quad\text{for all }x_1,\,x_2 \in \RR^d.
\label{eq: Lip b}
\end{equation}
Setting $\vartheta_{\Delta,\eps}:=c_{H} \Delta t (1+\frac{\widetilde c_2}{\eps })$, we deduce that $|F(x_1)-F(x_2)| \leq  \vartheta_{\Delta,\eps}|x_1-x_2|$. Our assumptions on $(\Delta,\epsilon)$ imply the existence of $\vartheta_0$, independent of $(\Delta,\epsilon)$ small enough, such that  $\vartheta_{\Delta,\eps}\leq\vartheta_0<1$. In particular,
the map $F$ is a contraction. Consequently, $\Phi^{\Delta,\eps}_k[\mu]$ is bijective.  Arguing in the same way, the function $\Phi_k^{\Delta,\eps}[\mu]$ is bi-Lipschitz.  Indeed,  for every $x_1,x_2 \in \RR^d$,~\eqref{eq: Lip b} implies 
\begin{equation*}
|\Phi_k^{\Delta,\eps}[\mu](x_1)-\Phi_k^{\Delta,\eps}[\mu](x_2)|\leq  |x_1-x_2|+ \vartheta_{\Delta,\eps}|x_1-x_2| \leq (1+\vartheta_0)|x_1-x_2|
\end{equation*}
and
$$
|\Phi_k^{\Delta,\eps}[\mu](x_1)-\Phi_k^{\Delta,\eps}[\mu](x_2)|\geq |x_1-x_2|
-\int_{t_k}^{t_{k+1}}|b(s,x_1)-b(s,x_2)|\dd s\geq (1-\vartheta_{\Delta,\eps})|x_1-x_2|,
$$
which shows that 
$$
\left|\big(\Phi_k^{\Delta,\eps}[\mu]\big)^{-1}(y_1)
-\big(\Phi_k^{\Delta,\eps}[\mu]\big)^{-1}(y_2)\right|
\leq \frac{1}{1-\vartheta_{0}}|y_1-y_2|
\qquad\text{for all }y_1,y_2\in\RR^d.
$$

\noindent\textit{Step 2. Generalized differential structure of the discrete
flow.} For $k\in \It^*$, set 
\begin{equation*}
q_k(x)\coloneqq D_xv_{\Delta,\eps}[\mu](t_k,x),\quad S_k(x)\coloneqq D^2_{xx}v_{\Delta,\eps}[\mu](t_k,x),\quad\text{and}\quad F_k(x)\coloneqq (x,q_k(x)).
\end{equation*}
For a.e. $s\in[t_k,t_{k+1}]$, define $G_s(x,p)\coloneqq D_pH(s,x,p)$ and notice that $b(s,\cdot)=G_s\circ F_k$. The map $F_k$ is continuously differentiable with $DF_k(x)=\left(
Id, S_k(x)\right)$ for all $x\in\RR^{d}$, whereas $G_s$ is Lipschitz by \cref{rem: properties H}\ref{item: properies DpH}. Therefore, \cref{thm: chain rule generalized J}  and  \cref{prop: generalized Jacobian elementary properties} \ref{item: Clarke cont diff} imply that 
\begin{equation}\label{eq: joint chain rule b}
\begin{aligned}
\partial\big(b(s,\cdot)\big)(x)
& \subseteq \co \big (  \partial G_s(x,q_k(x))DF_{k}(x) \big ) \\
&=\co \big ( 
\left\{B_x+B_pS_k(x)\,\middle|\,
[B_x\;B_p]\in\partial G_s(x,q_k(x))\right\} \big )  \\
&=\left\{B_x+B_pS_k(x)\,\middle|\,
[B_x\;B_p]\in\partial G_s(x,q_k(x))\right\},
\end{aligned}
\end{equation}
where in the last equality we have used that $\partial G_s(x,q_k(x))$ is convex (see  \cref{prop: generalized Jacobian elementary properties} \ref{item: Clarke convex set}). 
Moreover, for every $[B_x\;B_p]\in\partial G_s(x,q_k(x))$, one has
\begin{equation}\label{eq: properties joint blocks}
\|B_x\|\leq c_H,\qquad \|B_p\|\leq c_H,
\qquad B_p=B_p^{\mathsf T},
\qquad\text{and}\quad B_p\text{ is positive semidefinite}.
\end{equation}
Indeed, the norm estimates follow \cref{rem: properties H}\ref{item: properies DpH}. To justify the last two properties, first consider a point $(x',p')$ at which $G_s$ is differentiable. The $p$-block of
$DG_s(x',p')$ is given by $D^{2}_{p}H(s,x',p')$, which is symmetric and positive semidefinite by the convexity of $p\mapsto H(s,x',p)$. By definition, $B_p$ is a convex
combination of limits of such $p$-blocks and hence it is also symmetric and positive semidefinite. Let us write $\Phi_k^{\Delta,\eps}[\mu](x)=x-\mathscr{B}_k(x)$ with $\mathscr{B}_k(x)\coloneqq\int_{t_k}^{t_{k+1}}b(s,x)\dd s$. It follows from~\eqref{eq: Lip b} and~\cref{lem: differentiation under the integral sign}   that 
\begin{equation}\label{eq: generalized Jacobian integral b}
\partial\mathscr{B}_k(x)
\subseteq
\int_{t_k}^{t_{k+1}}\partial\big(b(s,\cdot)\big)(x)\dd s,
\end{equation}
where the integral on the right-hand side is the Aumann integral defined in~\eqref{eq:Aumann_integral} below. Thus, for every
$N\in\partial\mathscr{B}_k(x)$, there exists a measurable map $C\colon [t_k,t_{k+1}]\to\RR^{d\times d}$ such that
\begin{equation}
N=\int_{t_k}^{t_{k+1}}C(s)\dd s\quad\text{and}\quad C(s)\in\partial\big(b(s,\cdot)\big)(x)\quad\text{for a.e. }s\in [t_k,t_{k+1}].
\label{eq:N_properties}
\end{equation}
Notice that~\eqref{eq: Lip b} and our assumptions on $(\Delta,\eps)$ give
\begin{equation}
\|N\|\leq c_H\Delta t
\left(1+\frac{\widetilde c_2}{\eps}\right)
=\vartheta_{\Delta,\eps}=\mathcal{O}\left(\sqrt{\Delta t}\right)
\label{eq:N_stima}
\end{equation}
and, by~\cref{prop: generalized Jacobian elementary properties} \ref{item: Clarke linear} 
and~\eqref{eq: generalized Jacobian integral b}, every
$M\in\partial\Phi_k^{\Delta,\eps}[\mu](x)$ can be written as $M=Id-N$, with $N$ satisfying~\eqref{eq:N_properties} and~\eqref{eq:N_stima}. 

\textit{Step 3. $L^p$ bound.}  We first observe that by Step 1,  $(\Phi^{\Delta,\eps}_k[\mu])^{-1}$ is globally  Lipschitz on $\ov B_\infty(0, \widetilde C^*)$ and
there exists $R>0$ such that $(\Phi^{\Delta,\eps}_k[\mu])^{-1}(\ov B_\infty(0,\widetilde C^*))\subseteq \ov B_\infty(0,R)$. 
Thus, by~\cref{prop: properties discr m} \ref{item: compact support m} and the change-of-variables formula~\cite[Theorem 3.9]{EvansGariepy2015}, obtain that
\begin{eqnarray}\label{eq: M discr p 1}
        \int_{\RR^d}\Big(m_{\Delta,\eps}[\mu](t_{k+1},\Phi^{\Delta,\eps}_k[\mu](x))\Big )^{\rm p}\dd x=
        \int_{\ov B_\infty(0,R)}\Big(m_{\Delta,\eps}[\mu](t_{k+1},\Phi^{\Delta,\eps}_k[\mu](x))\Big )^{\rm p}\dd x 
        \nonumber \\ 
         =\int_{\Phi^{\Delta,\eps}_k[\mu](\ov B_\infty(0,R))} \Big(m_{\Delta,\eps}[\mu](t_{k+1},y)\Big )^{\rm p} \Big |\det D((\Phi^{\Delta,\eps}_k[\mu])^{-1})(y) \Big|\dd y.     
\end{eqnarray}
Denote by $\mathcal{N}_{1}$ the set of points at which $\Phi_k^{\Delta,\eps}[\mu]$ is not differentiable. By Step 1 and Rademacher's theorem, $\mathcal{N}_{1}$ has Lebesgue measure zero and, by~\cite[Theorem~2.8(i)]{EvansGariepy2015}, $\Phi_k^{\Delta,\eps}[\mu](\mathcal{N}_{1})$ has also measure zero. The inverse map is Lipschitz and is therefore differentiable outside a null set $\mathcal{N}_{2}$. Consequently, for a.e. $y\notin \mathcal{N}_{2}\cup\Phi_k^{\Delta,\eps}[\mu](\mathcal{N}_{1})$, the maps $\Phi_k^{\Delta,\eps}[\mu]$ and $(\Phi^{\Delta,\eps}_k[\mu])^{-1}$  are differentiable
at $y$ and at $x=(\Phi_k^{\Delta,\eps}[\mu])^{-1}(y)$, respectively.  Then, from $\Phi_k^{\Delta,\eps}[\mu]\circ(\Phi_k^{\Delta,\eps}[\mu])^{-1}=Id$, we deduce that
\begin{equation*}
\left|\det D(\Phi_k^{\Delta,\eps}[\mu])^{-1}(y)\right|
=\frac{1}{\left|\det D\Phi_k^{\Delta,\eps}[\mu](x)\right|}
\end{equation*}
and, hence,~\eqref{eq: M discr p 1} yields  
\begin{equation}
\int_{\RR^d}\Big(m_{\Delta,\eps}[\mu](t_{k+1},\Phi^{\Delta,\eps}_k[\mu](x))\Big )^{\rm p}\dd x=\int_{\Phi^{\Delta,\eps}_k[\mu](\ov B_\infty(0,R))}\frac{\Big(m_{\Delta,\eps}[\mu](t_{k+1},y)\Big )^{\rm p}}{\Big |\det D(\Phi^{\Delta,\eps}_k[\mu])((\Phi^{\Delta,\eps}_k[\mu])^{-1}(y)) \Big|} \dd y.
\label{eq:lp_estimate_determinant}
\end{equation}
Let us prove that there exists a constant $c_3>0$, independent of
$(\mu,\Delta,\eps)$, such that 
\begin{equation}
\left|\det\big(D\Phi_k^{\Delta,\eps}[\mu](x)\big)\right|^{-1}
\leq 1+c_3\Delta t\quad\text{for a.e. }x\in\ov B_\infty(0,R).
\label{eq:inverse_derterminant_bound}
\end{equation}
Let $x\in\ov B_\infty(0,R)$ be such that $\Phi_k^{\Delta,\eps}[\mu]$ is differentiable at $x$. By \cref{prop: generalized Jacobian elementary properties} \ref{item: Clarke diff}, one has $D\Phi_k^{\Delta,\eps}[\mu](x)\in\partial\Phi_k^{\Delta,\eps}[\mu](x)$ and hence, by Step~2, there exist a measurable map $C\colon [t_k,t_{k+1}]\to\RR^{d\times d}$ such that
\begin{equation}
D\Phi_k^{\Delta,\eps}[\mu](x)=Id-\int_{t_k}^{t_{k+1}}C(s)\dd s,\qquad
C(s)\in\partial\big(b(s,\cdot)\big)(x)
\quad\text{for a.e. }s.
\label{eq: representation derivative flow}
\end{equation}
Thus, by~\cref{lem: trace},~\eqref{eq:N_stima}, and our assumptions on $(\Delta,\eps)$, 
\begin{equation}
\left|\left|\det D\Phi_k^{\Delta,\eps}[\mu](x)\right|^{-1}-\left(1+{\rm tr}\left(\int_{t_k}^{t_{k+1}}C(s)\dd s\right)\right)\right|
=\mathcal{O}\left(\Delta t\right).
\label{eq:det_O}
\end{equation}
Fix $s$ for
which the inclusion~\eqref{eq: joint chain rule b} holds and
$C(s)\in\partial(b(s,\cdot))(x)$. There exists
$[B_x\;B_p]\in\partial G_s(x,q_k(x))$ such that $C(s)=B_x+B_pS_k(x)$.  By~\eqref{eq: properties joint blocks},
$B_p$ is symmetric and positive semidefinite. Since $S_k(x)$ is symmetric,~\eqref{eq: properties joint blocks}, 
\cref{lem: algebraicmatrix},
\cref{prop: properties decrete value function eps} \ref{item: gradient vDeltaEps}, and the conditions on $(\Delta,\eps)$ give  
\begin{multline*}
{\rm tr}(C(s))={\rm tr}(B_x)+{\rm tr}(B_pS_k(x))\leq d\|B_x\|+\lambda_{\max}(S_k(x)){\rm tr}(B_p)\\
\leq dc_H+\widetilde C_{\rm hb}
\left(1+\frac{(\Delta x)^2}{\eps^4}\right)dc_H\leq dc_H
+\widetilde C_{\rm hb}(1+c_1^2c_2^2)dc_H
\eqqcolon C_{\rm tr},
\end{multline*}
which, combined with~\eqref{eq:det_O}, yields~\eqref{eq:inverse_derterminant_bound}. 
Finally, from \eqref{eq:lp_estimate_determinant} and   \eqref{eq:inverse_derterminant_bound}, we have 
\begin{equation}\label{eq: M discr p 2}
    \begin{aligned}
         \int_{\RR^d}\Big(m_{\Delta,\eps}[\mu](t_{k+1},\Phi^{\Delta,\eps}_k[\mu](x))\Big )^{\rm p}\dd x\leq  \Big (1+ c_3\Delta t\Big ) \int_{\RR^d}\Big(m_{\Delta,\eps}[\mu](t_{k+1},x)\Big )^{\rm p}\dd x.
    \end{aligned}
\end{equation}
If $\rm p \neq \infty$, setting $\rm p^*  =\rm p /({\rm p}- 1)$, it follows from \cref{prop: properties discr m} \ref{item: linear prop m}, Holder's inequality and  \eqref{eq: M discr p 2}, that 
\begin{equation*}
    \begin{aligned}
        \|m_{\Delta,\eps}[\mu](t_{k+1},\cdot) \|_{L^{\rm p}(\RR^d)}^{\rm p}& =  \int_{\RR^d}(m_{\Delta,\eps}[\mu](t_{k+1},x))^{\rm p-1}m_{\Delta,\eps}[\mu](t_{k+1},x)\dd x  \\ 
        & =\int_{\RR^d}(m_{\Delta,\eps}[\mu](t_{k+1},\Phi^{\Delta,\eps}_k[\mu](x)))^{\rm p-1}m_{\Delta,\eps}[\mu](t_{k},x)\dd x \\ 
        &\leq \Big ( \int_{\RR^d}(m_{\Delta,\eps}[\mu](t_{k+1},\Phi^{\Delta,\eps}_k[\mu](x)))^{\rm p} \dd x \Big )^{\frac{1}{\rm p^*}}\| m_{\Delta,\eps}[\mu](t_{k},\cdot)\|_{L^{\rm p}(\RR^d)} \\ 
        & \leq \Big (1+  c_3\Delta t   \Big ) \| m_{\Delta,\eps}[\mu](t_{k+1},\cdot)\|^{\rm p-1}_{L^{\rm p}(\RR^d)} \| m_{\Delta,\eps}[\mu](t_{k},\cdot)\|_{L^{\rm p}(\RR^d)}
    \end{aligned}
\end{equation*}
and then  
\begin{equation*}
    \begin{aligned}
        \|m_{\Delta,\eps}[\mu](t_{k+1},\cdot)\|_{L^{\rm p}(\RR^d)} 
        & \leq \Big (1+  c_3\Delta t\Big) \| m_{\Delta,\eps}[\mu](t_{k},\cdot)\|_{L^{\rm p}(\RR^d)}.
    \end{aligned}
\end{equation*}
Iterating on $k$, we obtain   
\begin{equation*}
    \begin{aligned}
        \|m_{\Delta,\eps}[\mu](t_{k+1},\cdot) \|_{L^{\rm p}(\RR^d)}  \leq e^ {c_3T}|| m^*_0||_{L^{\rm p}(\RR^d)},\\
    \end{aligned}
\end{equation*}
and then the thesis with $\widetilde c= e^ {c_3T}$.
If ${\rm p}=\infty$, the thesis follows by letting $\rm p \to \infty$.

\end{proof}

We conclude this part with the following convergence result, which combines the properties of the proposed schemes with \cref{prop: converg discr v eps}.

\begin{prop}\label{prop: convergence}
  Assume \onlyref{hyp: compact set,hyp: initial measure,hyp: lagrangian new,hyp:terminal cost}. Let $((\Delta_n,\eps_n))_{n\in \NN}=((\Delta t_n,\Delta x_n, \eps_n ))_{n\in \NN} \subset(0,\infty)^3$, $(\mu_n)_{n\in \NN}  \subset \CC([0,T];\P_1(\RR^d))$, and $\mu  \in  \CC([0,T];\P_1(\RR^d))$  be such that, as $n\to \infty$, $(\Delta_n,\eps_n) \to 0$, $\Delta x_n = o(\Delta t_n)$, $\Delta t_n = \mathcal{O}((\eps_n)^2)$, and  $\mu_n \to \mu$ in $\CC([0,T];\P_1(\RR^d))$. 
  Then, up to some subsequence, the following hold:
\begin{enumerate}[label={\rm(\roman*)},     ref={\rm(\roman*)}]
    \item \label{item: conv veps in m} $(v_{\Delta_n,\eps_n}[\mu_n])_{n \in \NN}$ converges to $v[\mu ]$, uniformly over compact subsets of $[0,T]\times \RR^d$,
and, for every  $K \subset  [0,T] \times  \RR^d$ compact and $q \in  [1,\infty )$, $(\mathbb{I}_K (D_xv_{\Delta_n,\eps_n}[\mu_n]))_{n \in \NN}$
converges to $\mathbb{I}_K (D_xv[\mu]) $ in $L^q([0,T] \times \RR^d)$.
 \item \label{item: conv m} $(m_{\Delta_n,\eps_n}[\mu_n])_{n \in \NN}$ converges in $\CC([0,T];\P_1 (\RR^d))$ towards a solution $m\in  \CC([0,T];\P_1 (\RR^d)) \cap  L^{\mathsf{p}} ([0,T] \times  \RR^d) $ to \eqref{eq:CE}. Moreover, the convergence also holds
weakly in $L^{\rm p} ([0,T] \times \RR^d)$ if ${\rm p}  <\infty$ , and weakly$^*$  in $L^\infty ([0,T] \times  \RR^d)$ if ${\rm p}  =\infty$.
In addition, there exists $\widetilde c>0$ such that 
\begin{equation}\label{eq: bounded Lp}
    \begin{aligned}
           ||m(t)||_{L^{\rm p }(\RR^d)}\leq \tilde c  ||m_0^*||_{L^{\rm p }(\RR^d)}  \qquad \text{for all }t \in [0,T].
    \end{aligned}
\end{equation}
\end{enumerate}
\end{prop}
\begin{proof}

The convergence in~\ref{item: conv veps in m} corresponds to  \cref{prop: converg discr v eps}. We now prove~\ref{item: conv m}. By \cref{prop: properties discr m}~\onlyref{item: positive m,item: density m,item: compact support m}, the sequence $(m_{\Delta_n,\eps_n}[\mu_n](t))_{n \in \NN}$ is tight and has uniformly integrable moments. Therefore, by \cite[Proposition~7.1.5]{AmbrosioGigliSavare}, this set is compact in $\P_1(\RR^d)$ for every $t\in [0,T]$. Combined with the Lipschitz-in-time property~\ref{item: Lip time m} in \cref{prop: properties discr m}, this allows us to apply the Ascoli--Arzelà theorem, yielding that $(m_{\Delta_n,\eps_n}[\mu_n])_{n \in \NN}$ is compact in $\CC([0,T];\P_1(\RR^d))$. Hence, there exists $m \in \CC([0,T];\P_1(\RR^d))$ such that, up to some subsequence, $m_{\Delta_n,\eps_n}[\mu_n] \to m$ in $\CC([0,T];\P_1(\RR^d))$. By \cref{prop: stability CE}, the convergence also holds weakly in $L^{\rm p}([0,T]\times\RR^d)$ if ${\rm p}<\infty$, or weakly$^*$ in $L^\infty([0,T]\times\RR^d)$ if ${\rm p}=\infty$, and $m$ satisfies~\eqref{eq: bounded Lp}.
Following the argument in \cite[Proposition 4.6]{CarliniSilvaZorkot}, and using \cref{rem: properties H}, \cref{prop: properties discr m}, and assertion~\ref{item: conv veps in m}, one deduces that $m$ solves \eqref{eq:CE}.
\end{proof}
\subsection{A Lagrangian scheme for  the MFG system}
Let $(\Delta,\eps)= (\Delta t,\Delta x, \eps ) \in(0,\infty)^3$, we define 
\begin{equation*}
    \Ix\coloneqq \{i \in \mathbb{Z}^d  \, | \, x_i=i\Delta x \in \ov B_\infty(0, \widetilde C^*) \}
\end{equation*} 
and
\begin{equation*}
    \begin{aligned}
        \DD_{\Delta}\coloneqq \left\{ \mu=(\mu_{k,i})_{(k,i)\in \It\times \Ix} \,  \Bigg| \, \mu_{k,i}\geq 0, \, \sum_{i \in \Ix}\mu_{k,i}(\Delta x)^d=1, \, \text{for all } (k,i)\in \It \times \Ix\right\}.
    \end{aligned}
\end{equation*}
We observe that $\DD_{\Delta}$ is a convex and compact subset
of $\RR^{(N_T+1)\times (2N_{\Delta x}+1)^d}$, where $N_{\Delta x}\coloneqq \left\lceil\frac{\widetilde C^*}{\Delta x} \right\rceil$. 
Given $\mu \in \DD_{\Delta}$, we define its extension $\widetilde \mu \in \CC ([0,T];\P_1(\RR^d))$ as
\begin{equation}\label{eq: mu tilde tk}
    \begin{aligned}
       \dd [\widetilde \mu(t_k)](x)\coloneqq  \sum_{j\in \Ix}\mu_{k,j}\beta_j
        ^{0,\Delta x}(x)  \qquad \text{for all }k \in \It \text{ and }x \in \RR^d ,
    \end{aligned}
\end{equation}
and  if $t\in (t_k,t_{k+1})$ with $ k \in \It^*$ as   
\begin{equation}\label{eq: mu tilde t}
  \dd [\widetilde \mu(t)](x)\coloneqq  \frac{t_{k+1}-t}{\Delta t} \dd [\widetilde \mu(t_k)](x)+ \frac{t-t_k}{\Delta t} \dd [\widetilde \mu(t_{k+1})](x)  \qquad \text{for all }x \in \RR^d.
\end{equation}
We consider the following discretization of \eqref{eq:MFG}: find $\mu \in \DD_{\Delta}$ such that
\begin{equation}\label{eq:MFGd}\tag{MFG$_{\Delta,\eps}$}
    \begin{aligned}
        \mu_{k,i}=m_{\Delta,\eps}[\widetilde \mu](t_k,x_i) \text{ for every } (k,i)\in \It\times \Ix.
        \end{aligned}
\end{equation} 

\begin{theo}\label{thm: existence MFGd}
     Assume \onlyref{hyp: compact set,hyp: initial measure,hyp: lagrangian new,hyp:terminal cost}, let  $(\Delta,\eps) = (\Delta t,\Delta x, \eps ) \in(0,\infty)^3$.
     Then, if $\Delta t/\eps$  is small enough and there exists $c > 0$ with $\Delta x \leq  c\Delta t$, 
     \eqref{eq:MFGd} admits at least one solution.
\end{theo}    
\begin{proof}
We define the map $T \colon \mathscr{D}_{\Delta} \to \RR^{(N_T+1)\times (2N_{\Delta x}+1)^d}$ by
\begin{equation*}
    \big(T(\mu)\big)_{k,i} = m_{\Delta,\eps}[\widetilde\mu](t_k,x_i) \quad \text{for all } (k,i)\in \It^*\times \Ix,
\end{equation*}
and observe that a solution to \eqref{eq:MFGd} is equivalently a fixed point of $T$. The proof then follows by applying Brouwer's fixed-point theorem.
\cref{prop: properties discr m} \onlyref{item: density m,item: positive m,item: compact support m} imply that  $T(\mathscr{D}_{\Delta}) \subseteq  \mathscr{D}_{\Delta}$. 
 We now show that $T$ is continuous. Let $(\mu_n)_{n \in \NN}\subset \mathscr{D}_{\Delta }$ and $\mu \in \mathscr{D}_{\Delta }$ such that $\mu_n \to \mu$, as $n\to \infty$. We prove that $T(\mu_n)\to T(\mu)$, as $n\to \infty$.  
Let $(\widetilde{\mu}_n)_{n \in \NN}\subset \CC([0,T];\P_1(\RR^d))$ and  $\widetilde{\mu} \in \CC([0,T];\P_1(\RR^d))$ be as defined in \eqref{eq: mu tilde tk}--\eqref{eq: mu tilde t}, corresponding to $\mu_n$ and $\mu$, respectively.
It follows that $\widetilde{\mu}_n \to \widetilde{\mu}$ in $\CC([0,T];\P_1(\RR^d))$ as $n\to \infty$. The continuity of $f(t,x,\cdot)$ and $g(x,\cdot)$ assumed in \ref{hyp:terminal cost} implies that $v_{\Delta}[\widetilde{\mu}_n] \to v_{\Delta}[\widetilde{\mu}]$ pointwise in $[0,T]\times \RR^d$ as $n\to \infty$. By \Cref{prop: properties decrete value function}\ref{item: stability vDelta} and the dominated convergence theorem, $D_xv_{\Delta,\eps}[\widetilde{\mu}_n] \to D_xv_{\Delta,\eps}[\widetilde{\mu}]$ pointwise in $[0,T]\times \RR^d$ as $n\to \infty$. By \cref{rem: properties H}~\ref{item: properies DpH}, for every $k \in \It^*$, $\Phi^{\Delta,\eps}_k[\widetilde{\mu}_n]\to \Phi^{\Delta,\eps}_k[\widetilde{\mu}]$ pointwise in $\RR^d$ as $n\to \infty$. In particular,
\begin{equation*}
\beta^{0,\Delta x}_i\big(\Phi^{\Delta,\eps}_k[\widetilde{\mu}_n](x)\big) \to \beta^{0,\Delta x}_i\big(\Phi^{\Delta,\eps}_k[\widetilde{\mu}](x)\big) \quad \text{for all } x \in \RR^d \setminus \big((\Phi^{\Delta,\eps}_k[\widetilde{\mu}])^{-1}(\partial E_i^{\Delta x})\big).
\end{equation*}
Moreover, the set $(\Phi^{\Delta,\eps}_k[\widetilde{\mu}])^{-1}(\partial E_i^{\Delta x})$ has null measure. Indeed, as shown in \textit{Step 1} of the proof of \cref{prop: stability CE}, for $\Delta t/\eps$ sufficiently small, $\Phi^{\Delta,\eps}_k[\mu]$ is a global bi-Lipschitz homeomorphism. Since $\partial E_i^{\Delta x}$ has null measure in $\RR^d$, it follows that $(\Phi^{\Delta,\eps}_k[\widetilde{\mu}])^{-1}(\partial E_i^{\Delta x})$ also has null measure in $\RR^d$. Therefore, $\beta^{0,\Delta x}_i(\Phi^{\Delta,\eps}_k[\widetilde{\mu}_n](x)) \to \beta^{0,\Delta x}_i(\Phi^{\Delta,\eps}_k[\widetilde{\mu}](x))$ almost everywhere in $\RR^d$. By the dominated convergence theorem,
\begin{equation*}
    \int_{E_j^{\Delta x}}
    \beta^{0,\Delta x}_i\big(\Phi^{\Delta,\eps}_k[\widetilde{\mu}_n](x)\big)\dd x 
    \to   
    \int_{E_j^{\Delta x}}
    \beta^{0,\Delta x}_i\big(\Phi^{\Delta,\eps}_k[\widetilde{\mu}](x)\big)\dd x
    \quad \text{for all } k \in \It^*,\, i,j \in \mathbb{Z}^d,
\end{equation*}
as $n\to\infty$. By definition, this yields $T(\mu_n)\to T(\mu)$ as $n\to\infty$, hence $T$ is continuous. The existence of a fixed point of $T$, and therefore of a solution to \eqref{eq:MFGd}, follows from Brouwer's fixed-point theorem.

\end{proof}
\begin{theo}\label{thm: convergence scheme}
      Assume \onlyref{hyp: compact set,hyp: initial measure,hyp: lagrangian new,hyp:terminal cost}. Let  $((\Delta _n,\eps_n))_{n\in \NN}=((\Delta t_n, \Delta x_n,\eps_n))_{n\in \NN}\subset (0,\infty)^3 $ be such that, as $n\to \infty$, $(\Delta_n,\eps_n) \to 0$, $\Delta x_n = o(\Delta t_n)$, and $\Delta t_n = \mathcal{O}((\eps_n)^2)$. 
For every $n$ large enough, let $m_n \in  \DD_{\Delta_n}$ be a solution to {\rm MFG}$_{\Delta_n,\eps_n}$ and define
$\widetilde m_n$ by \eqref{eq: mu tilde tk}-\eqref{eq: mu tilde t}.
Then, there exists a solution $(v^*,m^*)$ to \eqref{eq:MFG} such that, up to some subsequence, the following hold:
\begin{enumerate}[label={\rm(\roman*)},     ref={\rm(\roman*)}]
    \item \label{item: convergence v}  $(v_{\Delta_n,\eps_n}[\widetilde m_n])_{n \in \NN}$ converges to $v^*$, uniformly over compact subsets of $[0,T]\times \RR^d$.
    \item \label{item: convergence m}   $(\widetilde m_n)_{n \in \NN}$ converges in $\CC([0,T];\P_1 (\RR^d))$ towards  $m^*$. Moreover, the convergence also holds
weakly in $L^{\rm p} ([0,T] \times \RR^d)$ if ${\rm p}  <\infty$ , and weakly$^*$  in $L^\infty ([0,T] \times  \RR^d)$ if ${\rm p}  =\infty$.
In addition, there exists $\widetilde c>0$ such that 
\begin{equation*}
    \begin{aligned}
           ||m^*(t)||_{L^{\rm p }(\RR^d)}\leq \widetilde c  ||m_0^*||_{L^{\rm p }(\RR^d)}  \qquad \text{for all }t \in [0,T].
    \end{aligned}
\end{equation*}
\end{enumerate}
 If, in addition, \cref{hyp: monoticityassumption} is satisfied, the whole sequence $((v_{\Delta_n,\varepsilon_n}[\widetilde m_n],\widetilde m_n))_{n \in \NN}$ converges to the unique solution $(v^*,m^*)$.
\end{theo}

\begin{proof}
For all $n$ sufficiently large, the existence of a solution $m_n \in \mathscr{D}_{\Delta_n}$ follows from \cref{thm: existence MFGd}, and we have $\widetilde{m}_n = m_{\Delta_n,\eps_n}[\widetilde{m}_n]$. Arguing as in \cref{prop: convergence}~\ref{item: conv m}, there exists $m^* \in \CC([0,T];\P_1(\RR^d))$ such that, up to some subsequence, $\widetilde{m}_n \to m^*$ in $\CC([0,T];\P_1(\RR^d))$ as $n\to\infty$. By \cref{prop: convergence}~\ref{item: conv m}, $m^*$ solves \eqref{eq:CE} with $\mu = m^*$, and therefore $(v[m^*],m^*)$ solves \eqref{eq:MFG}. Finally, assertions~\onlyref{item: convergence v,item: convergence m} follow respectively from \onlyref{item: conv veps in m,item: conv m} in \cref{prop: convergence}. If  \cref{hyp: monoticityassumption} is satisfied,  since \Cref{thm: unicità} implies the uniqueness of the solution to \eqref{eq:MFG}, the whole sequence converges to $(v^*,m^*)$.

\end{proof}

\section{Numerical Simulations}
We present numerical results to evaluate the performance of the proposed scheme.
Due to the nonlinear structure of the scheme, an iterative algorithm is required. Inspired by \cite{CarliniCoscetti26,HadikhanlooSilva}, we employ a Discrete Learning Value Iteration algorithm (DLVI), or fictitious play method, applied to the discretized MFG problem. 
Given an initial guess $\ov m^{(1)} \in \CC([0,T];\P_1(\RR^d))$, for every iteration $n\geq 1$,  we compute $v^{(n)}$ by solving the scheme \eqref{eq:HJBdiscr} given $\mu=\ov m^{(n)}$. Next, we compute $m^{(n+1)}$ with the scheme \eqref{eq: Ce scheme} and we extend it defining $\widetilde  m^{(n+1)}$ as in \eqref{eq: mu tilde tk}-\eqref{eq: mu tilde t}. 
Finally we update the average $\ov m^{(n+1)}=\left(1-\lambda_n\right) \ov m^{(n)}+\lambda_n \widetilde m^{(n+1)}$ for $\lambda_n \in [0,1]$. 
We define the errors
\begin{equation*}
     E_{1}(m,\nu):=\text{Int}_{ \GG_{\Delta x}}|m_{N_T}-\nu_{N_T}| \qquad  \text{and}\qquad    E_{\infty}(v,w):=\max_{i\in\ZZ^d}|v_{0,i}-w_{0,i}|,  
\end{equation*}
 for every  $m,\nu,v,w \in \BB(\GG_{\Delta})$, where \text{Int}$_{ \GG_{\Delta x}}$ denotes the approximation of the integral on $ \RR^d$ by using the rectangle rule on $ \GG_{\Delta x}$.
Given a tolerance $tol>0$,
DLVI stops at  $n>1$ if  $ E_{1}(m^{(n+1)},m^{(n)}) <tol$ and $E_{\infty}(v^{(n)},v^{(n-1)}) <tol$, 
and the solution returned by  the algorithm is  $(m^{(n+1)},v^{(n)})$. In the following simulations,  we initialize
the iterations of DLVI with $ \ov m^{(1)}(t,x)=m^*_0(x)$
and we denote by $(m^{\Delta},v^{\Delta})$ the solution computed. The convergence of this algorithm for \eqref{eq:MFGd} is not addressed here and is left for future work. As already outlined in \cite{CarliniSilvaZorkot}, it is necessary to derive an implementable version of the scheme and to approximate the integrals appearing in \eqref{eq: Ce scheme}. To this end, we adopt the area-weighted Lagrange-Galerkin  approximation analyzed in \cite[Section 6.1]{CarliniSilvaZorkot}, while maintaining our definition of $\mathsf{\Phi}_{k}^{\Delta,\epsilon}[\mu]$ in \eqref{Phi_mu_eps_k}.
We approximate the solution of three different MFG systems.
The first problem does not satisfy all the required assumptions. Nevertheless, it serves as a useful benchmark beyond the theoretical framework and, since the exact solution is known, it enables us to assess the accuracy of the proposed scheme.
In this case, we focus on the one-dimensional setting and DLVI is implemented with $\lambda_n=1$ for any iteration $n\geq 1$. In the second and third problem,
which are one and two-dimensional respectively, DLVI is implemented  with $\lambda_n =\frac{1}{n+1}$. In all the proposed examples, the time integrals appearing in \eqref{eq:infS} and  \eqref{Phi_mu_eps_k} are evaluated exactly. In  all the  simulations, we choose $\rho (x)=\frac{1}{\sqrt{2\pi}}e^{-|x|^2/2}$,
$tol =10^{-4}$ and $\Delta t=(\Delta x)^{4/5}/4$. 
The parameter $\eps>0$ is specified separately for each example and is always chosen to be of order $\sqrt{\Delta t}$. These choices of $\Delta t$ and $\varepsilon$ are consistent with the assumptions of the convergence result in \Cref{thm: convergence scheme}.

\subsection{Example with analytical solution}\label{testexaxt}
We consider a MFG system with running cost independent of the space variable: 
\[
\ell(t,x,a)=\frac{1}{2}|a-\widehat{\alpha}(t)|^2,
\qquad
\widehat{\alpha}(t)=
\alpha_1\mathbb{I}_{[0,\overline t)}(t) +\alpha_2\mathbb{I}_{[\overline t,T]}(t),
\quad \alpha_1,\alpha_2\in\mathbb R^d,\, \overline t\in[0,T],
\]
   Hamiltonian $H(t,x,p)=\sup_{a \in \RR^d}\{a\cdot p  - \ell (t,x,a) \}=\frac{1}{2}|p|^2+\widehat{\alpha}(t)\cdot p$,   coupling cost $  f(t,x,m)=\frac{1}{2}|x-\int_{\mathbb{R}^d} y m(y)\dd y|^2$, 
   terminal cost $g(x,m)= 0$, and   initial measure $m^*_0$  given as  the density of a Gaussian  with mean $\mu_0\in\RR^d$ and covariance  $\Sigma_0\in \RR^{d\times d}$. For simplicity, we  assume that $\Sigma_0$ is a diagonal matrix.
This problem does not satisfy \onlyref{hyp: compact set,hyp: initial measure,hyp:terminal cost,hyp: monoticityassumption}. Specifically, the coupling $f$ is neither equibounded nor Lipschitz in space, and therefore does not satisfy \ref{hyp:terminal cost}. 
Moreover, the control set $A=\mathbb{R}$ is unbounded, thus violating \ref{hyp: compact set}.   The initial distribution $m_0^*$ does not have compact support,  violating \ref{hyp: initial measure}. 
Finally, since $f$ is linear in the mean of $m$, the strict monotonicity condition in \ref{hyp: monoticityassumption} does not hold in general. Nevertheless, this problem is of particular interest because it admits a unique explicit solution $(v^*,m^*)$ (see Appendix~\ref{appendix: exact solution}), which allows us to study the numerical error rates of the proposed scheme.

The value function $v^*$ is given by
\begin{equation*}
    v^*(t,x)=\frac{1}{2}\Pi(t)x\cdot x+s(t)\cdot x+c(t) 
    \qquad \text{for all } (t,x)\in [0,T]\times\mathbb{R}^d,
\end{equation*}
where,  setting $p(t)=\tanh(T-t)$ for all $t\in [0,T]$,  
\begin{equation*}
\overline{Y}(t)=
\begin{cases}
\mu_0-t\alpha_1, & t<\overline{t},\\
\mu_0-\overline{t}\alpha_1-(t-\overline{t})\alpha_2, & t\geq\overline{t}, 
\end{cases}
\end{equation*}
and denoting by $I_d\in\mathbb{R}^{d\times d}$ the identity matrix, we have that 
\begin{equation*}
\begin{aligned}
&\Pi(t)=p(t)I_d,\quad  s(t)=-p(t)\overline{Y}(t),\\
&c(t)=-\int_t^T\frac{1}{2}p^2(r)|\overline{Y}(r)|^2\,\dd r+\int_t^T p(r)\widehat{\alpha}(r)\cdot\overline{Y}(r)\,\dd r +\int_t^T\frac{1}{2}|\overline{Y}(r)|^2\,\dd r.
\end{aligned}
\end{equation*} 
The explicit formula for the density $m^*$ is given by
 \begin{equation*}
     m^*(t,x)=\prod_{i=1}^dm_i^*(t,x) \qquad \text{for all } (t,x)\in (0,T)\times \RR^d,
 \end{equation*}
where, for every $t\in [0,T]$ and $i=1,\dots,d $, the density $m_i^*(t,\cdot)$ is a univariate Gaussian with mean $(\ov Y(t))_i$ and covariance $(\Sigma(t))_{i,i}$, with $\Sigma(t)=\left(\frac{\cosh(T-t)}{\cosh(T)}\right)^2 \Sigma_0\in \RR^{d\times d}$.
 In the following simulations, we consider the problem  in the one-dimensional spatial setting and with data  $T=0.5$, $\ov t=0.25$, $\alpha_1=2.5$, $\alpha_2=-2.5$, $\mu_0=0.1$, and $\Sigma_0=0.5$. The scheme is implemented with $\eps=\sqrt{\Delta t}$    on the spatial domain $[-3.5,3.5]$. Boundary conditions are imposed exactly, while linear extrapolation is used   instead of linear interpolation in \eqref{eq:HJBdiscr}, whenever the characteristic trajectory $x_i - \Delta t a$ exits the computational domain.  
In \Cref{fig: exact space} and \Cref{fig: exact timespace}, we report the computed density for $\Delta x = 1.25 \cdot 10^{-2}$, shown only on the reduced domain $[-2.5,2.5]$ for better illustration of the dynamics. In \Cref{fig: exact space}, the left panel shows the initial density, while the right panel displays the evolution of $m^{\Delta}_{5\cdot k}$ for  $k=0,\dots,\lfloor N_T/5 \rfloor$. In \Cref{fig: exact timespace}, we show $m^{\Delta}$ on the space-time domain, with a 2D representation on the left and a 3D representation on the right. We observe that the mass initially moves to the left and then changes direction to the right after time $\overline t$, with, moreover, a concentration effect.
 
  \begin{figure}[H]
        \centering 
            \includegraphics[width=4.5cm]{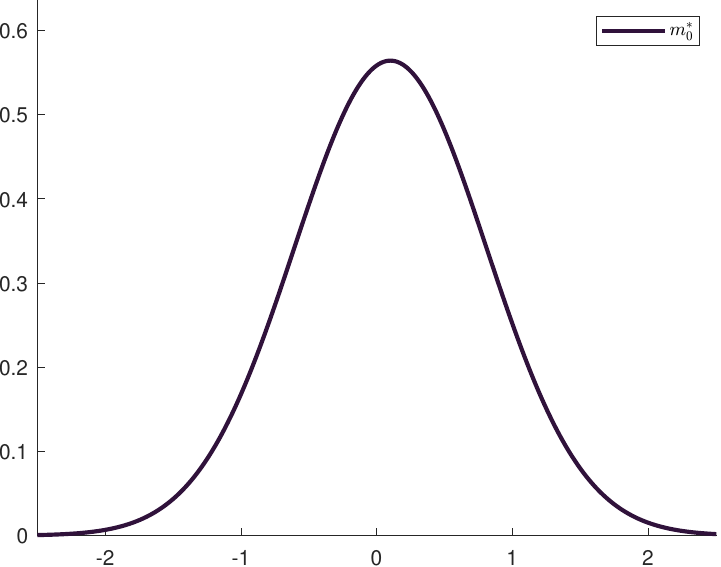}       
        \hspace{1cm}
            \includegraphics[width=4.5cm]{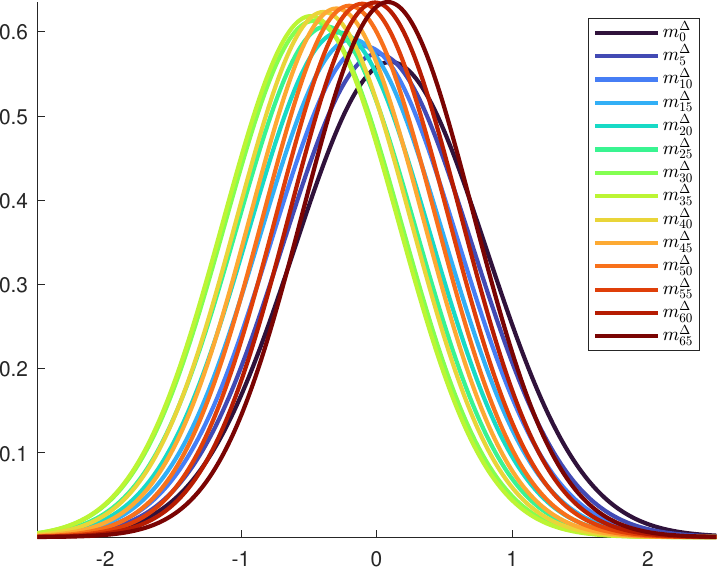}
        \caption{Example  \ref{testexaxt}. Initial density (left) and  time-evolution of the density (right).} \label{fig: exact space}
\end{figure}

  \begin{figure}[H]
        \centering 
            \includegraphics[width=5.5cm]{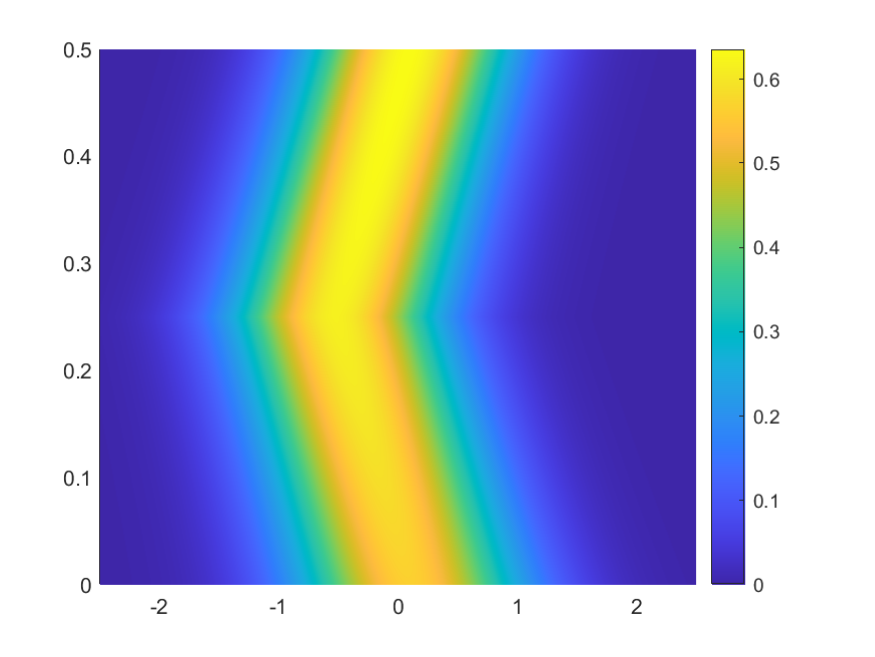}   
        \hspace{1cm}
            \includegraphics[width=5.5cm]{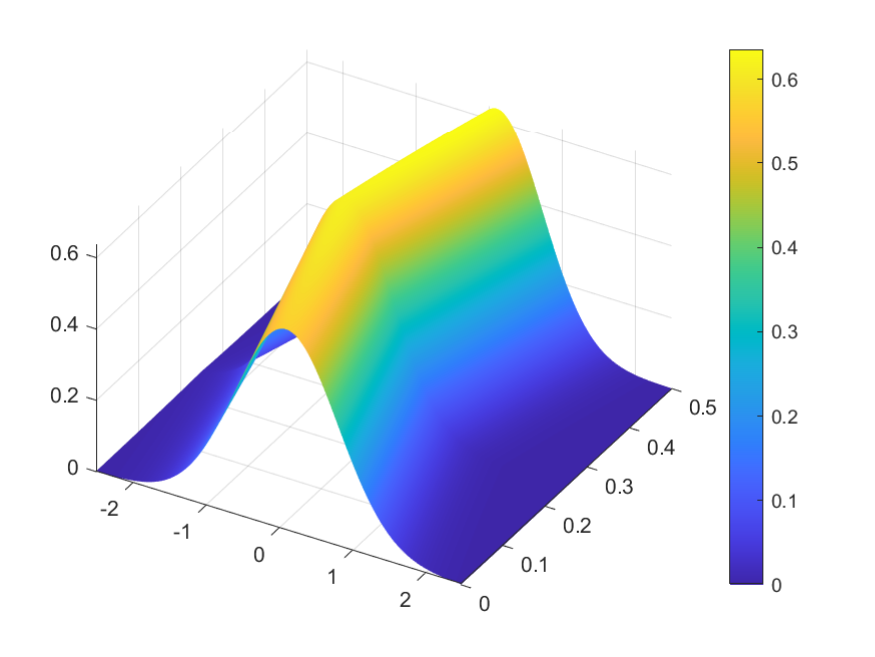}
        \caption{Example  \ref{testexaxt}. Density in the space-time domain in 2D view (left) and 3D view (right).} \label{fig: exact timespace}
\end{figure}

We investigate the behavior of the scheme according to whether or not the switching time $\overline{t}$ coincides with a time-grid node. Since the time-grid consists of $N_T+1$ nodes, $\overline t$ is a grid node whenever $N_T$ is even, whereas it lies strictly between two consecutive time nodes whenever $N_T$ is odd.
In the following simulations,  for each  refinement of $\Delta x$ (with $\Delta t=(\Delta x)^{4/5}/4$),  we consider two cases, corresponding to  $N_T$ even and odd, respectively. 
 When $\Delta t$ is prescribed, $T$ is not necessarily an integer multiple of $\Delta t$, and in our implementation  we construct a uniform time-grid on $[0,T]$ by setting $N_T=\lfloor T/\Delta t \rfloor$ or $N_T=\lceil T/\Delta t \rceil$, in order to guarantee that $N_T$ is even or odd,  and then redefining $\Delta t = T/N_T$.  

We  report in \Cref{table: test with analytical solution}, the errors $E_{1}(m^{\Delta},\hat m^*)$ and $E_{\infty}(v^{\Delta},\hat v^*)$, distinguishing between the even case (columns 2–3) and the odd case (columns 4–5). The errors for the value function are noticeably larger when $N_T$ is odd (column 5) than when $N_T$ is even (column 3), while both cases still exhibit a decreasing trend with an observed convergence rate of approximately $0.8$. To better understand the difference between these two cases, we report in~\Cref{fig: errU} the plot of $|v^*-v^{\Delta}|$, still computed with $\Delta x=1.25\cdot 10^{-2}$, shown for even $N_T$ (left) and odd $N_T$ (right), and  we  define 
\begin{equation*}
\ell_1(a)\coloneq \frac{1}{2}|a-\alpha_1|^2,
\qquad
\ell_2(a)\coloneqq\frac{1}{2}|a-\alpha_2|^2.  
\end{equation*}
For every $(t_k ,x_i)\in \GG_{\Delta }^*$, the exact optimal control $\alpha\colon [0,T]\to \RR^d$, minimizing the cost appearing in DPP over the time interval $[t_k,t_{k+1}]$ at $x_i$,  is approximated in the scheme by freezing the control to a single value $a  \in \RR^d$, which minimizes the discrete cost  in  \eqref{eq:HJBdiscr}. This leads, in particular, to the approximation
\begin{equation}\label{eq: integral approx HJB}
    \int_{t_k}^{t_{k+1}} \ell(s, \alpha(s))\,\dd s\simeq \int_{t_k}^{t_{k+1}} \ell(s,a)\,\dd s,
\end{equation}
where, since $\ell$ does not depend on $x$, this variable has been omitted in the notation. When the switching time $\overline t$ coincides with a grid point, each interval $[t_k,t_{k+1}]$ lies entirely within one of the two regimes $t \leq \overline t$ or $t \geq \overline t$, so that both 
terms in \eqref{eq: integral approx HJB} depend only on either $\ell_1$ or $\ell_2$. 
In contrast, when  the switching time $\overline t$ lies inside an interval $(t_k,t_{k+1})$, \eqref{eq: integral approx HJB} becomes 
\begin{equation*}
    \begin{aligned}
        \int_{t_k}^{\overline t} \ell_1(\alpha(s))\dd s
+
\int_{\overline t}^{t_{k+1}} \ell_2(\alpha(s))\dd s\simeq \int_{t_k}^{\overline t} \ell_1(a)\dd s
+
\int_{\overline t}^{t_{k+1}} \ell_2(a)\dd s.
    \end{aligned}
\end{equation*}
As a consequence, the scheme is forced to select a single value $a \in \mathbb{R}^d$ over the whole interval $[t_k,t_{k+1}]$, which must account for the two distinct regimes $\ell_1$ and $\ell_2$, whereas in the continuous setting the optimal control $\alpha(s)$ depends instantaneously on time and may switch at $\overline t$.
In accordance with this observation, an error appears at time $\overline t$ in the case of odd $N_T$, as shown in \Cref{fig: errU} (right). This error propagates backward in time due to the backward-in-time structure of the scheme, without accumulating over time steps, since it is generated only in the single interval containing $\overline t$.
When $N_T$ is even, no significant error arises at $\overline t$ from this approximation, see \Cref{fig: errU} (left). In this case, the numerical error is mainly concentrated near the boundary due to extrapolation effects and increases when moving backward in time, in agreement with the backward-in-time structure of the scheme.

Despite the difference in the approximation of the value function, \Cref{table: test with analytical solution} shows that the errors for the density are comparable in the even and odd cases (columns 2 and 4) and exhibit an observed convergence rate of approximately $1.35$.
This is not in contradiction with the previous observations, since the approximate density $m^{\Delta}$ is computed by solving the continuity equation driven by the discrete flow associated with $D(v^{\Delta}(t,\cdot) * \rho_{\varepsilon})$. In both cases, this quantity remains essentially the same in the region where the dynamics actually take place. Indeed, if  we compute the maximum difference over the space–time domain, for instance $[0,0.5]\times[-2.5,2.5]$, where the dynamics effectively occurs, the resulting error is $6.76 \cdot 10^{-5}$, which  explains why no discrepancy is observed at the level of the computed densities.

\begin{table}[H]
  
        \renewcommand\arraystretch{1.1}
        \centering
        \begin{tabular}{lllll}
                    \toprule
               $\Delta x$      & $ E_1(m^\Delta,\hat m^*)$  & $E_\infty(v^\Delta,\hat v^*)$  & $ E_1(m^\Delta,\hat m^*)$  & $E_\infty(v^\Delta,\hat v^*)$ \\ \midrule 
                $1.00 \cdot 10^{-1}$&    $1.19 \cdot 10^{-2}$  &   $2.66 \cdot 10^{-2}$ & $1.21 \cdot 10^{-2}$  &  $1.45 \cdot 10^{-1}$ \\ 
                  $5.00 \cdot 10^{-2}$  &  $3.58 \cdot 10^{-3}$  &  $1.46 \cdot 10^{-2}$   &  $3.63 \cdot 10^{-3}$  &  $8.95 \cdot 10^{-2}$   \\ 
                $2.50 \cdot 10^{-2}$  &   $2.72 \cdot 10^{-3}$ &  $8.19 \cdot 10^{-3}$  &  $2.46 \cdot 10^{-3}$  &  $4.81 \cdot 10^{-2}$  \\ 
                $1.25 \cdot 10^{-2}$  &  $1.13 \cdot 10^{-3}$  & $4.81 \cdot 10^{-3}$   &  $1.15 \cdot 10^{-3}$  &  $2.80 \cdot 10^{-2}$ \\ 
                $6.25\cdot 10^{-3}$  & $2.77 \cdot 10^{-4}$ &  $2.75 \cdot 10^{-3}$  &  $2.99 \cdot 10^{-4}$  &  $1.64 \cdot 10^{-2}$        \\  \bottomrule
        \end{tabular}
\caption{Example  \ref{testexaxt}.  space-grid steps (column 1), errors for even values of $N_T$ (columns 2-3) and errors for odd values of $N_T$ (columns 4-5).}
        \label{table: test with analytical solution}
    \end{table}

  \begin{figure}[H]
        \centering 
            \includegraphics[width=5.5cm]{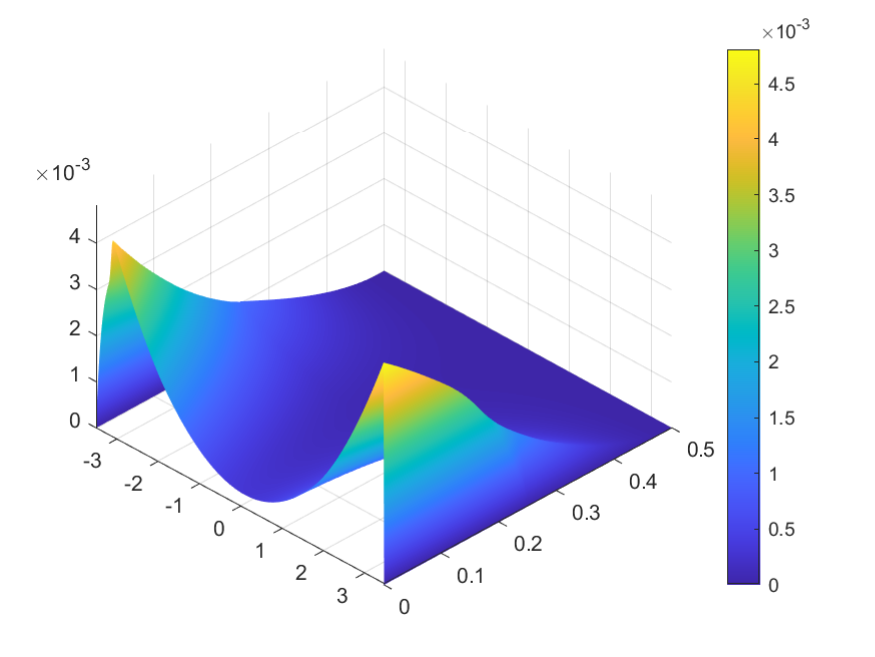} 
        \hspace{1cm}
            \includegraphics[width=5.5cm]{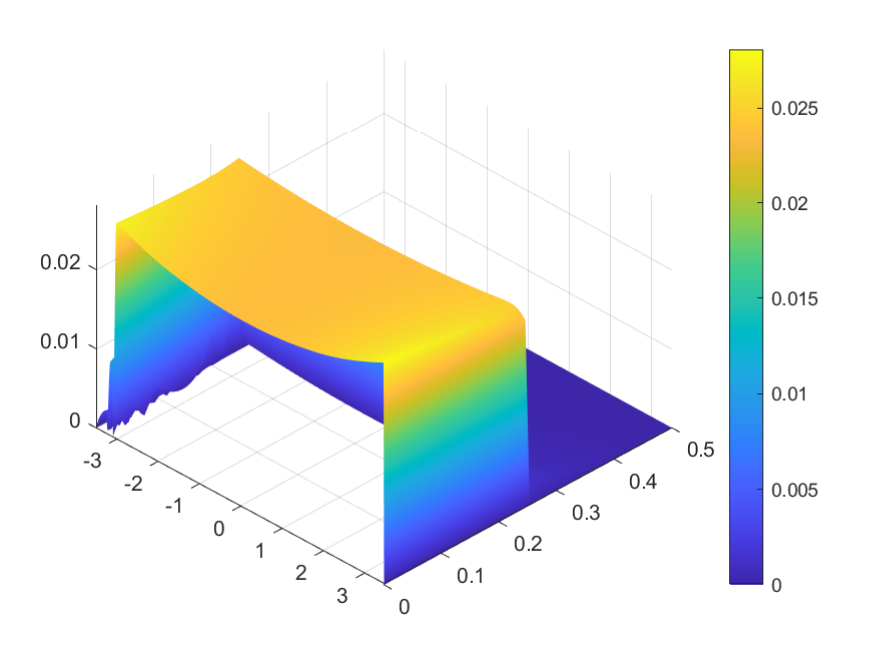}
        \caption{Example \ref{testexaxt}. Pointwise error $|v^*-v^{\Delta}|$ computed with $\Delta x=1.25\cdot 10^{-2}$  for even $N_T$ (left) and odd $N_T$ (right). } \label{fig: errU}
\end{figure}

\subsection{Examples with a target position and aversion to crowded regions}
We consider two MFG systems. The first one is set in dimension $d=1$ with $\ell$ being independent of $(t,x)$ (and hence we omit these variables in the notation below for $\ell$), while the second one is set in dimension $d=2$ with a time-state dependent running cost. In both cases, the coupling cost is given by 
\begin{equation*}
         f(t,x,m)=   \lambda \big( \mathbb{I}_{[0,\ov t)}(t)\min\{|x-\ov x|^2,R\}+  \mathbb{I}_{[\ov t,T]}(t)\min\{|x-\widetilde x|^2,R\}\big) +(r_{\sigma}*m)(x), \qquad  
\end{equation*}
where $\ov t\in [0,T]$, $ \lambda  ,R,\sigma > 0$, $\ov x,\widetilde x \in \RR^d$, $r_\sigma(x)=e^{-|x|^2/(2\sigma^2)}/\sqrt{2\pi \sigma}$. This MFG is inspired by similar test cases in \cite{CarliniSilvaZorkot,Lauriere21}, where the data are time-independent. Here, we extend this framework by considering time-dependent data with a discontinuity at $t=\ov t$. 
More precisely, in our example, the agents aim to reach the target point $\overline{x}$ during the time interval $[0,\overline{t})$, and then the target point $\widetilde{x}$ during $[\overline{t},T]$. These objectives are encoded by the first two components of the coupling cost $f$, while the third component penalizes the presence of other agents, thereby promoting congestion avoidance. The balance between these competing effects depends on the value of $\lambda$, as illustrated in the following simulations.

\subsubsection{Case $d=1$ with quadratic running cost}\label{test1d}
In this example, we consider the above problem in dimension $d=1$, with  running cost  independent of time and space $\ell(a)=\frac{1}{2}a^2$, control set  $A=[r_L,r_R]$ with $r_R>r_L$, terminal cost $g\equiv 0$, and  initial density
\begin{equation*} 
m^*_{0}(x)= \frac{\widetilde m(x)}{\int_{\RR}\widetilde  m(x)\dd x},\qquad \text{where} \qquad \widetilde m(x)=(1-|x|^2)^3\mathbb{I}_{\ov B(0,1)}(x).
\end{equation*} 
The Hamiltonian and its derivative with respect to $p$ are given by
\begin{equation*}
    H(p)=\begin{cases}
        r_Lp-\frac{1}{2}r_L^2 \quad  &\text{if } p<r_L, \\ 
        \frac{1}{2}p^2 \quad  &\text{if } r_L\leq p \leq r_R,\\ 
        r_Rp-\frac{1}{2}r_R^2 \quad  &\text{if } p>r_R, \\ 
    \end{cases}\qquad 
     D_pH(p)=\begin{cases}
        r_L\quad  &\text{if } p<r_L, \\ 
p \quad  &\text{if } r_L\leq p \leq r_R,\\ 
        r_R \quad  &\text{if } p>r_R. \\ 
    \end{cases}
\end{equation*}
This data satisfy \onlyref{hyp: compact set,hyp: initial measure,hyp: lagrangian new,hyp:terminal cost,hyp: monoticityassumption}.
We set $r_L=-2$,  $r_R=2$, $T=4$, $\overline t=2$, $R=36$, $\sigma=0.5$, $\overline x=1$, and $\widetilde x=-1$. In the following simulations, we test the scheme with $\Delta x=1.25\cdot10^{-2}$ and $\varepsilon=\sqrt{\Delta t}$ on the spatial computational domain $[-3,3]$, while the approximated densities are displayed on the restricted domain $[-2,2]$ to focus on the region where the density is effectively supported. We apply the scheme to solve the problem for $\lambda=2$ and $\lambda=\frac{1}{2}$. In \Cref{fig: m0 test1d}, the left panel shows the initial density $m_0^*$ (black line with circles) together with the  intermediate time approximated densities for $\lambda =2$ (red line with squares) and for $\lambda=\frac{1}{2}$ (blue line with asterisks). The same color and marker convention is preserved in the right panel, which displays the initial density $m_0^*$ together with the terminal time approximated densities $m^{\Delta}_{N_T}$ for the two values of $\lambda$. In \Cref{fig: test1d}, we report the space--time evolution of $m^{\Delta}$, with the left panel corresponding to $\lambda=2$ and the right panel corresponding to $\lambda=\frac{1}{2}$. 

In the case $\lambda=2$, \Cref{fig: m0 test1d} and \Cref{fig: test1d} (left) show that the coupling cost initially favors proximity to the target $\overline{x}$ during the time interval $[0,\overline{t}]$, and subsequently to the  target $\widetilde{x}$ over $[\overline{t},T]$.
During the first phase, the population moves coherently toward $\overline{x}$, progressively reducing its distance from the target. However, the agents do not reach $\overline{x}$ exactly, as a result of the trade-off between the running cost, which penalizes large controls, and the coupling cost. Indeed, reaching $\overline{x}$ exactly and then moving toward $\widetilde{x}$ would require a larger control effort. Since the agents forecast the  evolution of the game, they anticipate the switch of the target and begin deviating toward 
$\widetilde x$ before fully reaching $\ov x$.
During the second phase the agents converge to the final target $\widetilde{x}$ and eventually reach it. Once there, it is optimal to remain close to $\widetilde{x}$, since both the coupling cost and the running cost can be minimized by choosing a stationary dynamics.

In the case $\lambda=\frac{1}{2}$, as illustrated in \Cref{fig: m0 test1d} and \Cref{fig: test1d} (right),  the qualitative behavior with respect to the two target points remains the same. During the first time interval, the agents move toward the  target $\overline{x}$, while during the second they head toward the  target $\widetilde{x}$. However, in this case the aversion to crowded regions prevails over the attraction to the targets. As a result, the population does not concentrate during its evolution. Instead, it splits into two groups in order to reduce congestion while still progressing toward the prescribed targets.

  \begin{figure}[H]
        \centering 
            \includegraphics[width=4.5cm]{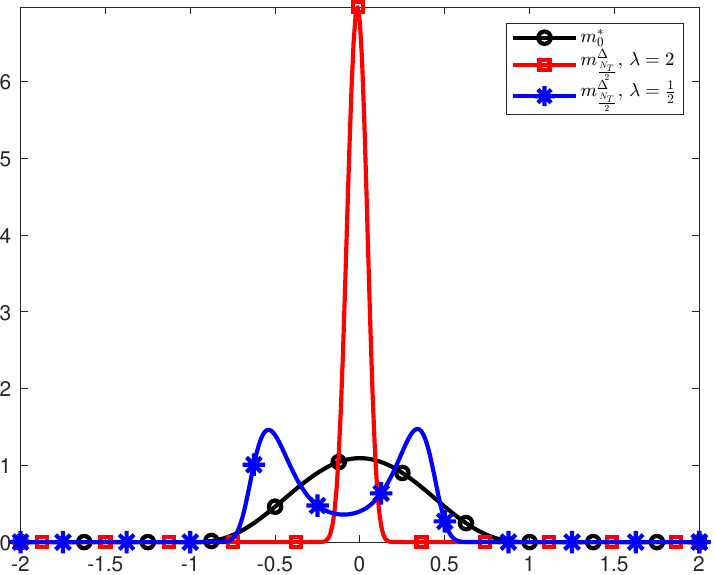} 
              \hspace{1cm}
        \includegraphics[width=4.5cm]{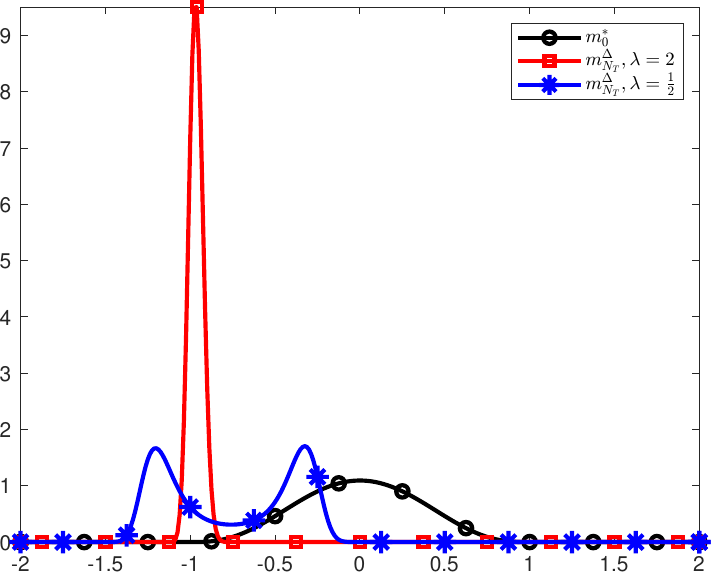} 
        \caption{Example  \ref{test1d} with  $A=[-2,2]$. Initial density $m_0^*$ and intermediate time approximations $m^{\Delta}_{N_T/2}$ for $\lambda=2$ and $\lambda=\frac{1}{2}$ (left). Initial density $m_0^*$ and terminal time approximations $m^{\Delta}_{N_T}$ for $\lambda=2$ and $\lambda=\frac{1}{2}$ (right). } \label{fig: m0 test1d}
\end{figure}

  \begin{figure}[H]
        \centering 
            \includegraphics[width=5.5cm]{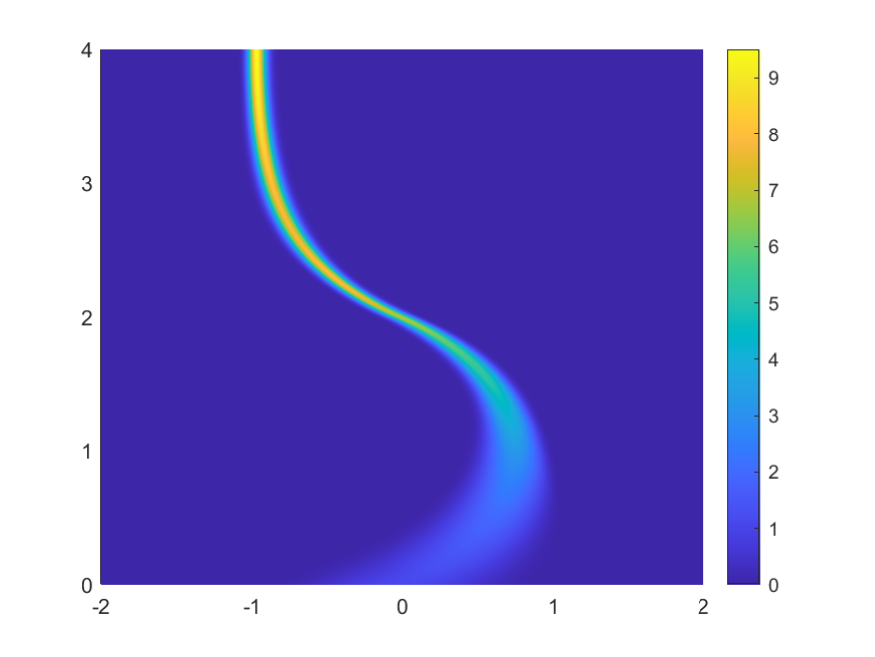}   
            \includegraphics[width=5.5cm]{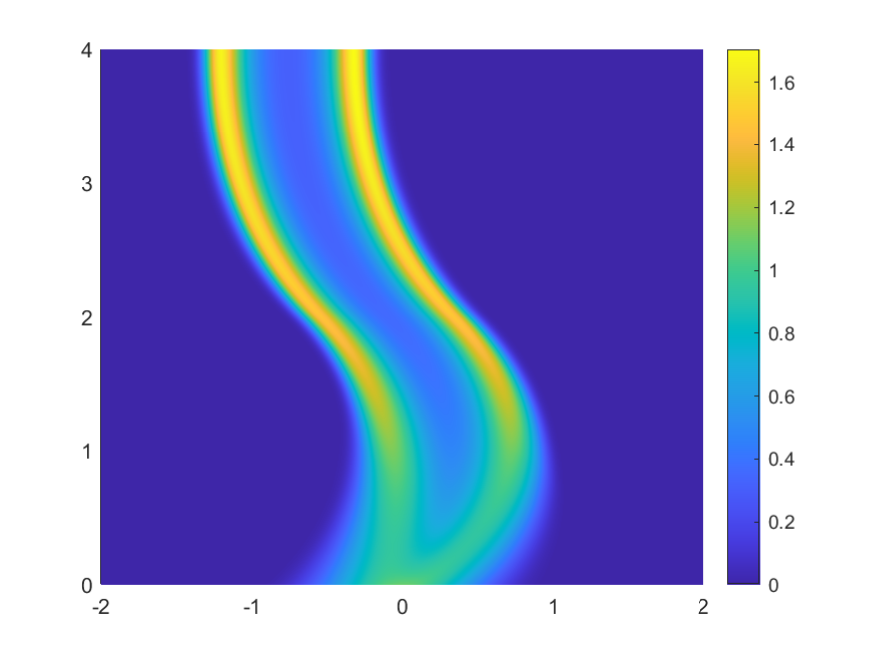}
        \caption{Example  \ref{test1d} with  $A=[-2,2]$. Density in the space-time domain  for $\lambda=2$ (left) and $\lambda=\frac{1}{2}$ (right).  } \label{fig: test1d}
\end{figure}

We now test the scheme using the same parameters as in the previous simulations, while modifying the control set $A$. More precisely, we first consider $A=[-2,0]$, corresponding to $r_L=-2$ and $r_R=0$, and then $A=[0,2]$, corresponding to $r_L=0$ and $r_R=2$.

In the case $A=[-2,0]$, \Cref{fig: test1d evolution neg} reports the evolution of the density for $\lambda=2$ (left) and $\lambda=\frac{1}{2}$ (right).  When $\lambda=2$, the agents initially located on the left side of the origin move toward the center and then remain there until the final time, whereas the agents initially located on the right side of the origin remain essentially stationary throughout the whole time interval.
Unlike the previous simulations, the agents no longer move first toward $\overline{x}$ and then toward $\widetilde{x}$. To better understand this behavior, we also display in  \Cref{fig: test1d drift neg} the discrete field
 $D_pH\left(D_x(\rho_{\varepsilon}*v^{\Delta}(t,\cdot))(x)\right),$
for $\lambda=2$ (left) and $\lambda=\frac{1}{2}$ (right), over the space--time domain $[-2,2]\times [0,T]$.
This field determines, up to a change of sign, the discrete trajectories \eqref{Phi_mu_eps_k}.
The discrete field  initially takes negative values in the negative part of the domain and for $t\geq \ov t$ becomes identically zero in the region where the dynamics is effectively supported, namely $[-1,1]$.
The corresponding dynamic is in agreement with the game. Indeed, agents, which are initially distributed around the origin,
can move  toward the first target $\overline{x}$  by exploiting positive drifts, corresponding to negative values of the displayed field. However, reaching the second target $\widetilde{x}$ from $\overline{x}$ would require negative drifts, corresponding to positive values of the field, which are not allowed by the choice of the control set $A$. 
Therefore, optimizing the first component of the coupling cost over the interval $[0,\overline{t}]$ would  move the agents too far from $\widetilde{x}$, resulting in a significantly larger coupling cost during the second time interval $[\overline{t},T]$. 
As a consequence, the optimal strategy is to remain close to the center, balancing the two competing objectives.
When $\lambda=\frac{1}{2}$, the dynamics is analogous to the case $\lambda=2$. 
However, since the aversion to crowded regions becomes the dominant cost, the field remains partially negative also in the positive part of the spatial domain, causing the population to split into two groups.

  \begin{figure}[H]
        \centering 
            \includegraphics[width=5.5cm]{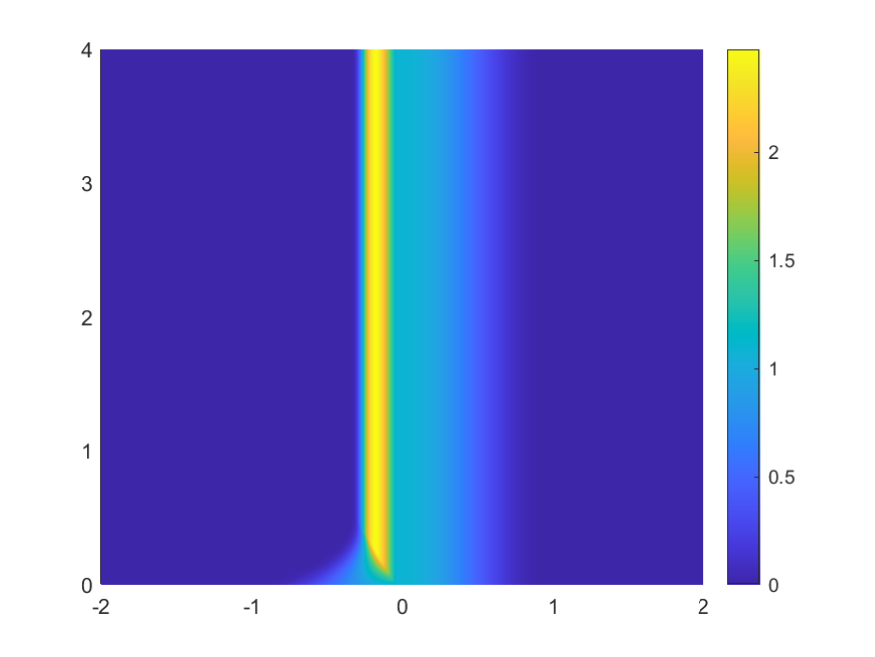}   
            \includegraphics[width=5.5cm]{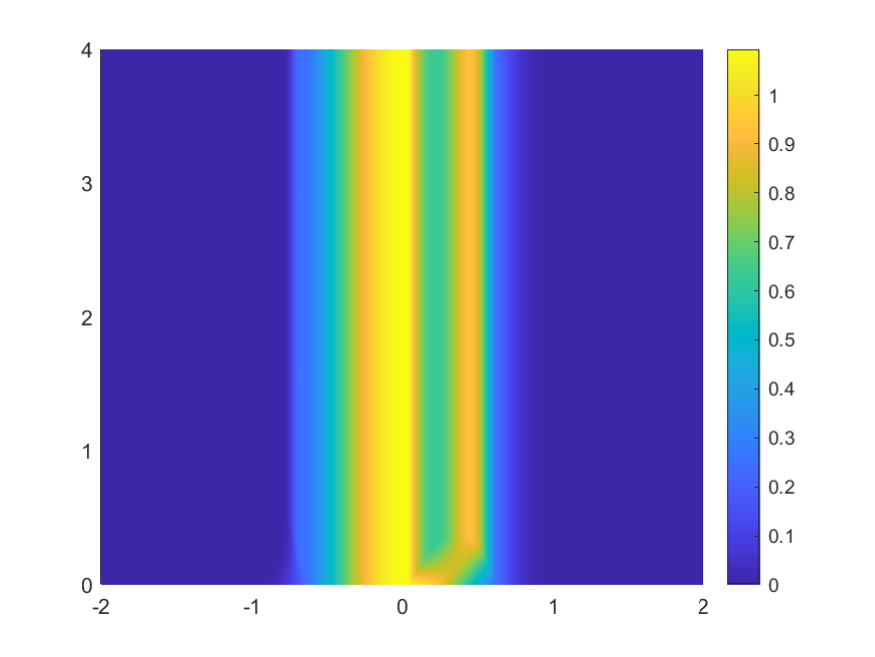} 
        \caption{Example  \ref{test1d} with  $A=[-2,0]$. Density in the space-time domain for $\lambda=2$ (left) and $\lambda=\frac{1}{2}$ (right).} \label{fig: test1d evolution neg}
\end{figure}

  \begin{figure}[H]
        \centering 
            \includegraphics[width=5.5cm]{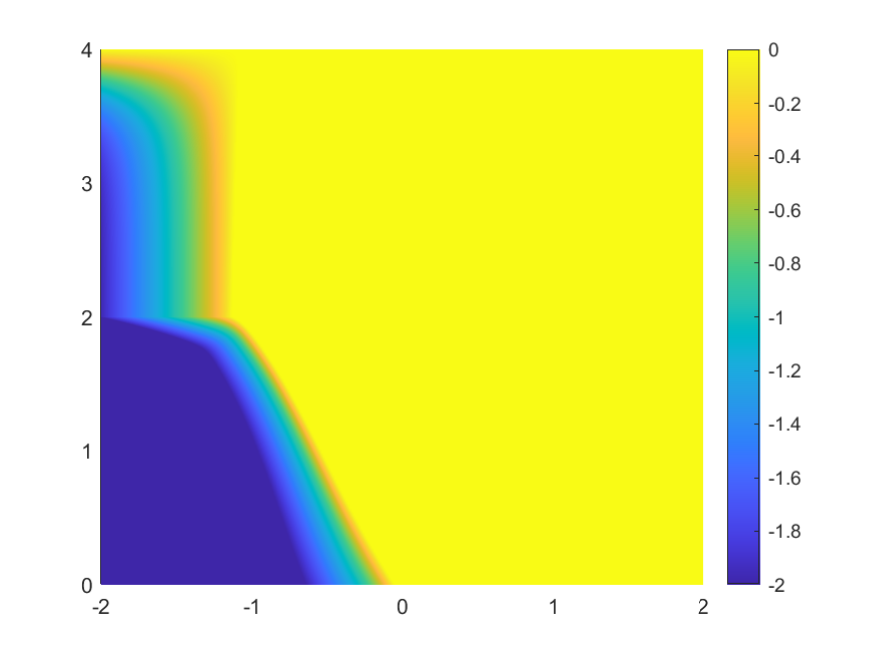}   
            \includegraphics[width=5.5cm]{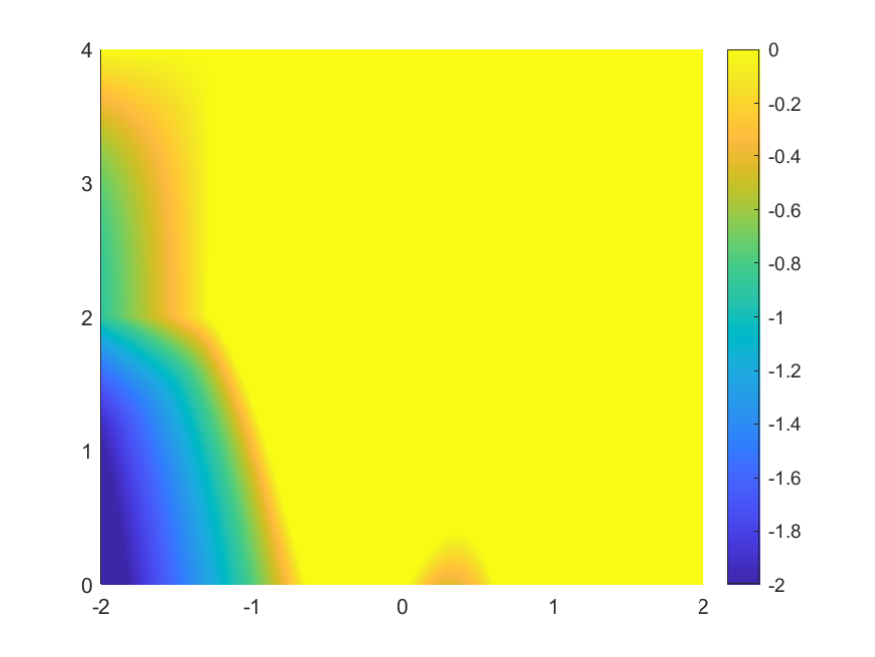} 
        \caption{Example  \ref{test1d} with  $A=[-2,0]$.  Field $
D_pH\left(D_x(\rho_{\varepsilon}*v^{\Delta}(t,\cdot))(x)\right)
$ in the space-time domain for $\lambda=2$ (left) and $\lambda=\frac{1}{2}$ (right).} \label{fig: test1d drift neg}
\end{figure}

In the case $A=[0,2]$, the evolution of the density is reported in \Cref{fig: test1d evolution pos}, while the corresponding discrete field
$D_pH\left(D_x(\rho_{\varepsilon}*v^{\Delta}(t,\cdot))(x)\right)$
is shown in \Cref{fig: test1d drift pos}. In both figures, the left panel refers to the case $\lambda=2$, whereas the right panel refers to the case $\lambda=\frac{1}{2}$.
In this setting, the agents cannot move toward $\overline{x}$, since negative controls are not admissible. Therefore, the optimal strategy during the first time interval $[0,\overline{t}]$ is to remain stationary in order to minimize the cost. During the second interval $[\overline{t},T]$, the agents move toward the second target $\widetilde{x}$, minimizing the coupling cost, since negative drifts, corresponding to positive values of the field,  allow them to move toward the left. As in the previous case, when $\lambda=\frac{1}{2}$, the aversion to crowded regions becomes dominant and, around the switching time $\overline{t}$, when the agents approach the point $x=0$, the population splits into two groups.

  \begin{figure}[H]
        \centering 
            \includegraphics[width=5.5cm]{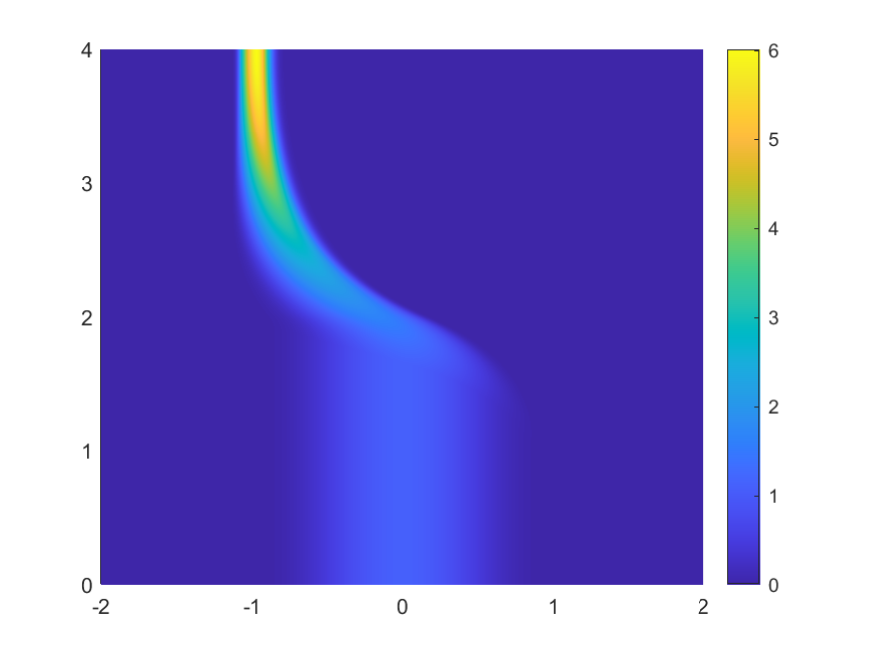}   
            \includegraphics[width=5.5cm]{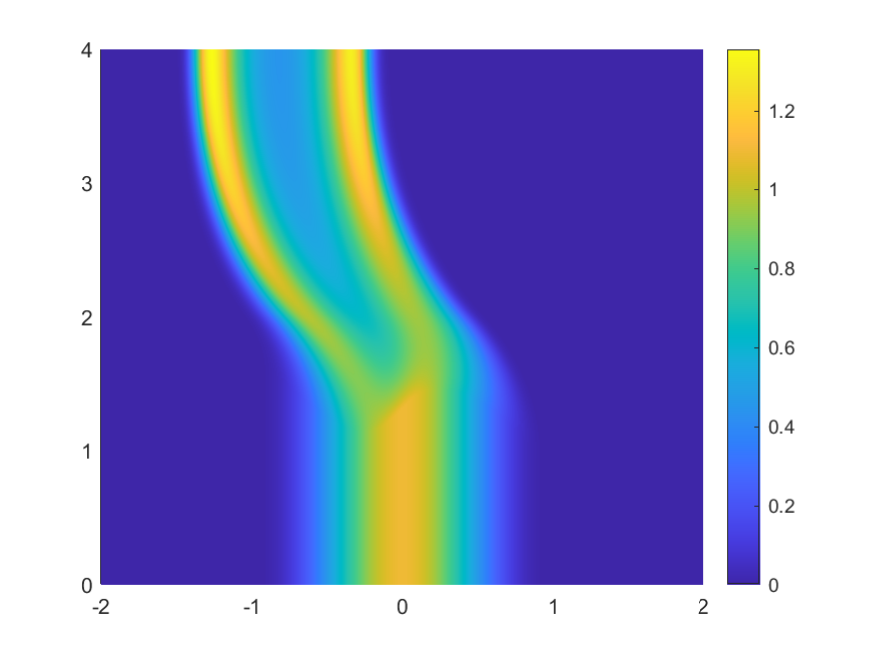} 
        \caption{Example  \ref{test1d} with  $A=[0,2]$. Density in the space-time domain for $\lambda=2$ (left) and $\lambda=\frac{1}{2}$ (right).} \label{fig: test1d evolution pos}
\end{figure}

  \begin{figure}[H]
        \centering 
            \includegraphics[width=5.5cm]{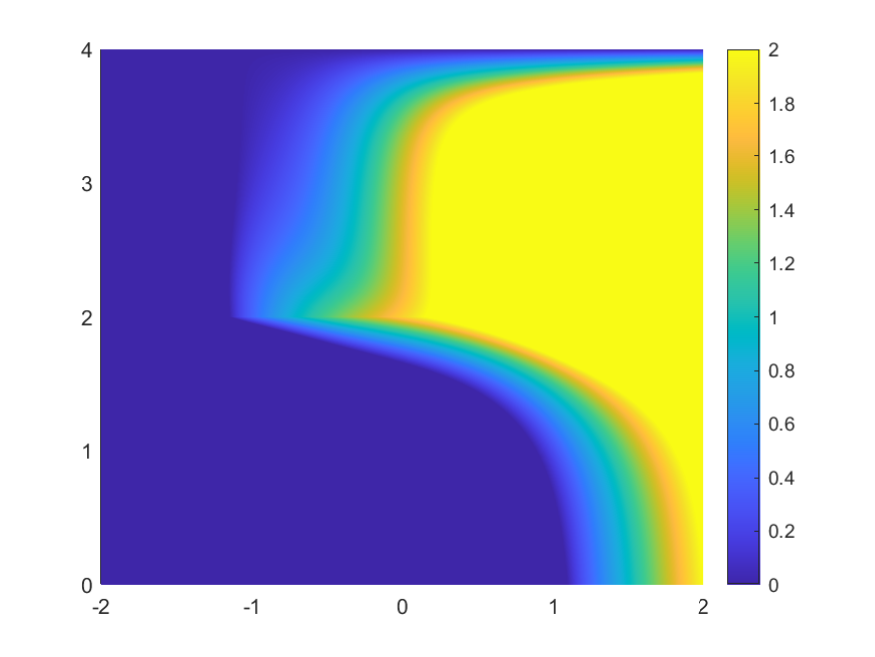}   
            \includegraphics[width=5.5cm]{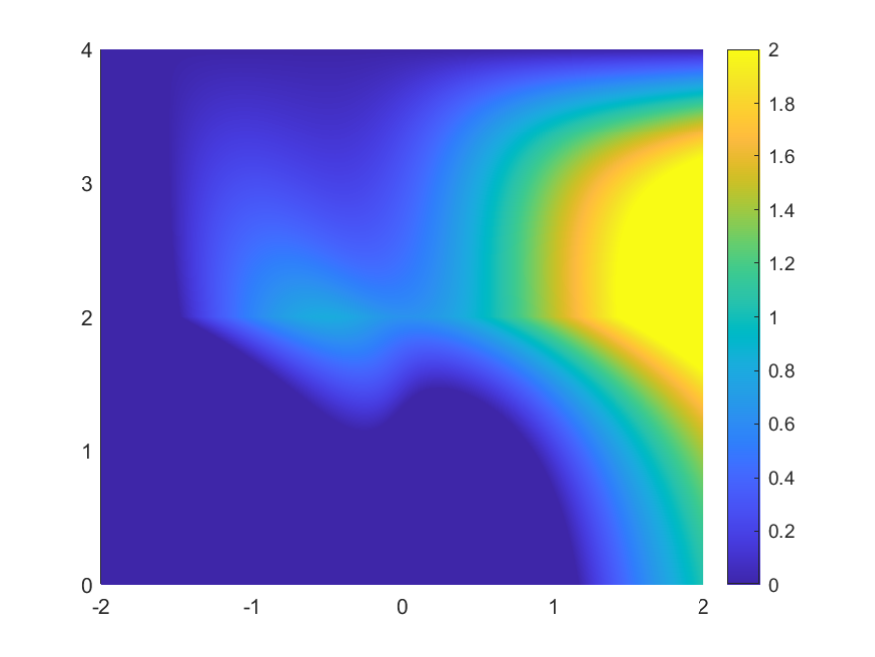} 
        \caption{Example  \ref{test1d} with  $A=[0,2]$. Field $
D_pH\left(D_x(\rho_{\varepsilon}*v^{\Delta}(t,\cdot))(x)\right)
$ in the space-time domain for $\lambda=2$ (left) and $\lambda=\frac{1}{2}$ (right).} \label{fig: test1d drift pos}
\end{figure}

\subsubsection{Case $d = 2$ with time-state dependent quadratic running cost}\label{test2d}
In this example, we consider the above problem in dimension $d=2$, with  running cost
\begin{equation*}
     \ell(t,x,a)=\frac{1}{2}\left|a-\mathbb{I}_{[0,\ov t)}(t)b(x)+\mathbb{I}_{[\ov t,T]}(t)b(x)\right|^2,
\end{equation*}
given the rotating velocity field
\begin{equation*}
     b(x)= (-x_2,x_1) \qquad \text{for all }x=(x_1,x_2)\in \RR^2,
\end{equation*}
terminal cost  $g(x,m)= \lambda  \min\{|x-\widetilde x|^2,R\}$, initial density
\begin{equation*} 
m^*_{0}(x)= \frac{\widetilde m(x)}{\int_{\RR^2}\widetilde m(x)\dd x}, \quad \text{where }\quad   \widetilde m(x)=(1-4|x|^2)^3\mathbb{I}_{\ov B(0,\frac{1}{2})}(x),
\end{equation*} 
and control set  $A=[-2,2]^2$. This data satisfy \onlyref{hyp: compact set,hyp: initial measure,hyp:terminal cost,hyp: monoticityassumption}, while $\ell$ does not satisfy all the required assumptions in \onlyref{hyp: lagrangian new}.
However, the drift $b$ could be regularized by multiplying the linear term by a suitable mollifier, chosen so that the running cost  remains unchanged within the selected spatial domain, and the computations carried out with the regularized data  yield the same results as those obtained with the unregularized one.
We set $T=4$, $\overline t=2$, $R=72$, $\sigma=0.5$, $\ov x=(-2,0)$ and $\widetilde x=(2,0)$. In the following simulations, we test the scheme with $\Delta x=5.00\cdot10^{-2}$ and $\varepsilon=4\sqrt{\Delta t}$ on the spatial  domain $[-3,3]^2$. We consider two values of the parameter $\lambda$, namely $\lambda=1$ and $\lambda=\frac{1}{2}$. 

In the case  $\lambda=1$, as shown in \Cref{fig: test2d} (left), the population initially starts concentrated around the origin $(0,0)$ and during the first time interval $[0,\overline{t}]$, the agents are attracted toward the  target $\overline{x}$, while the presence of the rotational field $b(x)$ induces a clockwise rotation of the dynamic. 
Similarly to the one-dimensional case, the agents do not reach the first target exactly, since their ability to anticipate the future evolution of the game allows them to change the direction of motion toward the second target before reaching the first one. 
During the second time interval $[\overline{t},T]$, the agents move toward the  target $\widetilde{x}$ with a counterclockwise rotation. 
Near the second target $\widetilde{x}$, a slight overshoot in the vertical direction can also be observed at the terminal time. This phenomenon results from the competition between the coupling cost, which tends to minimize the distance from the target, and 
the rotational field $b(x)$, which induces an upward drift.

For $\lambda=\frac{1}{2}$, as illustrated in \Cref{fig: test2d} (right), the qualitative structure of the dynamics remains unchanged. However, the population is less concentrated than in the case $\lambda=1$, as reflected by the color scales in the figures. This is because the third component of the coupling cost becomes dominant, thereby strengthening the repulsive interaction among the agents.

 \begin{figure}[H]
        \centering 
            \includegraphics[width=5.5cm]{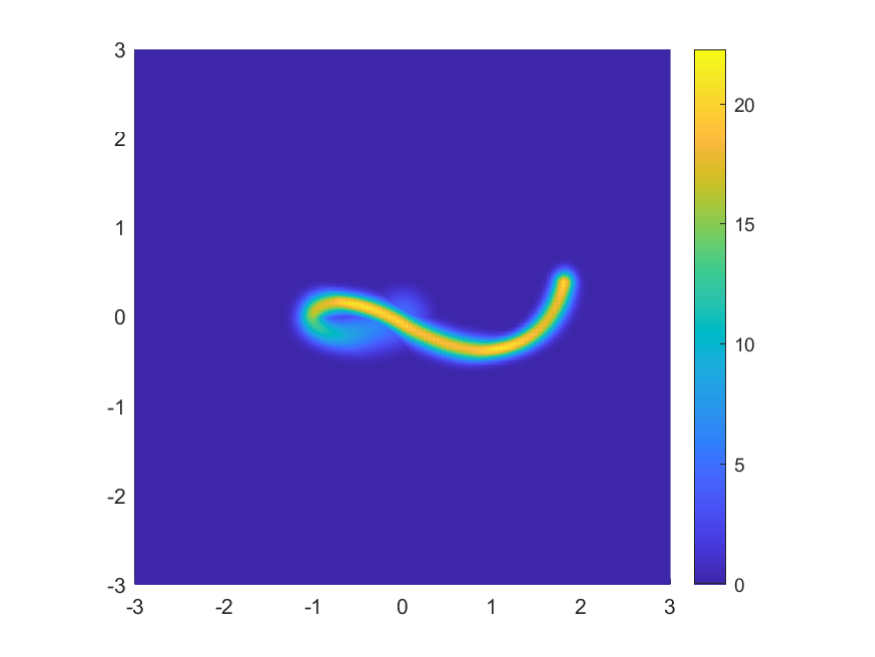}   
            \includegraphics[width=5.5cm]{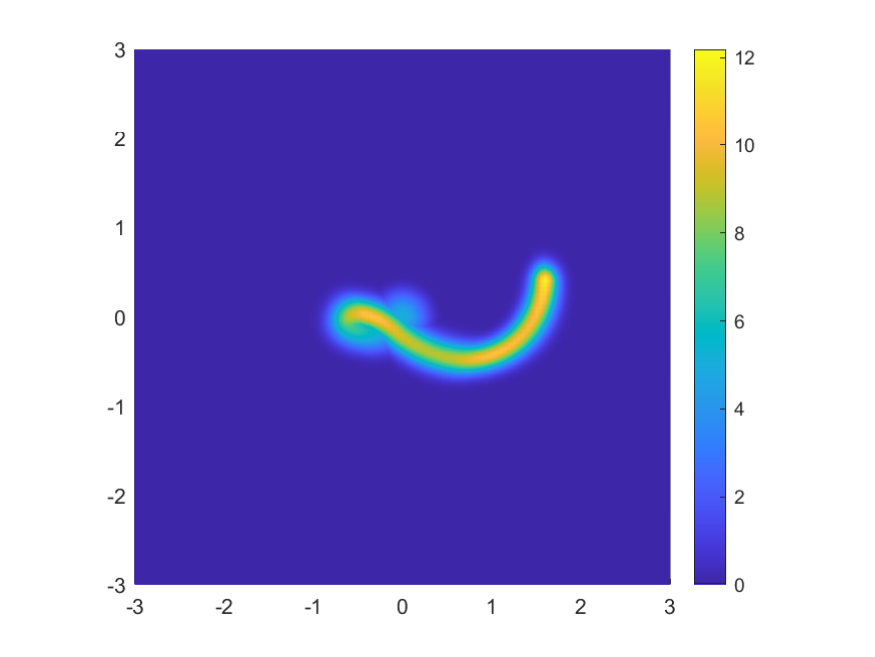}
        \caption{Example  \ref{test2d}. Time-evolution (2D view) of the density, $\lambda =1$ (left) and $\lambda =\frac{1}{2}$ (right).} \label{fig: test2d}
\end{figure}

\appendix
\section{Appendix}
This appendix collects the matrix results and properties of  Clarke's generalized Jacobian used in the proofs of \cref{prop: stability CE}.

\subsection{Matrix properties}
We first establish a first-order expansion of the determinant around the identity. 

\begin{lem}\label{lem: trace} 
 Let $A\in L^1([0,T]; \RR^{d \times d})$. Then, for every $t_1,t_2\in [0,T]$ with $t_2\geq t_1$ and  $(t_2-t_1)$ sufficiently  small, 
            \begin{equation*}\label{eq: det inver}
             \left | \det \left (I+\int_{t_1}^{t_{2}}A(s)\dd s\right )\right |^{-1}=1-\int_{t_1}^{t_{2}}{\rm tr }(A(s))\dd s+ \mathcal{O}\left (\left (\int_{t_1}^{t_{2}}||(A(s)||\dd s\right )^2\right ). 
        \end{equation*}
 \end{lem}
\begin{proof}
Set $B\coloneqq \int_{t_1}^{t_2}A(s)\dd s$. By linearity of the trace,
${\rm tr }(B)=\int_{t_1}^{t_2}{\rm tr }(A(s))\dd s$, while
$\|B\|\le \int_{t_1}^{t_2}\|A(s)\|\dd s$.
The first-order expansion of the determinant yields
$\det(I+B)=1+{\rm tr }(B)+\mathcal O(\|B\|^2)$.
Since $A\in L^1([0,T];\mathbb R^{d\times d})$, for $(t_2-t_1)$ sufficiently small one has
$|{\rm tr }(B)|+\mathcal O(\|B\|^2)<1$. Hence, applying the expansion
$\frac{1}{|1+\alpha|}=1-\alpha+\mathcal O(\alpha^2)$ with
$\alpha={\rm tr }(B)+\mathcal O(\|B\|^2)$ gives
    \begin{equation*}
|\det(I+B)|^{-1}=1-{\rm tr }(B)+\mathcal O(\|B\|^2),
    \end{equation*}
which yields the desired estimate.
\end{proof} 

We recall the following matrix trace estimate from \cite[Lemma 1]{algebraicmatrix}.
\begin{lem}\label{lem: algebraicmatrix}
    Let $K,S \in \RR^{d\times d}$. Assume that $K$ is symmetric and positive semidefinite, and $S$ is symmetric, then
    \begin{equation*}
        \lambda_{\min}(S){\rm tr}(K) \leq {\rm tr} (KS)\leq \lambda_{\max}(S){\rm tr}(K),
     \end{equation*}
     where  $\lambda_{\min}(S)$ and $ \lambda_{\max}(S)$ denote the smallest and largest eigenvalues of $S$, respectively.
\end{lem}

\subsection{Clarke's Generalized Jacobian }\label{appendix: generalized g}

We briefly recall the definition of Clarke's Generalized Jacobian and some useful properties, and refer to \cite{Warga,Clarke,ClarkeLedyaevStern} for a more comprehensive treatment.
\begin{defin}
 Let $F\colon\RR^d \to \RR^m$ written in terms of component functions as $F(x)=(F_1(x),\cdots,F_m(x))$. We assume that any $F_i\colon\RR^d \to \RR$ is Lipschitz near a given point $x \in \RR^d$. We define  the generalized Jacobian as
  \begin{equation*}
      \partial F(x)\coloneqq \co \left (\left\{ \lim_{n \to \infty} DF(x_n) \, \middle| \, x_n \to x, \, x_n \notin \Omega_F\right \}\right ), 
  \end{equation*}
  where   ``$\co$'' denotes the convex hull and  $\Omega_F\coloneqq \{x \in \RR^d \, | \,  F \text{ is not differentiable at } x\}$.
\end{defin}
The following proposition collects the basic structural properties of Clarke's generalized Jacobian, including convex compactness, linearity, and the relation to the classical derivative.

\begin{prop}\label{prop: generalized Jacobian elementary properties}
Let $F\colon\RR^d\to\RR^m$ be locally Lipschitz near $x$.
\begin{enumerate}[label={\rm(\roman*)},     ref={\rm(\roman*)}]
\item \label{item: Clarke convex set} $\partial F(x)$ is a nonempty convex compact subset of $\RR^{m\times d}$.
\item \label{item: Clarke linear} If $L\colon\RR^d\to\RR^m$ is linear, then
$$
\partial(L-F)(x)=L-\partial F(x).
$$
\item \label{item: Clarke diff} If $F$ is differentiable at $x$, then
$DF(x)\in\partial F(x)$.
\item  \label{item: Clarke cont diff} If $F$ is continuously differentiable at $x$, then
$\partial F(x)=\{ DF(x)\}$.
\end{enumerate}
\end{prop}

We recall the chain rule theorem for  Lipschitz functions (see  \cite[Theorem 4]{imbert2002support}).
\begin{theo}\label{thm: chain rule generalized J}
     Let $F\colon\RR^d \to \RR^m $ be Lipschitz near $x \in \RR^d$  and $G\colon\RR^m \to \RR^k$  be Lipschitz near $F(x) \in \RR^m$, then 
      \begin{equation*}
          \partial (G\circ F)(x)\subseteq \co \{\partial G(F(x))\partial F(x) \}. 
      \end{equation*}
\end{theo}

The following property of the generalized gradient of integral functionals is a consequence of \cite[Theorem 2.7.2]{Clarke}.
We use the Aumann integral for the set-valued map appearing below. More
precisely,
\begin{equation}
\int_{t_1}^{t_2}\partial\big(F(s,\cdot)\big)(x)\dd s
\coloneqq
\left\{
\int_{t_1}^{t_2}C(s)\dd s\,\middle|\,
\begin{array}{l}
C:[t_1,t_2]\to\RR^{d\times d}\text{ is measurable},\\
C(s)\in\partial\big(F(s,\cdot)\big)(x)\text{ for a.e. }s
\end{array}
\right\}.
\label{eq:Aumann_integral}
\end{equation}
Thus, membership in the integral on the right-hand side of the following
inclusion includes the existence of a measurable selection.

\begin{theo}\label{lem: differentiation under the integral sign} 
Let $F\colon[0,T]\times \mathbb{R}^d \to \mathbb{R}^d$ be measurable in the first variable. Assume that there exists  $C_F \in L^1([0,T])$ a.e. positive such that
\begin{equation*}
|F(t,x_1)-F(t,x_2)| \le  C_F(t) |x_1-x_2|, \quad \text{for  a.e. } t \in [0,T], \, \text{for all } x_1,\,x_2 \in \mathbb{R}^d.
\end{equation*}
Then,  for every $t_1,t_2 \in [0,T]$ with $t_2\geq t_1$, we have
\begin{equation*}
    \partial \left ( \int_{t_1}^{t_2} F(s,\cdot)\dd s\right )(x) \subseteq \int_{t_1}^{t_2} \partial \big ( F(s,\cdot) \big )(x)\dd s \qquad  \text{for all } x \in \RR^d.
\end{equation*}
\end{theo}

\section{Appendix}\label{appendix: exact solution}
In what follows, we compute explicitly the unique (see  \cite{BensoussanETal}) solution $(v^*,m^*)$ to the MFG system appearing in \Cref{testexaxt},
where $H(t,x,p)=\frac{1}{2}|p|^2+ \hat \alpha(t)\cdot p$,  
\[
\ell(t,x,a)=\frac{1}{2}|a-\hat\alpha(t)|^2,
\qquad
\hat \alpha(t)=
\alpha_1\mathbb{I}_{[0,\overline t)}(t) +\alpha_2\mathbb{I}_{[\overline t,T]}(t),
\quad \alpha_1,\alpha_2\in\mathbb R^d,\, \overline t\in[0,T],
\]
 $  f(t,x,m)=\frac{1}{2}|x-\int_{\mathbb{R}^d} y m(y)\dd y|^2$, $g(x,m)= 0$, and  initial measure $m^*_0$  given as  the density of a Gaussian  with mean $\mu_0\in\RR^d$ and covariance  $\Sigma_0\in \RR^{d\times d}$ assumed, for simplicity, diagonal.
We first look for functions $\Pi:[0,T] \to\RR^{d\times d}$, $s \colon  [0,T] \to\RR^d$, $c \colon  [0,T] \to\RR$ such that  the $L^1-$ viscosity solution to \eqref{eq:HJB} with $\mu=m^*$  can be written as  
\begin{equation*}
    v^*(t,x)=\frac{1}{2}  \Pi (t)x\cdot x  + s(t)\cdot x + c(t) \qquad \text{for all }(t,x)\in [0,T]\times \RR^d.
\end{equation*}
Setting $\ov Y(t)\coloneqq \int_{\RR^d}xm^*(t,x)\dd x$ and requiring $v^*$ to be a generalized solution to \eqref{eq:HJB} with $\mu=m^*$,  we obtain that $\Pi,s$ and $c$ satisfy
\begin{equation*}
    \setlength{\jot}{8pt}
    \begin{aligned}
        - \dot   \Pi (t) & = -   \Pi ^2(t) + I_d
        && \text{for a.e. } t\in(0,T),\\
        -\dot s(t) & = -  \Pi (t)s(t) - \Pi (t)\hat \alpha(t)-\ov Y(t) 
        && \text{for a.e. } t\in(0,T),\\
        -\dot c(t) & = -\smash{\frac{1}{2}}|s(t)|^2+\smash{\frac{1}{2}} |\ov Y(t)|^2-\hat \alpha(t)\cdot s(t)
        && \text{for a.e. } t\in(0,T),\\
        \Pi(T)&=0,\quad s(T)=0,\quad c(T)=0.
    \end{aligned}
\end{equation*}
Since $D_xv^*(t,x)=  \Pi(t)x+s(t)$  for every $(t,x)\in [0,T]\times \RR^d$, then   $D_pH(t,x,D_xv^*(t,x))=D_xv^*(t,x)+\hat \alpha(t)$.
Arguing as in \cite{CalzolaCarliniSilva2024}, we derive that the explicit formula of $m^*$ is given by
 \begin{equation*}
     m^*(t,x)=\prod_{i=1}^dm_i^*(t,x) \qquad \text{for all } (t,x)\in (0,T)\times \RR^d,
 \end{equation*}
where, for every $t\in [0,T]$ and $i=1,\dots,d $, the density $m_i^*(t,\cdot)$ is a univariate Gaussian with mean $(\ov Y(t))_i$ and covariance $(\Sigma(t))_{i,i}$, where $\Sigma(t)\in \RR^{d\times d}$ .
 The expectation
 $\ov Y(t)$ solves
\begin{equation*}
    \begin{aligned}
&\dot{ \ov Y}(t)=-\Pi(t) \ov Y(t) - s(t)-\hat \alpha(t)  & \text{for a.e. } t\in(0,T),\\
& \ov Y(0)= \mu_0.
\end{aligned}
\end{equation*}
We notice that  $\Pi$ 
satisfies a Riccati equation whose analytical solution is given by
$\Pi(t)= p(t)I_d$, where   $p(t)=\tanh(T-t)$. Once  $\Pi$ is determined,  the solution to the system 
\begin{equation*}
    \begin{aligned}
        -\dot s(t) & = -\Pi(t)s(t)-\Pi(t)\hat\alpha(t)- \ov Y(t)   & \text{for a.e. } t\in(0,T),\\ 
        \dot{\ov  Y}(t)&=-\Pi(t)\ov  Y(t) - s(t)-\hat\alpha(t)  & \text{for a.e. } t\in(0,T), \\
         s(T)&=0, \quad \ov Y(0)= \mu_0
    \end{aligned}
\end{equation*}
is given by
\begin{equation*}
 \ov Y(t)=
\begin{cases}
 \mu_0 -\alpha_1t &\text{for } t<  \ov t, \\ 
 \mu_0 -\alpha_1 \ov t - \alpha_2(t-\ov t)   &\text{for } t\geq  \ov t,
\end{cases}
\end{equation*}
and  $s(t)=-p(t) \ov Y(t)$.
Therefore, we have $|s(t)|^2 
= p^2(t)|\ov  Y(t) |^2$ and $\hat \alpha(t)\cdot s(t)=-p(t)\hat \alpha(t)\cdot \ov Y(t)$, which imply that $c$ solves the system 
\begin{equation*}\label{c_dot_uv} 
\begin{aligned}
\dot c(t)
&=\frac{1}{2} p^2(t) |\ov Y(t)|^2-p(t) \hat\alpha(t)\cdot \ov Y(t) - \frac{1}{2} |\ov Y(t)|^2 \qquad \text{for a.e. } t \in (0,T),\\ 
c(T)&=0.
\end{aligned}
\end{equation*}
Integrating  on $[t,T]$ for $t \in [0,T]$, the system yields
\begin{equation*}
c(t)=-\int_t^T \frac{1}{2} p^2(r) |\ov Y(r)|^2\dd r+\int_t^T p(r) \hat\alpha(r) \cdot \ov Y(r) \dd r+ \int_t^T \frac{1}{2} |\ov Y(r)|^2\dd r,
\end{equation*}
which completes the explicit formula for $v^*$.
The covariance $\Sigma$ can be computed, once $\Pi,s$ and $\ov Y$ have been determined,  following \cite{CalzolaCarliniSilva2024}, as
\begin{equation*}
\Sigma(t)=\left(\frac{\cosh(T-t)}{\cosh(T)}\right)^2 \Sigma_0,
\qquad \text{for all } t\in[0,T].
\end{equation*}
The previous computations establish that 
$v^*$ solves \eqref{eq:HJB} with $\mu=m^*$ for every $(t,x)$ with $t\neq\overline{t}$ and that $m^*$ is a distributional solution to \eqref{eq:CE} thanks to \cite[Lemma 8.1.6]{AmbrosioGigliSavare}. It remains to prove that $v^*$ is a $L^1$-viscosity solution.
We have that $v^*(x,T)=g(x)=0$. Let $(t^*,x^*)\in (0,T)\times \RR^d$, $\phi \in \CC^{\infty}_c((0,T)\times \RR^d)$, $b \in L^1([0,T])$, and  $G\in \CC([0,T]\times \RR^d\times \RR^d)$  as in \cref{def: L1viscosity} for the subsolution property.  Let $R^*>0$ be sufficiently small so that   $ [t^*-R^*,t^*+R^*]\subset(0,T)$ and $(t^*,x^*)$ is a maximum of $(t,x) \mapsto v[\mu](t,x)-\phi(t,x)- \displaystyle\int_t^Tb(s)\dd s$ over $[t^*-R^*,t^*+R^*]\times  \ov B(x^*,R^*)$. We have that 
 \begin{equation}\label{eq: max}
     v^*(t,x^*)-v^*(t^*,x^*)\leq \phi(t,x^*)-\phi(t^*,x^*)-\int_{t^*}^tb(s)\dd s \quad\text{for all } t\in [t^*-R^*,t^*+R^*]
 \end{equation}
 and 
 \begin{equation}\label{eq: gradient}
     D_xv^*(t^*,x^*)=  D_x\phi(t^*,x^*).
 \end{equation}
 Moreover, since  the map $t \mapsto v^*(t,x^*)$  is Lipschitz, \eqref{eq: max} implies that
 \begin{equation}\label{eq: TFC}
     \int_{t^*}^{t}(\partial_tv^*(s,x^*)-\partial_t\phi(s,x^*)+b(s) )\dd s\leq 0 \quad\text{ for all } t \in [t^*,t^*+R^*].
 \end{equation}
Since $v^*$ solves the HJB equation  for every $t \neq \ov t$, we have 
\begin{equation*}
     -\partial_tv^*(t,x^*)+H(t,x^*,D_xv^*(t,x^*))=f(t,x^*,m(t)) \quad \text{for all } t\neq\ov t,
 \end{equation*}
 which, combined with \eqref{eq: TFC}, implies 
 \begin{equation*}
 \begin{aligned} 
   0&=  \int_{t^*}^{t}(-\partial_tv^*(s,x^*)+H(s,x^*,D_xv^*(s,x^*))-f(s,x^*,m(s)))\dd s \\
   &\geq \int_{t^*}^{t}(-\partial_t\phi(s,x^*)+b(s)+H(s,x^*,D_xv^*(s,x^*))-f(s,x^*,m(s)))\dd s \\ 
   & \geq \int_{t^*}^{t}(-\partial_t\phi(s,x^*)+G(s,x^*,D_xv^*(s,x^*)))\dd s.
    \end{aligned}
 \end{equation*}
Dividing by $(t-t^*)$, we have 
   \begin{equation*}
 \begin{aligned} 
   0& \geq \frac{1}{t-t^*} \int_{t^*}^{t}(-\partial_t\phi(s,x^*)+G(s,x^*,D_xv^*(s,x^*)))\dd s,
    \end{aligned}
 \end{equation*}
 and, since the map $t \mapsto D_xv^*(t,x^*)$ is continuous, passing to the limit as  $t \to t^*$ and using \eqref{eq: gradient}, we get $-\partial_t\phi(t^*,x^*)+G(t^*,x^*,D_x\phi(t^*,x^*))\leq 0$. Analogously, we can obtain that $v^*$ is a $L^1$-viscosity supersolution. 
\bibliographystyle{abbrv}   
\bibliography{references}
\end{document}